%% file: ms.tex
\documentclass[12pt]{article}
\pdfoutput=1
\usepackage{amsmath}
\usepackage{graphicx}
\usepackage{enumerate}
\usepackage{natbib}
\usepackage{url} 

\usepackage[inline]{enumitem}
\usepackage{amssymb}   
\usepackage{amsthm}    
\usepackage{mathrsfs}  
\usepackage{dsfont}    
\usepackage{mathtools} 
\usepackage{bm}        
\usepackage[inline]{enumitem}  
\usepackage{tikz}
\usepackage{pgfplots}
\usepackage{wrapfig}
\usepackage{hyperref}
\usepackage{multirow}
\usepackage{xr-hyper}  
\usepackage{caption}

\DeclareCaptionLabelFormat{boldcap}{\bf #1 #2}
\mathtoolsset{showonlyrefs=true} 
\theoremstyle{plain}
\newtheorem{theorem}{Theorem}
\newtheorem{cor}[theorem]{Corollary}

\newtheorem{prop}[theorem]{Proposition}

\newtheorem{lemma}[theorem]{Lemma}
\theoremstyle{remark}
\newtheorem{defn}{Definition}
\newtheorem*{remark}{Remark}
\newtheorem{ex}{Example}

\DeclareMathOperator{\N}{N}
\DeclareMathOperator{\Gammadist}{Gamma}

\DeclareMathOperator*{\argmin}{arg\,min}
\DeclareMathOperator{\Var}{Var}

\DeclareMathOperator{\mise}{MISE}
\DeclareMathOperator{\ise}{ISE}
\DeclareMathOperator{\imise}{(M)ISE}

\makeatletter
\newcommand*{\addFileDependency}[1]{
  \typeout{(#1)}
  \@addtofilelist{#1}
  \IfFileExists{#1}{}{\typeout{No file #1.}}
}
\makeatother

\newcommand*{\myexternaldocument}[1]{%
    \externaldocument{#1}%
    \addFileDependency{#1.tex}%
    \addFileDependency{#1.aux}%
}
\myexternaldocument{supplemental}

\newcommand{\blind}{0}

\newcommand{\showproofs}{1}
\newcommand{\elide}{\begin{center} \(\longleftarrow\) \textbf{Proof temporarily elided} \(\longrightarrow\) \end{center}}

\begin{document}

\def\spacingset#1{\renewcommand{\baselinestretch}%
{#1}\small\normalsize} \spacingset{1}


\if0\blind
{
  \title{\bf Smoothness-Penalized Deconvolution (SPeD) of a Density Estimate}
  \author{David Kent\footnote{Department of Statistics and Data Science, Cornell University} \footnote{This work was supported by the National Science Foundation under Grant AST-1814840. The opinions, findings, and conclusions, or recommendations expressed are those of the authors and do not necessarily reflect the views of the National Science Foundation.}\ \ and David Ruppert\footnotemark[1] \footnotemark[2] \footnote{School of Operations Research and Information Engineering, Cornell University}\hspace{.2cm}}
  \maketitle
} \fi

\if1\blind
{
    \title{\bf Smoothness-Penalized Deconvolution (SPeD) of a Density Estimate}
    \maketitle
} \fi

\bigskip
\begin{abstract}
This paper addresses the deconvolution problem of estimating a square-integrable probability density from observations contaminated with additive measurement errors having a known density. The estimator begins with a density estimate of the contaminated observations and minimizes a reconstruction error penalized by an integrated squared $m$-th derivative. Theory for deconvolution has mainly focused on kernel- or wavelet-based techniques, but other methods including spline-based techniques and this smoothness-penalized estimator have been found to outperform kernel methods in simulation studies. This paper fills in some of these gaps by establishing asymptotic guarantees for the smoothness-penalized approach. Consistency is established in mean integrated squared error, and rates of convergence are derived for Gaussian, Cauchy, and Laplace error densities, attaining some lower bounds already in the literature. The assumptions are weak for most results; the estimator can be used with a broader class of error densities than the deconvoluting kernel. Our application example estimates the density of the mean cytotoxicity of certain bacterial isolates under random sampling; this mean cytotoxicity can only be measured experimentally with additive error, leading to the deconvolution problem. We also describe a method for approximating the solution by a cubic spline, which reduces to a quadratic program.
\end{abstract}

\noindent%
\vfill

\newpage
\spacingset{1.9} 

\section{Introduction} 

A smoothness-penalized density deconvolution estimator was introduced
in \cite{yang_density_2020} which is fast to compute, amenable to
shape constraints, and in simulation studies has substantially
improved finite-sample performance over the common deconvoluting
kernel density estimator of \cite{stefanski_deconvoluting_1990}. A
spline-based Bayesian approach for a related problem in
\cite{staudenmayer_density_2008} outperforms deconvoluting kernels in
simulation studies as well, and \cite{sarkar_bayesian_2014} does yet
better. In spite of these appealing properties, these estimators have
not yet received much attention. This is perhaps due to a lack of
theoretical guarantees; most asymptotic results for deconvolution
estimators focus on kernel- or wavelet-based
(e.g.\ \cite{pensky_adaptive_1999}) methods, while these other methods
have only been addressed in simulations.

In this paper, we address a continuous version of the
smoothness-penalized estimator in \cite{yang_density_2020} and provide
some theoretical guarantees. We prove the consistency of the density
estimates in \(L_2\) and derive upper bounds for the rate of
convergence, which are found to be optimal when compared to lower
bounds already in the literature. We also prove in
Theorem~\ref{theorem:smoothrates}, under stronger smoothness
conditions inspired by typical assumptions in the ill-posed problem
literature, fast rates of convergence which hold for any error
density, whether smooth or super-smooth. We are not aware of similar
results for kernel-based deconvolution estimators. Along the way, we
derive a representation of the estimator which is more convenient for
theoretical work than the variational formulation in
\cite{yang_density_2020} and investigate the finite-sample error for a
few settings.

Suppose a real-valued random variable of interest \(X\) has pdf \(f\),
and we wish to estimate \(f\). However, we instead observe independent
copies of \(Y = X+E\), a surrogate of \(X\) which has been
contaminated with an independent error \(E\). Suppose
further that \(E\) has known pdf \(g\). Under these
conditions, the pdf \(h\) of \(Y\) is given by the \emph{convolution}
of \(g\) and \(f\), i.e.
\begin{equation}
  h(y) = g*f(y) = \int f(t)g(y-t)\,dt.
\end{equation}
The task of estimating the density \(f\) from a sample \(Y_1, \dots,
Y_n\) of independent random variables with pdf \(h = g*f\) is
sometimes called a \emph{deconvolution} problem, since we can think of
the main goal as ``undoing'' the convolution with \(g\).

The method in \cite{yang_density_2020}, to be described shortly,
begins with a density estimate of \(h\), and proceeds to estimate
\(f\) through this density estimate of \(h\). To that end, we will
introduce one more abstraction: we will assume that we have access to
an \(L_2(\mathbb{R})\)-consistent estimator of \(h\), which we will
denote \(h_n\) (Estimators of functions will be indicated by a
subscript \(n\) rather than the customary ``hat'' to avoid clutter
when taking Fourier transforms, which will be denoted by an overset
twiddle). The mean integrated squared error (MISE) of \(h_n\) will be
denoted \(\delta_n^2 = \mathbb{E} \|h_n - h\|^2 = \mathbb{E}\int (h_n
- h)^2\). We will think of \(h_n\) as the ``data'' in this problem and
express the performance of our estimator in terms of
\(\delta_n^2\). Note that since the estimator is consistent, we have
\(\delta_n^2 \to 0\).

This setting occurs whenever a density estimate is required, but the
variable is measured with error; it is therefore a nearly ubiquitous
phenomenon, but typically ignored when the measurement error is
small. For a window into the meaning of ``small'' here, note that
ignoring measurement error \(E\) means to estimate \(g*f\) in place of
\(f\), incurring at a point \(x\) the error
  \(g*f(x) - f(x) = \mathbb{E}[f(x-E) - f(x)]\),
where \(E\) has pdf \(g\). Thus wherever \(f\) has large curvature on
the scale of \(E\), \(g*f\) and \(f\) will not be similar, and in such
cases measurement error should not be ignored. 

Because of the ubiquity of the setting, the application domains are
diverse. In this paper, we estimate a conditional
density that occurs when estimating the cytotoxicity of bacterial
isolates; in \cite{yang_bias-corrected_2021}, the authors use
deconvolution to estimate the density of a conditional expectation
occurring in nested Monte Carlo
simulations. \cite{staudenmayer_density_2008} apply deconvolution to
nutritional data from a clinical trial involving a dietary supplement,
and \cite{stefanski_deconvoluting_1990} treat data on saturated fat
intake.

The treatment of this problem makes great use of the Fourier
transform. Following conventions in \cite{folland_fourier_1992}, for
\(v \in L_1(\mathbb{R}) \cup L_2(\mathbb{R})\) we will write \(\tilde
v(\omega) = \lim_{r \to \infty} \int_{-r}^r e^{-i\omega x}v(x)\,dx =
\int e^{-i\omega x} v(x)\,dx\) for the Fourier transform of \(v\),
with the second equality holding as long as \(v \in L_1(\mathbb{R})\).
Let \(\tilde P_n(\omega) = \frac1n \sum_{j=1}^n e^{-iY_j\omega}\) be
the Fourier transform of the empirical distribution. If \(u(x) =
v*w(x)\), then \(\tilde u(\omega) = \tilde v(\omega) \tilde
w(\omega)\), so that the Fourier transform reduces convolution to
multiplication.

A well-known estimator of \(f\) in this setting is the deconvoluting
kernel (density) estimator (DKE), introduced in
\cite{stefanski_deconvoluting_1990}, which takes advantage of the
reduction of convolution to multiplication. First, we form a kernel
density estimate \(h_n^\lambda(x) = \frac{1}{n\lambda} \sum_{j=1}^n
K((x-Y_j)/\lambda)\) of \(h\), in which case \(\tilde
h_n^\lambda(\omega) = \tilde P_n(\omega) \tilde K(\lambda
\omega)\). Then we divide by \(\tilde g(\omega)\) and inverse
transform:
\begin{equation}
  \label{eqn:dkde}
  f^{\lambda}_n(x) = \frac{1}{2\pi} \int e^{i\omega x} \tilde P_n(\omega) \tilde K(\lambda \omega)/\tilde g(\omega)\,d\omega.
\end{equation}
In \cite{stefanski_deconvoluting_1990}, they find that if \(K_\lambda^*(x) = (2\pi)^{-1} \int e^{i\omega x} \tilde K(\omega)/\tilde g(\omega/\lambda)\,dt\), then \(f_n^\lambda\) has representation
\begin{equation}
  f_n^\lambda(x) = \frac{1}{n\lambda} \sum_{j=1}^n K_\lambda^*( (x-Y_j)/\lambda )
\end{equation} and are able to borrow from results on standard kernel
density estimators in their analysis. To ensure that the Fourier
inversion in Equation~\eqref{eqn:dkde} is well-defined,
\cite{stefanski_deconvoluting_1990} require \(K\) to be chosen to
satisfy \(\sup_\omega |\tilde K(\omega)/\tilde g(\omega/\lambda)| <
\infty\) and \(\int |\tilde K(\omega)/\tilde
g(\omega/\lambda)|\,d\omega < \infty\) for all \(\lambda > 0\),
suggesting band-limited kernels, including \(K(x) = \frac{1}{\pi}
(\sin(x)/x)^2\), which has Fourier transform \(\kappa(\omega) =
\mathds{1}_{[-2,2]}(1-|\omega|/2)\). Note that in particular, \(K(x) =
g(x)\) typically cannot satisfy these conditions. Additionally,
\cite{stefanski_deconvoluting_1990} restrict attention to \(g\) for
which \(|\tilde g(\omega)| > 0\), since the estimator involves
division by \(\tilde g(\omega)\). Under appropriate choice of
\(\lambda_n \to 0\), the estimator is consistent and attains optimal
rates in several settings. In \cite{fan_optimal_1991}, optimal rates
are addressed for \(f\) in a class of functions with \(m\)th
derivative H\"older-continuous. In \cite{zhang_fourier_1990}, optimal
rates are addressed over for \(f\) in a class of functions satisfying
\(\|\omega \tilde f(\omega)\|^2 < M < \infty\).

The density deconvolution technique introduced in
\cite{yang_density_2020} discretizes both the functions and the
convolution operator. The estimate \(h_n\) is approximated on a grid
by a vector \(\mathbf{h}_n\), and the convolution operator by a matrix
\(\mathbf{C}\), so that if \(\mathbf{v}\) is a discrete approximation
of a function \(v\), then \(\mathbf{C}\mathbf{v}\) is a discrete
approximation of \(g*v\). Then a discrete approximation
\(\mathbf{f}_n^\alpha\) of \(f\) is computed by solving the matrix
problem
\begin{equation}
  \label{eqn:yangqp}
  \mathbf{f}^\alpha_n = \argmin_{\mathbf{x}}\;\|\mathbf{C}\mathbf{x} - \mathbf{h}_n\|^2 + \alpha Q(\mathbf{x}),
\end{equation}
where \(Q(\,\cdot\,)\) is a quadratic penalty. (Vectors and matrices
will always be typeset in boldface, and we overload \(\|\cdot\|\) to
denote the vector 2-norm when the argument is a vector.)  For
\(Q(\,\cdot\,)\) the authors suggest, among other choices, the squared
norm of a second-differencing operator applied to \(\mathbf{x}\):
\(Q(\mathbf{x}) = \|\mathbf{D}_2\mathbf{x}\|_2^2\). Heuristically,
this approach yields an estimate \(\mathbf{f}_n^\alpha\) whose
convolution \(\mathbf{C}\mathbf{f}_n^\alpha\) is close to the density
estimate \(\mathbf{h}_n\) (due to the first term), but which is not
too wiggly (due to the second term). They observe that
Equation~\eqref{eqn:yangqp} can be formulated as a quadratic program
and solved efficiently using standard methods, and that linear
constraints can be introduced as well.

In this paper, we analyze the exact, continuous version of the
estimator introduced in \cite{yang_density_2020}, with penalty \(Q(v)
= \|v^{(m)}\|^2\), where \(v^{(m)}\) denotes the \(m\)th derivative of
\(v \in L_2(\mathbb{R})\), i.e.
\begin{equation}
  \label{eqn:firstestimator}
  f_n^\alpha = \argmin_{v}\;\|g*v - h_n\|^2 + \alpha \|v^{(m)}\|^2.
\end{equation}
The argument \(v\) is taken to range over the subset of
\(L_2(\mathbb{R})\) for which the objective function is well-defined,
which we will make specific in Section~\ref{sec:estimator}.  We will
occasionally find it useful to use operator notation, with \(T:
L_2(\mathbb{R}) \to L_2(\mathbb{R})\), \(T: v \mapsto g*v\) and \(L:
\mathcal{D}(L) \to L_2(\mathbb{R})\), \(L: v \mapsto v^{(m)}\), so
that we can alternatively write
\begin{equation}
  \label{eqn:firstestimator_operator}
  f_n^\alpha = \argmin_{v}\;\|Tv - h_n\|^2 + \alpha \|Lv\|^2.
\end{equation} In Section~\ref{sec:practice}, we suggest an
alternative to the discretization approach in
\cite{yang_density_2020}. We instead solve
Equation~\eqref{eqn:firstestimator} out of an approximation space of
piecewise polynomial spline functions. Calling this approximation
\(s_n^\alpha\), we prove in Theorem~\ref{theorem:splineasgood} that
the resulting approximation error \(\mathbb{E}\|s_n^\alpha -
f_n^\alpha\|^2\) can be made to decrease faster than the order of
convergence of \(\mathbb{E}\|f_n^\alpha - f\|^2\) by choosing a
suitably rich approximation space. It follows that
\(\mathbb{E}\|s_n^\alpha - f\|^2\) has the same order of convergence.

One appealing property of this estimator is that the computational
techniques proposed here and in \cite{yang_density_2020} can be quite
fast, with the computational complexity determined primarily by the
dimension of the discretization grid or spline basis. We will see that
computing the spline approximation can be formulated as a quadratic
program, and that many useful linear constraints can be imposed. Among
these, positivity and integrate-to-one constraints are easily imposed,
as are support constraints and some shape
constraints. \cite{yang_density_2020} even suggest a method for
imposing a unimodal constraint by a family of ``unimodal at a point
\(x_0\)'' constraints, each of which can be imposed as a linear
constraint.

Finally, the DKE approach requires \(g\) to have non-vanishing Fourier
transform and therefore cannot be applied to, for example, uniformly
distributed errors. The estimator addressed here has no such
requirement; instead, the Fourier transform of \(g\) must only be
non-vanishing almost-everywhere, which is also a necessary condition
for identifiability in this model.

After an overview of the inherent difficulties of deconvolution in
Section~\ref{sec:illposed} and introducing the estimator in detail in
Section~\ref{sec:estimator}, we prove global
\(L_2(\mathbb{R})\)-consistency, as well as rates of convergence in
Section~\ref{sec:asymptotics}. In Section~\ref{sec:practice}, we
address the practical issue of computing the estimate, investigate its
performance in finite samples, and apply it to a problem on bacterial
cytotoxicity.

\section{Ill-Posedness of the Problem}
\label{sec:illposed}

Deconvolving a density estimate is a typical ``ill-posed'' problem.
We will see that ill-posedness means a naive solution to the
deconvolution problem must fail to be consistent, and any consistent
deconvolution estimator must reflect some aspect of regularization.  A
problem is said to be well-posed if \cite[Chapter
2]{engl_regularization_1996} the following conditions are met:
``\begin{enumerate*}
 \item \label{cond:exists} For all admissible data, a solution exists,
 \item \label{cond:unique} For all admissible data, the solution is unique, and
 \item \label{cond:continuous} The solution depends continuously on the data,
\end{enumerate*}''
and ill-posed otherwise.

For the deconvolution problem, consider the operator
\(T:L_2(\mathbb{R}) \to L_2(\mathbb{R})\) which convolves a function
with \(g\), i.e.\ \(v \mapsto g*v\). Since \(h = g*f\), plugging in \(v
= f\) clearly solves the following operator equation:
\begin{equation}\label{eqn:exactopeq}
Tv = h.
\end{equation}
However, we do not know \(h\). We have an estimate \(h_n\) of \(h\),
and we would like to solve the analogous problem with our estimate
\(h_n\) on the right-hand side, i.e.
\begin{equation}\label{eqn:approxopeq}
  Tv = h_n.
\end{equation}
There is an immediate issue with this approach: there is no \(v\)
solving this equation unless \(h_n \in \mathcal{R}(T)\), i.e.\ \(h_n =
T\psi\) for some \(\psi \in L_2(\mathbb{R})\). If \(h_n\) is
unrestricted, the problem of solving Equation~\eqref{eqn:approxopeq}
violates Condition \ref{cond:exists} of well-posedness. However, we
can overcome this problem by using a generalized inverse of \(T\), so
we will set it aside for the moment.

Instead, we will focus on a more critical deficiency: the solution
operator for Equation~\eqref{eqn:approxopeq} is not continuous in
\(h_n\). This means that a small perturbation of the right-hand side
can lead to arbitrarily large fluctuations in the solution, so that
problem of solving \(Tv = h_n\) is not a good approximation of solving
\(Tv = h\) no matter how well \(h_n\) approximates \(h\). If we
require \(h_n \in \mathcal{R}(T)\) so that the solution operator is
simply \(T^{-1}\), then this discontinuity would entail that for any
\(\varepsilon > 0\) and \(C > 0\), we can have \(\|h_n - h\| <
\varepsilon\), but \(\|T^{-1}h - T^{-1}h_n\| = \|f - T^{-1}h_n\| >
C\). No matter how good we require the estimate \(h_n\) of \(h\) to
be, its exact deconvolution may yet be an arbitrarily bad estimate of
\(f\). Let's prove it formally: the following proposition guarantees
the existence of a function \(u\) so that taking \(h_n = h + u\)
creates the unhappy situation just described.

\begin{prop}\label{prop:discontinuous}
  Assume that the Fourier transform \(\tilde g\) of \(g\) is a.e. non-vanishing, so that \(T: L_2(\mathbb{R}) \to L_2(\mathbb{R})\) is injective (Fact~\ref{fact:injective}). Let \(T^{-1}:\mathcal{R}(T) \to L_2(\mathbb{R})\) be the inverse of \(T\) from its range. Then, for any \(M > 0\), there is some \(u \in \mathcal{R}(T)\) for which \(\label{eqn:discontinuous} \|T^{-1}u\| > M\|u\|.\)
\end{prop}

\if1\showproofs
{
  \input{proofs/prop_discontinuous}
} \fi

\if0\showproofs
{
  \elide
} \fi

Now, even if \(h_n \not \in \mathcal{R}(T)\), a generalized inverse
like the Moore-Penrose inverse \(T^\dag\) may be used in place of
\(T^{-1}\), ensuring that Conditions \ref{cond:exists} and
\ref{cond:unique} are met. But these generalized inverses extend
\(T^{-1}\) from \(\mathcal{R}(T)\), so they too fail to be continuous
by Proposition~\ref{prop:discontinuous}.

We turn to a method of regularization solution. A regularized solution
of Equation~\eqref{eqn:exactopeq} is a family of operators
\(\{R_\alpha\}_{\alpha>0}\) which approximate \(T^{-1}\) or an
extension thereof, and which has the property that for each
\(\alpha\), \(R_\alpha\) is a continuous operator. For our choice of
regularization by smoothness penalty, we will see in
Theorem~\ref{theorem:representing} that each \(R_\alpha\) is a bounded
operator. In Theorem~\ref{theorem:consistency}, we will see that the
regularization does approximate the exact solution to
Equation~\eqref{eqn:exactopeq}, and in
Theorems~\ref{theorem:smoothrates}~\&~\ref{theorem:rates}, we will see
the rates of convergence under a few different conditions.

\section{Assumptions}
We assume throughout that \(f\) and \(g\) are probability densities,
and that \(h_n \in L_2(\mathbb{R})\).  The following is a list of all
further assumptions that recur in the theoretical results; in each
statement we will name the assumptions required from this
list. Assumptions that are used only for a single result are stated in
that result. First, assumptions that will be made on the target
density \(f\):
\begin{enumerate*}[label=(F\arabic*)]
\item\label{assumption:fisl2} \(f \in L_2(\mathbb{R})\);
\item\label{assumption:fkfour} \(\int |\omega^k \tilde f(\omega)|^2\,d\omega < \infty\) for some \(1 \leq k \leq 2m\).
\end{enumerate*}
Now, assumptions that will be made on the error density \(g\):
\begin{enumerate*}[label=(G\arabic*)]
\item\label{assumption:gzeros} \(\tilde g\) vanishes only on a set of Lebesgue measure zero;
\item\label{assumption:gisl2} \(g \in L_2(\mathbb{R})\). 
\item\label{assumption:gtwidintgbl} \(\int |\tilde g(\omega)|\,d\omega < \infty\);
\end{enumerate*}
Note that Assumptions~\ref{assumption:gzeros}-~\ref{assumption:gtwidintgbl}
all hold for Normal, Cauchy, and Laplace errors. Note also that if \(f
\in L_2(\mathbb{R})\), then by Young's convolution inequality, \(h \in
L_2(\mathbb{R})\) as well.

\section{The Estimator}
\label{sec:estimator}

One family of solution operators for Equation~\eqref{eqn:approxopeq}
which extend \(T^{-1}\) are those which map \(h_n\) to a
\emph{least-squares solution}, i.e.\ to a \(v\) minimizing \(\|Tv -
h_n\|\). The familiar Moore-Penrose generalized inverse \(T^\dag\) is
a least-squares extension---it is the operator which maps \(h_n\) to
the least-squares solution \(v\) for which \(v\) has minimal norm
\(\|v\|\). Classical Tikhonov regularization approximates \(T^\dag\)
by the family of operators \(\{S_\alpha\}_{\alpha>0}\) mapping
\begin{equation}
  S_\alpha: h_n \mapsto \argmin_v\; \|Tv - h_n\|^2 + \alpha \|v\|^2.
\end{equation}
Intuitively, the solution \(S_\alpha h_n\) is a function which is
reasonably small in \(L_2(\mathbb{R})\) due to the second term, and
for which \(\|TS_\alpha h_n - h_n\|\) is reasonably small.

Here we address a similar approach, but rather than preferring a
function which is small in \(L_2(\mathbb{R})\), we prefer one which is
smooth, in the sense that its \(m\)th derivative, \(m \geq 1\), has
small norm. Thus we have a family \(\{R_\alpha\}_{\alpha > 0}\)
mapping
\begin{equation}
  R_\alpha: h_n \mapsto \argmin_v\; \|Tv - h_n\|^2 + \alpha \|v^{(m)}\|^2.
\end{equation}
This is a particular case of Tikhonov regularization with differential
operators, which has been treated in an abstract, non-statistical
framework in \cite{locker_regularization_1980}, \citet[Chapter
8]{engl_regularization_1996}, and \cite{nair_trade-off_1997}.

Since our estimator will measure the smoothness of a possible estimate
by the magnitude of its square-integrated \(m\)th derivative, the
estimate must be chosen from among those functions for which this
quantity is finite. To that end, let \(H^m(\mathbb{R}) = \{v \in
L_2(\mathbb{R}) : v^{(k)} \in L_2(\mathbb{R})\text{ for }0 \leq k \leq
m\}\) denote the Sobolev space of square-integrable functions with
square-integrable weak derivatives up to order \(m\). Assume
throughout that \(A = \{\omega | \tilde g(\omega) = 0\}\) has Lebesgue
measure zero.

\begin{defn}
  The \textbf{Tikhonov functional} with data \(u\in L_2(\mathbb{R})\) and penalty parameter \(\alpha>0\) is a function defined by
  \begin{equation}\begin{aligned}
      G(\,\cdot\,;u,\alpha): H^m(\mathbb{R}) &\to \mathbb{R}\\
      v &\mapsto \|g*v - u\|^2 + \alpha \|v^{(m)}\|^2.
  \end{aligned}\end{equation}
\end{defn}

\
\begin{defn}\label{defn:estimator}
  Let \(h_n\) be a density estimate of \(h\) from the sample \(Y_1,
\dots, Y_n\), and let \(\alpha > 0\). The \textbf{Smoothness-penalized
deconvolution of \(h_n\)} or \textbf{Smoothness-penalized
deconvolution estimate (SPeD) of \(f\)} is defined variationally by
  \begin{equation}\begin{aligned}
      \label{eqn:estimator}
      f_n^\alpha &= \argmin_{v \in H^m(\mathbb{R})}\; G(v;h_n,\alpha) \\
      &= \argmin_{v \in H^m(\mathbb{R})}\;\|g*v - h_n\|^2 + \alpha \|v^{(m)}\|^2.
  \end{aligned}\end{equation}
\end{defn}

\begin{remark}
  For a given \(h_n \in L_2(\mathbb{R})\) and \(\alpha > 0\), the
estimator \(f_n^\alpha\) in Definition~\ref{defn:estimator} is
uniquely defined\cite[Theorem
3.5]{locker_regularization_1980}. Moreover, \(f_n^\alpha \in
H^{2m}(\mathbb{R})\).
\end{remark}

\subsection{Representations of the estimator}
Since the variational characterization of \(f_n^\alpha\) does not lend
itself to easy analysis, in Theorem \ref{theorem:representing} we
present an explicit representation for \(f^\alpha_n\), both in terms
of \(h_n\) and the Fourier transform of \(h_n\); if a kernel density
estimator is used for \(h_n\), we will see that \(f^\alpha_n\) can be
computed as a kernel estimate as well, though this is not the approach
we take in the sequel. The Fourier representation will make clear the
manner in which the Tikhonov regularization approximates the ill-posed
exact deconvolution problem.

\begin{figure}[h]
\center\input{tikz_plots/approxinv}\caption{For \(\tilde g\) corresponding to \(\N(0,1)\). Thick line is \(1/\tilde g\), while dashed lines are, from lower to upper, the multiplier \(\tilde \varphi_\alpha\) in Theorem \ref{theorem:representing}\ref{theorem:representing:mult} for \(\alpha = 1, 10^{-2}, 10^{-4}\).}
\label{fig:approxinv}
\end{figure}
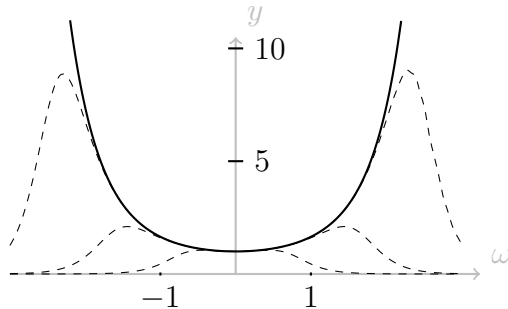

\begin{theorem}{(Representing the solution)}
\label{theorem:representing}
Let
\begin{equation}\label{eqn:varphi}
  \tilde \varphi_\alpha(\omega) = \frac{\overline{\tilde g(\omega)}}{|\tilde g(\omega)|^2 + \alpha \omega^{2m}}\qquad\text{ and }\qquad\varphi_\alpha(x) = \lim_{r\to\infty}\frac{1}{2\pi} \int_{-r}^r e^{i\omega x} \tilde \varphi_\alpha(\omega)\,d\omega.
\end{equation}

Then 
\begin{enumerate}[label=(\roman*)]
\item\label{theorem:representing:mult} \(\tilde f^\alpha_n(\omega) = \tilde \varphi_\alpha(\omega) \tilde h_n(\omega)\),
\item\label{theorem:representing:conv} \(f^\alpha_n(x) = \varphi_\alpha * h_n(x)\), and
\item\label{theorem:representing:kernel} if \(h_n\) is a kernel density estimate with bandwidth \(\nu\), then there is another kernel \(K_{\alpha,\nu}\) for which
\[f^\alpha_n(x) = \frac{1}{n} \sum_{i=1}^n K_{\alpha,\nu}(x - Y_i).\]
\end{enumerate}
Furthermore,
\begin{enumerate}[label=(\roman*)]\setcounter{enumi}{3}
\item\label{theorem:representing:bound} \(\sup_\omega |\tilde \varphi_\alpha(\omega)| \leq C\alpha^{-\frac12}\) for all \(\alpha < M\), 
\item\label{theorem:representing:unitint} If \(\varphi_\alpha \in L_1(\mathbb{R})\), then \(\int \varphi_\alpha(x)\,dx = 1\), so that \(\int f_n^\alpha(x)\,dx = 1\).
\item\label{theorem:representing:apriori} \(\|D^k f_n^\alpha\| \leq C\alpha^{-1}\) for all \(\alpha < M\) and \(0 \leq k < 2m\) with \(C\) depending only on \(g\), \(m\), and \(M\). Under Assumption~\ref{assumption:gisl2}, it holds for \(k = 2m\) as well.
\item\label{theorem:representing:apriorisup} \(\sup_x |D^k f_n^\alpha(x)| \leq C\alpha^{-1}\) for all \(\alpha < M\) and \(0 \leq k < 2m - 1\) with \(C\) depending only on \(g\), \(m\), and \(M\). Under Assumption~\ref{assumption:gtwidintgbl}, it holds for all \(0 \leq k \leq 2m\).
\end{enumerate}
\end{theorem}

\begin{remark}
  Before moving on to the proof, it is worth making a few observations.
  \begin{itemize}
  \item Except where mentioned in \ref{theorem:representing:apriori} and \ref{theorem:representing:apriorisup}, Theorem~\ref{theorem:representing} does not need any assumptions on \(g\) or \(h_n\) beyond the fact that they are probability densities.
  \item If \(g\) is an even function, then \(\tilde \varphi_\alpha(\omega)\) is even and purely real.
  \item Theorem~\ref{theorem:representing}\ref{theorem:representing:bound} holds for any \(g\),
    but can be made sharper with information about a particular choice of
    \(g\). See, for example, the proof of
    Theorem~\ref{theorem:rates}\ref{theorem:rates:laplace}.
  \item Theorem \ref{theorem:representing}\ref{theorem:representing:bound} equivalently says that,
    with \(R_\alpha\) denoting the operator which maps \(h_n \mapsto
    f_n^\alpha\), the operator norm has a bound \(\|R_\alpha\| \leq
    C\alpha^{-\frac12}\), showing that the solution operator for each
    \(\alpha\) is bounded.
  \item If the density estimate is a kernel density estimate
    \(h_n^\lambda\) with kernel appropriate for the DKE
    (e.g.\ bandlimited), then for fixed data and bandwidth \(\lambda\),
    if we let \(\alpha \to 0\), we have that \(f_n^\alpha \to
    f_n^\lambda\), i.e.\ we obtain the DKE defined in
    Equation~\eqref{eqn:dkde}.
  \end{itemize}
\end{remark}

\if1\showproofs
{
  \input{proofs/theorem_representing}
} \fi

\if0\showproofs
{
  \elide
} \fi

Regularized solutions are well-behaved approximations to a poorly
behaved exact problem, and the Fourier view of our estimator gives a
nice picture of the manner of approximation. Suppose briefly that
\(\tilde g\) is even and non-vanishing. Taking the Fourier transform
of Equation~\eqref{eqn:approxopeq} reduces the convolution to
multiplication, giving \(\tilde g \tilde v = \tilde h_n\), so that we
may write \(\tilde v = \tilde h_n / \tilde g\). Thus in Fourier space,
exact deconvolution of \(h_n\) corresponds to multiplying \(\tilde
h_n\) by \(1/\tilde g\). In
Theorem~\ref{theorem:representing}\ref{theorem:representing:mult}, we
see that in Fourier space, our regularized solution \(\tilde
f_n^\alpha\) corresponds to multiplying \(\tilde h_n\) by this
\(\tilde \varphi_\alpha\) function. Inspection of \(\tilde
\varphi_\alpha(\omega)\) shows that when \(|\omega|\) is small,
\(\tilde \varphi_\alpha(\omega) \approx 1/\tilde g(\omega)\), but that
when \(|\omega|\) is large, the \(\alpha \omega^{2m}\) term dominates
the expression and \(\varphi_\alpha(\omega) \approx 0\), since
\(\tilde g(\omega)\) is bounded. Thus multiplying by \(\tilde
\varphi_\alpha\) performs similarly to multiplying by \(1/\tilde g\)
at low frequencies, but \(\tilde \varphi_\alpha\) prevents
high-frequency features of \(h_n\) from transferring to
\(f_n^\alpha\). This is pictured in Figure~\ref{fig:approxinv} for
Gaussian \(g\) and a variety~of~\(\alpha\).


\subsection{Decomposing the error}

To analyze the error \(f_n^\alpha - f\) , it is useful to introduce a
non-random function \(f^\alpha\) for which \(f^\alpha - f\) represents
the systematic error induced by solving the \(\alpha\)-regularized
problem in place of the exact problem.

\begin{defn}
  The \textbf{\(\alpha\)-smoothed \(f\)}, denoted \(f^\alpha\), is given by
    \(f^\alpha = \argmin_{v \in H^m(\mathbb{R})}\;G(v;h,\alpha).\)
\end{defn}

\begin{remark}
  The \(\alpha\)-smoothed \(f\) is the smoothness-penalized
deconvolution of the exact data~\(h\). In Supplement
Proposition~\ref{prop:alphasmoothed}, it is shown to have
representations \(f^\alpha = \varphi_\alpha*h\) and \(\tilde f^\alpha
= \tilde \varphi_\alpha \tilde h\), and approximates \(f\) in the
sense that \(\|f^\alpha - f\| \to 0\) as \(\alpha \to 0\).
\end{remark}

As the next lemma shows, an appealing property of \(f^\alpha\) is
that, for fixed \(\alpha\), \(\mathbb{E} \|f_n^\alpha - f^\alpha\|^2\)
becomes small when \(\delta_n^2 = \mathbb{E}\|h_n-h\|^2\) gets
smaller, in contrast to the issue with exact deconvolution outlined in
Proposition~\ref{prop:discontinuous}.
\begin{lemma}\label{lem:randompart}
  Assume~\ref{assumption:fisl2}. There is a \(C\) depending only on \(g\), such that for each \(\alpha>0\), we have \(\mathbb{E} \|f_n^\alpha - f^\alpha\|^2 \leq C\delta_n^2/\alpha\)
\end{lemma}

\if1\showproofs
{
  \input{proofs/lem_randompart}
} \fi

\if0\showproofs
{
  \elide
} \fi

\begin{cor}
\label{cor:upper}
Assume~\ref{assumption:fisl2}. For sufficiently small \(\alpha\), we have the upper bound
\[
\mathbb{E} \|f_n^\alpha - f\|^2 \leq C\delta_n^2/\alpha + 2\|f^\alpha - f\|^2.
\]
\end{cor}

\if1\showproofs
{
  \input{proofs/cor_upper}
} \fi

\if0\showproofs
{
  \elide
} \fi

The rate at which \(\|f^\alpha - f\| \to 0\) with \(\alpha\) depends
intimately on the particular form of \(g\). In
Lemma~\ref{lem:systematic_bounds}, we present upper bounds for
\(\|f^\alpha - f\|^2\) in terms of \(\alpha\).

\begin{lemma}\label{lem:systematic_bounds}
  Assume~\ref{assumption:fisl2},~\ref{assumption:fkfour}. Then,
  with \(W(\,\cdot\,)\) denoting the principal branch of the Lambert W function,
  \begin{enumerate}[label=(\roman*)]
  \item\label{lem:systematic_bounds:normal} (Normal errors) If \(g(x) = (2\pi)^{-1} e^{-x^2/2}\), 
    then
    \[
      \|f^\alpha - f\|^2 \leq \frac{C}{m^kW(m^{-1}\alpha^{-\frac{1}{m}})^k} \sim
      \frac{C}{m^k\log(m^{-1}\alpha^{-\frac{1}{m}})^k}
    \]
  \item\label{lem:systematic_bounds:cauchy} (Cauchy errors) If \(g(x) = \frac{1}{\pi(1+x^2)}\), then
  \[
    \|f^\alpha - f\|^2 \leq \frac{C}{m^{2k}W(m^{-1}\alpha^{-\frac{1}{2m}})^{2k}} \sim
    \frac{C}{m^{2k}\log(m^{-1}\alpha^{-\frac{1}{2m}})^{2k}}
  \]
  \item\label{lem:systematic_bounds:laplace} (Laplace errors) If \(g(x) = \frac{1}{2}e^{-|x|}\), then
  \[
    \|f^\alpha - f\|^2 \leq C\left ( \frac{1}{4\alpha} \right )^{-\frac{k}{m+1}},
  \]
  \end{enumerate}
\end{lemma}
The asymptotic equivalences in the first two parts follow from Fact~\ref*{fact:lambertW}\ref{fact:lambertW:asympt}.

\if1\showproofs
{
  Proof deferred to Supplemental Proof~\ref{defproof:lem:systematic_bounds}.
} \fi

\if0\showproofs
{
  \elide
} \fi

\section{Asymptotics}
\label{sec:asymptotics}
\subsection{Consistency and Rates of Convergence}

If Assumption~\ref{assumption:gzeros} holds, then \(\|f^\alpha - f\|
\to 0\) as \(\alpha \to 0\), and the upper bound in
Corollary~\ref{cor:upper} provides a sufficient condition for
\(L_2(\mathbb{R})\)-consistency of \(f_n^\alpha\):
\begin{theorem}{(\(L_2(\mathbb{R})\) consistency)}\label{theorem:consistency}
Assume~\ref{assumption:fisl2},~\ref{assumption:gzeros}. Assume that \(\delta_n^2 \to 0\), and \(\alpha_n\) is chosen so that
\(\delta_n^2/\alpha_n \to 0\) and \(\alpha_n \to 0\). Then
\[
\lim_{n \to \infty} \mathbb{E} \| f^{\alpha_n}_n - f\|^2 = 0.
\]
\end{theorem}

\if1\showproofs
{
  \input{proofs/theorem_consistency}
} \fi

\if0\showproofs
{
  \elide
} \fi

In deriving rates of convergence for ill-posed problems, it is
typically assumed that the solution \(f\) is drawn from a ``source
set,'' assuming some \emph{a priori} degree of smoothness
\cite[Section 3.2]{engl_regularization_1996}. In
\cite{nair_trade-off_1997} Theorem 5.1 and \cite{engl_optimal_1985}
Theorem 3.5, an abstract version of this Tikhonov problem is analyzed,
and they find fast \(\delta^{4/3}\) rates of convergence (in a
stronger norm) compared to the often logarithmic rates in the
statistical literature. The price is that strong assumptions are made
on the target density. In \cite{engl_optimal_1985}, it is assumed that
\(f \in \mathcal{D}(L^*L)\) and \(L^*Lf \in \mathcal{R}(T^*T)\). With
\(T\) the operator that convolves a function with \(g\) and \(L\) the
\(m\)th-derivative operator, this assumption requires that \(f \in
H^{2m}(\mathbb{R})\), and \(f^{(2m)} = g \star g * \psi\) for some
\(\psi \in L_2(\mathbb{R})\), which we express in terms of the Fourier
transforms in the theorem. Nothing is required of \(g\), since their
result holds for any bounded operator \(T\), and convolution with a
probability measure is bounded on \(L_2(\mathbb{R})\).

Below is an analogue of those abstract results, in an
explicitly statistical framework, and with a novel proof. The proof in
the present framework turns out to be quite simple.

\begin{theorem}{(Rates when \(f\) is very smooth)}
\label{theorem:smoothrates}
Suppose \(\int |\omega^{2m}\tilde f(\omega)|^2\,d\omega < \infty\) and
\(|\tilde f(\omega)| = |\tilde g(\omega)|^2 |\omega^{-2m}\tilde
\psi(\omega)|\) for some \(\psi \in L_2(\mathbb{R})\) (Note that this
condition implies Assumption~\ref{assumption:fisl2}). Then for
sufficiently small
\(\alpha_n\),
\[
\mathbb{E}\|f^{\alpha_n}_n - f\|^2 \leq C_1 \delta_n^2/\alpha_n + C_2\alpha_n^2,
\]
and if \(\alpha_n = C_3 \delta_n^{\frac23}\), then
\[
\mathbb{E} \|f^{\alpha_n}_n - f\|^2 = O(\delta_n^\frac43).
\]
\end{theorem}

\if1\showproofs
{
  \input{proofs/theorem_smoothrates}
} \fi

\if0\showproofs
{
  \elide
} \fi

\begin{remark}
Note that Theorem~\ref{theorem:smoothrates} does not require Assumption~\ref{assumption:gzeros}; identifiability issues are sidestepped by the second assumption on \(\tilde f\), which guarantees that \(\tilde f\) is zero whenever \(\tilde g\) is zero.  
\end{remark}
\begin{remark}
  If \(Z\) has pdf \(\eta(z) \in H^{2m}(\mathbb{R})\), with \(\omega^{2m}\tilde \eta(\omega) = \tilde \psi(\omega)\), and \(E_1\) and \(E_2\) are independent with pdf \(g\), then the hypothesis of Theorem~\ref{theorem:smoothrates} is satisfied for the pdf \(f\) of \(X = Z + (E_1 - E_2)\).
\end{remark}

\begin{cor}\label{cor:smoothrates}
Assume the conditions of Theorem \ref{theorem:smoothrates}, and assume that \(\alpha_n = C\delta_n^\frac23\).\\
If \(h_n\) is a kernel density estimate with optimal choice of bandwidth, then
\begin{equation}
\mathbb{E} \|f^{\alpha_n}_n - f\|^2 = O(n^{-8/15}).
\end{equation}

\noindent If \(h_n\) is a histogram with optimal choice of bin widths, then
\begin{equation}
\mathbb{E} \|f^{\alpha_n}_n - f\|^2 = O(n^{-4/9}).
\end{equation}
\end{cor}

\if1\showproofs
{
  \input{proofs/cor_smoothrates}
} \fi

\if0\showproofs
{
  \elide
} \fi

The rates in Theorem~\ref{theorem:smoothrates} are appealing, but are found under conditions different than those typically assumed in the literature. Now, we will assume a particular form for \(g\)---either Gaussian, Cauchy, or Laplace---and leverage the approximation bounds for the \(\alpha\)-smoothed \(f\) from Lemma~\ref{lem:systematic_bounds} to derive rates of convergence under a weaker smoothness assumption on \(f\), namely Assumption~\ref{assumption:fkfour} that \(\int |\omega^k \tilde f(\omega)|^2\,d\omega < \infty\). This is a slight weakening of the assumption in \cite{zhang_fourier_1990}.

\begin{theorem}\label{theorem:rates}
  Assume~\ref{assumption:fkfour}. Then,
  \begin{enumerate}[label=(\roman*)]
  \item\label{theorem:rates:normal} (Normal errors) If \(g(x) = (2\pi)^{-1} e^{-x^2/2}\), then for \(\alpha\) small enough,
    \[
      \mathbb{E}\|f_n^\alpha - f\|^2 \leq C_1\delta_n^2/\alpha + \frac{C_2}{m^k\log(m^{-1}\alpha^{-\frac{1}{m}})^k}
    \]
    and if \(\alpha_n = \delta_n^2 W(\delta_n^{-\frac{2}{k}})^{k}\) then
    \(
      \mathbb{E}\|f_n^\alpha - f\|^2 = O([\log \delta_n^{-1}]^{-k}).
    \)
  \item\label{theorem:rates:cauchy} (Cauchy errors) If \(g(x) = \frac{1}{\pi(1+x^2)}\), then for \(\alpha\) small enough,
  \[
    \mathbb{E}\|f_n^\alpha - f\|^2 \leq C_1\delta_n^2/\alpha + \frac{C}{m^{2k}\log(m^{-1}\alpha^{-\frac{1}{2m}})^{2k}}
  \]
    and if \(\alpha_n = \delta_n^2 W(\delta^{-\frac{1}{k}})^{2k}\), then
  \(
    \mathbb{E}\|f_n^\alpha - f\|^2 = O([\log \delta_n^{-1}]^{-2k}).
  \)
\item\label{theorem:rates:laplace} (Laplace errors) If \(g(x) = \frac{1}{2}e^{-|x|}\), then for \(\alpha\) small enough,
  \[
    \mathbb{E}\|f_n^\alpha - f\|^2 \leq C_1 \delta_n^2 \alpha^{-\frac{2}{m+2}} + C_2 \alpha^{\frac{k}{m+2}},
  \]
  and if \(\alpha_n = \delta_n^{\frac{2(m+2)}{k+2}}\), then
  \(
    \mathbb{E}\|f_n^\alpha - f\|^2 = O(\delta_n^{\frac{2k}{k+2}}).
  \)
  \end{enumerate}
\end{theorem}

\if1\showproofs
{
  Proof deferred to Supplemental Proof~\ref{defproof:theorem:rates}.
} \fi

\if0\showproofs
{
  \elide
} \fi

\begin{cor}\label{cor:ratesinpractice}
  Let \(k = 1\). Then if \(h_n\) is a KDE or histogram estimate with
  optimal bandwidth or bin choice, we have, assuming the conditions of
  Theorem~\ref{theorem:rates} hold,
  \begin{enumerate}[label=(\roman*)]
  \item (Normal errors) \(\mathbb{E}\|f_n^\alpha - f\|^2 = O([\log n]^{-1})\) for KDE and histogram.
  \item (Cauchy errors) \(\mathbb{E}\|f_n^\alpha - f\|^2 = O([\log n]^{-2})\) for KDE and histogram.
  \item (Laplace errors) \(\mathbb{E}\|f_n^\alpha - f\|^2 = O(n^{-\frac{4}{15}})\) for the KDE and \(\mathbb{E}\|f_n^\alpha - f\|^2 = O(n^{-\frac{2}{9}})\) for the histogram.
  \end{enumerate}
\end{cor}

For normal and Cauchy errors, Corollary~\ref{cor:ratesinpractice}
shows that the smoothness-penalized deconvolution estimate attains the
optimal rates derived in \cite{zhang_fourier_1990}. However, for
Laplace errors, the upper bound here is slower than the rate
\(n^{-2/7}\) attained by the deconvoluting kernel density estimator in
\cite{zhang_fourier_1990}. However, the SPeD estimator \emph{can}
attain the \(n^{-2/7}\) rate for a certain choice of estimator
\(h_n\). Recall (cf. Equation~(\ref{eqn:dkde})) that the DKE can be
thought of as involving a kernel estimate of \(h\) using a kernel
\(K(\cdot)\) which has quickly decaying Fourier transform; in
\cite{zhang_fourier_1990} the kernel is required to be
band-limited. If we use as our \(h_n\) a kernel estimator satisfying
the conditions in \cite{zhang_fourier_1990}, then the SPeD estimator
attains the \(n^{-2/7}\) rate.

\begin{prop}\label{prop:particularlaplace}
  Assume \(g\) is Laplace, and suppose \(\int |\omega \tilde
  f(\omega)|^2\,d\omega = C < \infty\). Let \(k(\cdot)\) be a pdf
  satisfying \(k(x) = k(-x)\), \(\int x^2k(x)\,dx < \infty\), \(\int
  |xk'(x)|\,dx < \infty\), and \(\tilde k(\omega) = 0\) for \(\omega
  \not\in [-1,1]\). Suppose that \(h_n\) is a kernel density estimate
  with kernel \(k\), i.e.\ \(h_n(y) = \frac1{n\lambda}\sum_{j=1}^n k\left
    (\frac{y-Y_j}{\lambda}\right)\). Suppose \(\lambda_n =
  c_0n^{-\frac17}\), and \(\alpha_n = O(n^{-\frac{-2(m+2)}{7}})\). Then
  \[\mathbb{E} \|f_n^\alpha - f\|^2 = O(n^{-\frac27}).\]

\end{prop}
\if1\showproofs
{
  Proof deferred to Supplemental Proof~\ref{defproof:prop:particularlaplace}.
} \fi

\if0\showproofs
{
  \elide
} \fi

The following examples show that there is a kind of critical variance
or width imposed by the conditions of
Theorem~\ref{theorem:smoothrates}, at least for a subclass of
densities: if \(E \sim N(0,\sigma^2)\), then a normal target density
\(f\) with variance \(2\sigma^2 + \varepsilon\) satisfies the
conditions of Theorem~\ref{theorem:smoothrates}, but a normal target
density with variance \(2\sigma^2 - \varepsilon\) does not. In
contrast, notice that if \(f(\cdot)\) satisfies the conditions of
Theorem~\ref{theorem:rates}, then a re-scaling \(f_\sigma(\cdot) =
\frac1\sigma f(\frac{\cdot}\sigma)\) satisfies them as well (possibly
with a different constant for the rate).
\begin{ex}\label{ex:both}
Suppose \(E \sim \N(0,\sigma^2)\), \(X \sim \N(0,2\sigma^2 +
\varepsilon)\), with \(\varepsilon > 0\). Then the pdf \(f\) of \(X\)
satisfies the conditions of both Theorem~\ref{theorem:rates} and
Theorem~\ref{theorem:smoothrates}. For the former, it suffices to note
that \(f \in H^{k}(\mathbb{R})\) for any \(k \geq 0\). For the latter,
letting \(\nu(x) = \frac1{\sqrt{2\pi\varepsilon}}
e^{-x^2/2\varepsilon}\), we can take \(\psi(x) = (-1)^m\nu^{(2m)}(x)\).
\end{ex}

\begin{ex}\label{ex:onlyone} Now take \(E \sim \N(0,\sigma^2)\), but
\(X \sim \N(0,2\sigma^2 - \varepsilon)\), with \(0 < \varepsilon <
2\sigma^2\). Then the pdf \(f\) of \(X\) satisfies the conditions of
Theorem~\ref{theorem:rates}, but not
Theorem~\ref{theorem:smoothrates}. The former holds for the same
reason as before. To see why the conditions for
Theorem~\ref{theorem:smoothrates} cannot hold, suppose that there was
a \(\psi\) s.t. \(L^*Lf = T^*T\psi\). Then we would have \((-1)^m
f^{(2m)} = g\star g*\psi\), and taking Fourier transforms yields
\((-1)^m(i \omega)^{2m} e^{-(2\sigma^2 - \varepsilon)\omega^2/2} =
e^{-\sigma^2 \omega^2} \tilde \psi(\omega)\), so that \(\tilde
\psi(\omega) = (-1)^m(i \omega)^{2m} e^{\varepsilon \omega^2/2}\). But
then \(|\tilde \psi(\omega)|^2 \to \infty\) as \(\omega \to \infty\),
so \(\psi \not\in L_2(\mathbb{R})\).
\end{ex}

\subsection{Constrained Solution}
\label{sec:constrained}

We may wish to incorporate \emph{a priori} knowledge about \(f\) into
our estimate. Suppose we know that \(f \in \mathcal{B}\), a closed,
convex set. One easy-to-manage approach is to first solve the
unconstrained problem and find an estimate \(f_n^\alpha\) not
necessarily belonging to \(\mathcal{B}\), and then somehow project
this unconstrained estimate onto \(\mathcal{B}\). Define the
projection operator \(P_\mathcal{B}\) onto a closed, convex set \(\mathcal{B} \subset
L_2(\mathbb{R})\) by
\begin{equation}\label{eqn:projection}
P_\mathcal{B}u = \argmin_{v \in \mathcal{B}}\;\|u - v\|.
\end{equation}
In words, \(P_\mathcal{B}\) maps \(u\) to the \(L_2(\mathbb{R})\)-nearest
element of \(\mathcal{B}\). The projection operator onto a closed convex set is
non-expansive (\cite{engl_regularization_1996}, Section 5.4), meaning
that for all \(u,v \in L_2(\mathbb{R})\), \(\|P_{\mathcal{B}}u -
P_{\mathcal{B}}v\| \leq \|u - v\|\). An immediate consequence is that
if \(f \in \mathcal{B}\), then projecting \(f_n^\alpha\) to
\(\mathcal{B}\) has error at least as small as
\(f_n^\alpha\). Remembering that \(P_\mathcal{B}f = f\), we have
\begin{equation}
  \|P_\mathcal{B}f_n^\alpha - f\| = \|P_\mathcal{B}f_n^\alpha - P_\mathcal{B}f\| \leq \|f_n^\alpha - f\|
\end{equation}

Now, we know \emph{a priori} that \(f\) is a probability density
function, so we ought to ensure that our estimate is a probability
density function as well. Consider the set \(\mathcal{C} = \{v \in
L_2(\mathbb{R}): \int v(t)\,dt = 1, \, v(t) \geq 0\, \forall t \in
\mathbb{R}\}\); this is the set of square-integrable probability
density functions, and now we can express this requirement as
\(f_n^\alpha \in \mathcal{C}\).

Unfortunately, while \(\mathcal{C}\) is convex, it is not closed. To
see this, note that the zero function is a limit point of
\(\mathcal{C}\): let \(\psi_n = \frac1n \mathds{1}_{[0,n]}\), and note
that \(\|\psi_n - 0\| = \|\psi_n\| = n^{-\frac12} \to 0\) as \(n \to
\infty\). Indeed, any non-negative function \(v\) with \(\int v < 1\)
is a limit point of \(\mathcal{C}\). Thus the minimum in
Equation~\eqref{eqn:projection} may not be attained, and the
projection operator \(P_\mathcal{C}\) is not well-defined. Instead,
we can work with approximations to \(\mathcal{C}\). Let
\(
\mathcal{C}_a = \{v \in \mathcal{C},\,v(t) = 0\,\,\forall t \not\in[-a,a]\}
\)
be the subset of \(\mathcal{C}\) of functions with support contained in
\([-a,a]\). 
\begin{lemma}\label{lem:Caclosed}
For fixed \(a\), the set \(\mathcal{C}_a\) is closed and convex.
\end{lemma}

\if1\showproofs
{
  Proof deferred to Supplemental Proof~\ref{defproof:lem:Caclosed}.
} \fi

\if0\showproofs
{
  \elide
} \fi

Let the unconstrained estimator \(f_n^\alpha\) projected to
\(\mathcal{C}_a\) be denoted \(\mathring f_n^\alpha =
P_{\mathcal{C}_a} f_n^\alpha\). The non-expansiveness of the
projection suggests that \(\mathring f_n^\alpha\) may inherit the
asymptotics of \(f_n^\alpha\). If \(f \in C_a\) for some \(a\) and \(a
\to \infty\), then this is immediate from the earlier argument.  If
\(f \not\in \mathcal{C}_a\) for all \(a\), we need to do a little more
work, and for that we will need to know the size of
\(\|P_{\mathcal{C}_a}f - f\|\) in terms of \(a\).

\begin{lemma}\label{lem:projf}
Assume~\ref{assumption:fisl2}, \(\mathbb{E}[|X|^\beta] < \infty\) and that \(f(t) = o(1)\) as
\(|t| \to \infty\).  Then for large enough \(a\),
\(\|P_{\mathcal{C}_a}f - f\| \leq
2\mathbb{E}[|X|^\beta]a^{-\beta}\). If also \(\mathbb{E}[e^{\beta
|X|}] < \infty\), then \(\|P_{\mathcal{C}_a}f - f\| \leq
2\mathbb{E}[e^{\beta |X|}]e^{-\beta a}\).
\end{lemma}

\if1\showproofs
{
  Proof deferred to Supplemental Proof~\ref{defproof:lem:projf}.
} \fi

\if0\showproofs
{
  \elide
} \fi

With this in hand, we can say that our constrained estimator will be
as good (in an asymptotic sense) as the unconstrained estimator, as
long as we let \(a\) grow fast enough that the first term dominates:
\begin{lemma}\label{lem:asgood}
Assume \(\mathbb{E}[|X|^\beta] < \infty\) and that \(f(t) = o(1)\) as \(|t| \to \infty\). Then
\begin{equation}
\|\mathring f_n^\alpha - f\| \leq \|f_n^\alpha - f\| + Ca^{-\beta}.
\end{equation}
If \(\mathbb{E}[e^{\beta |X|}] < \infty\), then
\(\|\mathring f_n^\alpha - f\| \leq \|f_n^\alpha - f\| + Ce^{-\beta a}.\)
\end{lemma}

\if1\showproofs
{
  \input{proofs/lem_asgood}
} \fi

\if0\showproofs
{
  \elide
} \fi

\section{The Estimator in Practice}
\label{sec:practice}

In this section we deal with using the estimator in practice, and
compare its performance to the deconvoluting kernel density estimator
in finite samples. 
\subsection{Computing the Estimate}
\label{sec:computing}


The forms of \(f_n^\alpha\) in
Theorem~\ref{theorem:representing}\ref{theorem:representing:mult}-\ref{theorem:representing:kernel} are useful, but not the
most practical for work on the computer; we need a convenient way to
project our estimate to the set of pdfs as described in
Section~\ref{sec:constrained}, and to impose other shape constraints
as desired. Instead, we compute the estimate in
Equation~\eqref{eqn:estimator} out of an approximation space \(X_n
\subset H^m(\mathbb{R})\) of splines of degree \(r > m\). We will find
that this turns out to be a quadratic program, so that linear
constraints are easily imposed.

Before discussing the details of the computations, we present a
Theorem showing that this is a legitimate approximation to make. If we
denote the spline approximation by \(s_n^\alpha\),
Theorem~\ref{theorem:splineasgood} says that if the parameters of the
spline space are selected appropriately, then \(s_n^\alpha -
f_n^\alpha\) is of a smaller order than the rate of convergence we
found in Theorem~\ref{theorem:rates}. This means that asymptotically,
\(s_n^\alpha\) and \(f_n^\alpha\) are the same estimator. As a
consequence, the spline approximation \(s_n^\alpha\) attains the same
rate of convergence as the exact estimator \(f_n^\alpha\).

\begin{theorem}\label{theorem:splineasgood}
  Suppose \(X_n\) is the space of \(r\)th-order splines, \(r>m\), with uniform knot spacing \(\gamma\) on \([a,b]\) and uniform knot spacing \(\gamma^*\) on \([a - \gamma^*m,a]\) and \([b,b+\gamma^*m]\), with the condition that for all \(s \in X_n\), and for \(0 \leq k \leq m - 1\), we have \(s^{(k)}(a - \gamma^*m) = s^{(k)}(b+\gamma^*m) = 0\). Take our spline estimate to be \(s_n^\alpha = \argmin_{s \in X_n}\;\|g*s - h_n\|^2 + \alpha \|s^{(2)}\|^2\). Suppose also that \(\mu_{\hat Y} = \mathbb{E}\int_{-\infty}^\infty |x|h_n(x)\,dx = O(1)\). Adopt the assumptions of either Theorem~\ref{theorem:smoothrates} or Theorem~\ref{theorem:rates}, and let \(r_n = \mathbb{E}\|f_n^\alpha - f\|^2\) denote the resulting rate of convergence of the exact estimator. Choose \(\gamma,\gamma^*,a\), and \(b\) so that \(\alpha^{-4}\gamma^{2(r-m)} = o(r_n)\), \(\alpha^{-4} \gamma^*(1+\gamma) = o(r_n)\), and \(\alpha^{-4}(|a|\wedge|b|)^{-1} = o(r_n)\). Then
  \(
    \mathbb{E}\|s_n^\alpha - f_n^\alpha\|^2 = o(r_n).
  \)
  It follows also that
  \(
    \mathbb{E}\|s_n^\alpha - f\|^2 = O(r_n).
  \)
\end{theorem}
Proof deferred to Supplemental Proof~\ref{defproof:theorem:splineasgood}

Now we describe how we compute \(s_n^\alpha\) in concrete terms. In
all of the following, unless otherwise stated, we fix \(m = 2\).  Fix
\(r = 3\), and let \(\mathscr{S}_q =
\mathscr{S}(r=3,\xi_1,\dots,\xi_{q+4})\) denote the space of cubic
splines (cf. \cite[Chapter 3]{powell_approximation_1981}) with knots
\(\xi_1, \dots, \xi_{q+4}\), with \(\xi_1 < Y_{(1)}\) and \(\xi_{q+4}
> Y_{(n)}\), evenly spaced knots, no knots of multiplicity larger than
one, and end conditions \(s^{(k)}(\xi_1) = s^{(k)}(\xi_{q+4}) = 0\)
for \(k = 0, 1, 2\). The end conditions specify that members of
\(\mathscr{S}_q\) vanish outside the interval \([\xi_1,\xi_{q+4}]\)
and are twice continuously-differentiable at the boundary. This space
\(\mathscr{S}_q\) has as a basis the collection of \(q\) unit-integral
B-splines \(\{b_i\}_{i=1}^q\), so that if \(s \in \mathscr{S}_q\),
then \(s(x) = \sum_{i=1}^q \theta_i b_i(x)\).  Note that
\(\mathscr{S}_q \subset H^m(\mathbb{R})\).

We now take as our estimate
\(s_n^\alpha = \argmin_{s \in \mathscr{S}_q}\;\|g*s - h_n\|^2 + \alpha \|s^{(2)}\|^2.\)
If \(s(x) = \sum_{i=1}^q \theta_i b_i(x)\), then
\(\|g*s - h_n\|^2 + \alpha\|s^{(2)}\|^2 = \bm{\theta}^T \mathbf{M}
\bm{\theta} - 2\bm{\theta}^T\mathbf{d} + \|h_n\|^2 + \alpha
\bm{\theta}^T \mathbf{P} \bm{\theta},\)
where \(\bm{\theta}\) is the vector of coefficients \(\theta_i\), and
\(\mathbf{M}\), \(\mathbf{d}\), and \(\mathbf{P}\) are a \(q\times q\)
matrix, \(q \times 1\) vector, and \(q \times q\) matrix respectively,
with typical entries \(M_{ij} = \int (g*b_i)(g*b_j)\), \(d_i = \int
(g*b_i) h_n\), and \(P_{ij} = \int b_i^{(2)}b_j^{(2)}\).

With this matrix representation, we can see, using standard
techniques, and noting that \(\|h_n\|^2\) does not depend on
\(s_n^\alpha\), that the coefficients of \(s_n^\alpha\) are
\(\bm{\theta}_n^\alpha = (\mathbf{M} + \alpha
\mathbf{P})^{-1}\mathbf{d}\), so that \(s_n^\alpha(x) = \sum_{i=1}^n
\theta_{n,i}^\alpha b_i(x)\).  Analogous to the exact solution,
\(s_n^\alpha\) need not be a pdf. To produce a pdf, we now solve
\begin{equation}\label{eqn:splineproject}
  \mathring s_n^\alpha = \argmin_{\substack{\int s = 1\\s \geq 0}}\;\|s - s_n^\alpha\|^2,
\end{equation}
At this stage, other linear constraints my be introduced by expressing
them against the B-spline basis. If \(\mathbf{G}\) is a matrix with
typical entry \(G_{ij} = \int b_ib_j\), and, letting \(\xi_1 = x_1 <
x_2, \dots, x_{n_x} = \xi_{q+4}\) be a grid of evenly spaced values on
the support of \(\mathscr{S}_q\), with \(\mathbf{B}_x\) the \(n_x
\times q\) matrix with \(i,j\)th entry \(b_j(x_i)\), the coefficients
of the solution to Equation~\eqref{eqn:splineproject} are given
(approximately) by the linearly-constrained quadratic program
\begin{equation}
  \label{eqn:vectorproject}
  \mathring{\bm{\theta}}_n^\alpha = \argmin_{\substack{\mathbf{1}^T\bm{\theta} =
      1\\\mathbf{B}_x\bm{\theta} \geq 0}}\;(\bm{\theta} - \bm{\theta}_n^\alpha)^T\mathbf{G}(\bm{\theta} - \bm{\theta}_n^\alpha).
\end{equation}
The reason this is approximate is that the the convex constraint
\(s(x) \geq 0\) for all \(x\) is approximated by the collection of
linear constraints \(s(x_i) = \sum_{j=1}^q \theta_jb_j(x_i) \geq 0\),
\(i = 1, \dots, n_x\). Equation~\eqref{eqn:vectorproject} is a
quadratic program with \(q\)-dimensional objective and \(n_x+1\)
linear constraints.

The entries of \(\mathbf{G}\) and \(\mathbf{P}\) can be computed by
hand from the piecewise-polynomial representation of the B-splines.
Computing the entries of \(\mathbf{M}\) and \(\mathbf{d}\) benefits
from the Fourier representation
\begin{equation}
  M_{ij} = \frac{1}{2\pi} \int \tilde b_i(\omega)\overline{\tilde
    b_j(\omega)}|\tilde g(\omega)|^2\,d\omega\qquad\text{and}\qquad
  d_i = \frac{1}{2\pi} \int \tilde g(\omega) \tilde b_i(\omega)
  \overline{\tilde h_n(\omega)}\,d\omega,
\end{equation}
which can then be computed by an appropriate quadrature, bypassing the
problem of dealing with the convolutions. When \(h_n\) is a kernel
density estimate or a histogram, \(\tilde h_n\) is not difficult to
compute, and the \(\tilde b_i\) are straightforward to compute, as
B-spline basis functions can be represented as shifted, scaled
self-convolutions of \(\mathds{1}_{[0,1]}\).

\subsection{Finite Sample Behavior}
\label{sec:finite}

In \cite{wand_finite_1998}, the author points out that while
asymptotic rates for deconvolution are very slow no matter the size of
the measurement error (cf. Theorem~\ref{theorem:rates} here,
\cite{stefanski_rates_1990}, \cite{zhang_fourier_1990},
\cite{fan_optimal_1991}), there is another side of the coin: for very
small measurement error we ought to expect to be able to estimate
\(f\) with MISE quite close to that of the error-free setting. For
example, we could simply ignore measurement error and increase our
MISE by at most \(\|g*f - f\|^2\), which becomes arbitrarily small as
the measurement error decreases. Thus, we might expect that the
pessimistic picture given by asymptotic rates is limited to truly
large samples, especially when measurement error is small, and a
direct investigation into small-sample behavior is required for a
better understanding of deconvolution estimators.

To get a handle on the small-sample behavior, \cite{wand_finite_1998}
creates two products for the deconvoluting kernel estimator: a log-log
plot of the minimum attainable MISE, i.e.\ \(\inf_{\lambda > 0}
\mathbb{E}\|f_n^\lambda - f\|^2\), against the sample size, as well
as a table listing the smallest sample size required for the minimum
attainable MISE in deconvolution to be at least as small as the
minimum attainable MISE in the no-measurement-error case with some
fixed sample size.

We will investigate these same properties for analogous quantity, the
minimum attainable MISE for the SPeD estimator, given by
\(\inf_{\alpha>0} \mathbb{E}\|f_n^\alpha - f\|^2\). Supplemental
Figure~\ref{fig:apprxMISE} shows a plot of \(\mathbb{E}\|f_n^\alpha -
f\|^2\) as a function of \(\alpha\).  Since the MISE involves unknown
quantities, in practice \(\alpha\) will have to be chosen from the
data, and the search for a good data-driven choice of \(\alpha\) is
ongoing; in Section~\ref{sec:cyto}, we use what is essentially an
iterated bootstrap, but at this point do not claim that it is optimal.

The settings addressed in \cite{wand_finite_1998}, which we will use
here as well, are as follows. The target random variable \(X\) has one
of the following densities; (i) standard normal, (ii) normal mixture
\(\frac23 \N(0,\,\sigma=1) + \frac13 \N(0,\,\sigma=\frac15)\), (iii)
\(\Gammadist(\zeta=4,\,\beta=1)\), (iv) gamma mixture \(\frac25
\Gammadist(\zeta=5,\,\beta=1) + \frac35
\Gammadist(\zeta=13,\,\beta=1)\), with \(\zeta\) and \(\beta\) the
shape and rate parameters, respectively. We will consider normal
measurement error \(E\) with \(\Var(E) = p\cdot \Var(Y)\), with
various choices of \(p\).

To investigate these properties for the smoothness-penalized
deconvolution estimator, we must compute the MISE of our estimator,
\(\mise(f_n^\alpha) = \mathbb{E} \int (f_n^\alpha - f)^2\). From
Theorem~\ref{theorem:representing}\ref{theorem:representing:mult}, we have that the Fourier
transform of the estimate is given by \(\tilde \varphi_\alpha(\omega)
\tilde h_n(\omega)\); to simplify calculations, we approximate
\(\tilde h_n\) by the Fourier transform of the empirical distribution,
\(\tilde P_n(\omega) = \frac1n \sum_{j=1}^n e^{-i\omega Y_j}\), using
instead \(\tilde f_n^\alpha(\omega) = \tilde \varphi_\alpha(\omega)
\tilde P_n(\omega) = \sum_{j=1}^n \tilde \varphi_\alpha(\omega)
e^{-i\omega Y_j}\). Even though we have replaced the density estimate
\(h_n\) by the empirical distribution, which has no density at all,
the approximation is quite good; see Figure~\ref{fig:apprxMISE} in the
supplemental material. The resulting MISE, derived in
Fact~\ref{fact:apprxMISE}, is
\begin{equation}
  \label{eqn:apprxMISE}
  \mise(\alpha) = \frac1{2\pi} \left [ \int | \tilde \varphi_\alpha(\omega) \tilde g(\omega) - 1 |^2 | \tilde f(\omega) |^2\,dt + \frac1n \int |\tilde \varphi_\alpha(\omega)|^2(1 - |\tilde g(\omega)\tilde f(\omega)|^2) d\omega \right ]
\end{equation}
which we will evaluate numerically in the following.

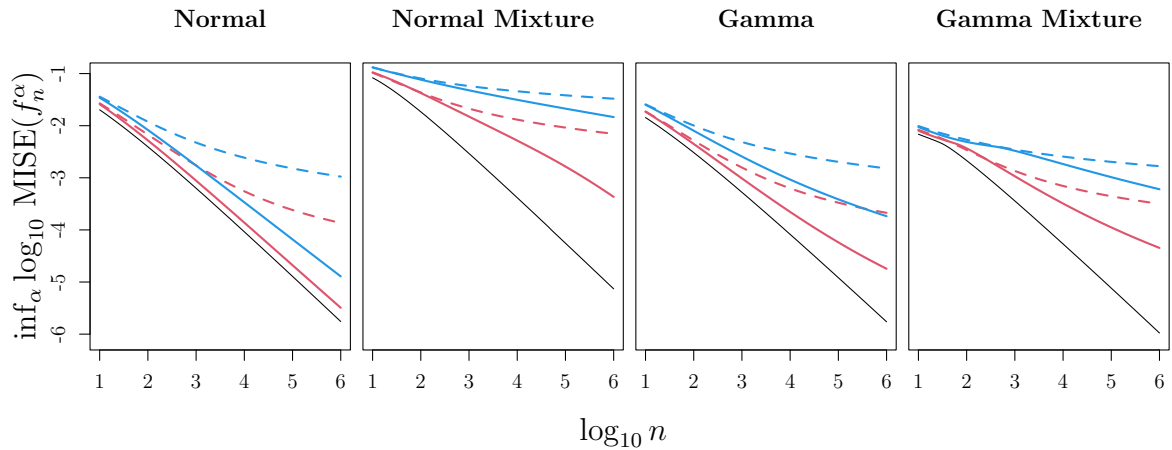
\begin{figure}[h]
\input{tikz_plots/infMISEs_all}\caption{MISE all under oracle choice of tuning parameter, densities (i)-(iv), left-to-right. Solid black is MISE for kernel estimator in the error-free setting. Solid lines are SPeD, and dashed lines are DKE. Red and blue lines have \(p = 0.1,\,0.3\), where \(\Var(E) = p\cdot \Var(Y)\). }
\label{fig:infMISEs}
\end{figure}

In Figure \ref{fig:infMISEs}, we show plots of best-attainable MISE,
i.e.\ \(\inf_{\alpha > 0} \mathbb{E}\|f_n^\alpha - f\|^2\), for the
SPeD (computed via Equation~\eqref{eqn:apprxMISE}), and the same, but
with infimum over the bandwidth, for the DKE and a conventional kernel
estimator on the non-contaminated \(X\)'s for reference (both computed
via the expressions in \cite{wand_finite_1998}). For the deconvoluting
kernel density estimate, we use a base kernel \(K_{\text{DKE}}\) with
Fourier transform \(\kappa_{\text{DKE}}(\omega) =
\mathds{1}_{|\omega|<1}(1-\omega^2)^3\); this is \(\kappa_1\) in
\cite{wand_finite_1998}, and is the default choice in the
\texttt{deconvolve} R package \cite{delaigle_deconvolve_R}. For the
error-free kernel estimator, we use kernel with Fourier transform
\(\kappa_{\text{ef}}(\omega) = (1+\omega^4)^{-1}\). This relates to
the smoothness-penalized deconvolution estimator in the following
sense: the error-free setting is equivalent to the measurement error
problem where \(E\) is a point-mass at zero. In that case, \(\tilde
g(\omega) = 1\), and then \(\tilde \varphi_\alpha(\omega) =
(1+\alpha\omega^{2m})^{-1}\). If we replace \(\tilde h\) by \(\tilde
P_n\) again in
Theorem~\ref{theorem:representing}\ref{theorem:representing:bound}, we
have a kernel estimator with \(K_{\text{ef}}(x) = \frac{1}{2\pi}\int
e^{i\omega x}\kappa_{\text{ef}}(\omega)\,d\omega\). Note that \(\int
K_{\text{ef}} = 1\), but \(K_{\text{ef}}\) is not non-negative. In
fact, when \(m = 2\), \(K_{\text{ef}}\) is a fourth-order kernel.

\begin{table}[h]
  \caption{Minimum sample sizes for stated estimator, with \(p\) the
    proportion of measurement error, to achieve MISE as small as
    error-free kernel density estimation on the \(X\)'s with kernel
    \(K_{\text{ef}}\). The analogous value with respect to kernel
    \(K_{\text{DKE}}\) is in parentheses.}
  \spacingset{1}  
  \input{tables/MISE_min_n.tex}
  \spacingset{1.9}
  \label{tab:equivn}
\end{table}

In Figure~\ref{fig:infMISEs}, the smoothness-penalized deconvolution
estimator gives a much more optimistic picture of the deconvolution
problem in finite samples compared to the deconvoluting kernel
estimator. The SPeD has nearly uniformly lower MISE, excepting a small
range of \(n\) in setting (iv). In setting (i), which satisfies the
conditions of Theorem~\ref{theorem:smoothrates}, the SPeD under 30\%
measurement error has better optimal MISE than the DKE under 10\%
measurement error, for sample sizes small enough to be commonly
encountered in practice.

Table~\ref{tab:equivn} lists sample sizes required for the
deconvolution estimators to attain MISE as small as the error-free
setting. We can see that in every case listed in the table, the SPeD
requires fewer samples than the DKE; in some cases the difference is
dramatic. To achieve the same MISE as estimating the Gamma mixture
density in setting (iv) in the error-free setting with a sample of
size \(n = 1,000\) when there is 10\% measurement error, the SPeD
would require \(7,963\) samples, while the DKE would require
\(388,770\) samples. In practice, this may mean the difference between
an expensive experiment and an impossible one. Another takeaway is
how strongly the required \(n\) varies with the target density. In
setting (i), the problem does not seem so bad; in setting (ii), it
seems all but impossible.

\subsection{Application to Cytotoxicity Data}
\label{sec:cyto}
\begin{figure}[h]
\center
\input{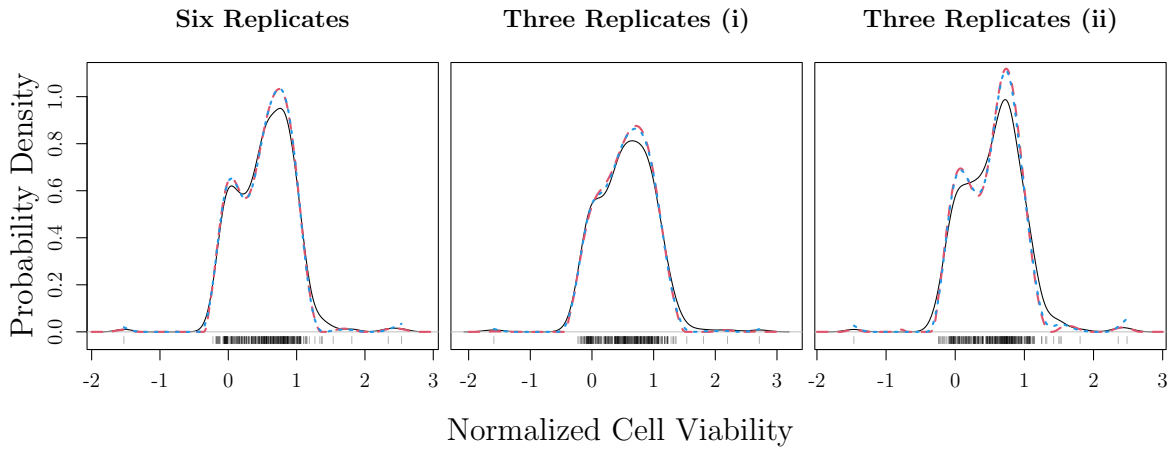}\caption{Density estimates of cytotoxicity data described in Section~\ref{sec:cyto}. Standard Gaussian kernel density estimate of the \(Y_i\) as solid black line. Smoothness-penalized density estimate of the \(X_i\) as dashed red line; QP estimator as dotted blue line. Individual data locations marked below plot. Leftmost panel is full data; right two panels each use only three of the available six replicates for each measurement.}
\label{fig:real_data}
\end{figure}
\emph{Bacillus cereus sensu lato (s.l)} is a group of closely-related
bacteria with diverse relationships to humans, including
\emph{B. thuringiensis}, which is used on crops as a pesticide,
\emph{B. anthracis}, which can cause anthrax disease, and others which
can cause other illness and spoil food \cite{ceuppens_diversity_2013}.
These bacteria are ubiquitous in many environments, their taxonomy is
``complex and equivocal,'' \cite{ceuppens_diversity_2013}, and
distinguishing between members of \emph{B. cereus s.l.} with typical
methods can be difficult. Scientists are therefore interested in
developing practical laboratory tests which can readily discriminate
between harmful representatives of this group and those less likely to
cause harm.

As one element of that investigation, a colleague requires a density
estimate of a certain conditional expectation. Suppose \(i\) is an
isolate of \emph{B. cereus s.l.}, sampled from a large
collection. Suppose it is cultured under certain conditions,
centrifuged, and the supernatant is applied to human cells. Let
\(X_i\) denote the mean normalized cytotoxicity of isolate \(i\), and
\(C_{ij} = X_i + \varepsilon_{ij}\) denote the cytotoxicity observed
the \(j\)th time this procedure is applied to isolate \(i\), and
further assume that the \(\varepsilon_{ij}\), are i.i.d., have mean
zero and are independent of \(X_i\). We are interested in the density
\(f\) of \(X_i\) as \(i\) varies over the collection of
isolates. However, the investigator only has access to a sample
approximation \(Y_i = \frac{1}{k}\sum_{j=1}^k C_{ij}\) of \(X_i\)
obtained by fixing \(i\) and repeatedly measuring the
cytotoxicity. With \(E_i = \frac1k \sum_{j=1}^k \varepsilon_{ij}\), we
are in the setting described in the introduction. We do not know the
density of \(g\) of \(E_i\) exactly, as assumed for the theory;
however, we may approximate it by \(\N(0,\sigma^2_\varepsilon/k)\) as
long as the \(\varepsilon_{ij}\) are not too skewed. We then only need
to estimate \(\sigma_\varepsilon^2\), which can be done at parametric
rates much faster than the rates involved in deconvolution.

We have been provided preliminary data, which comprise a table of
measured cytotoxicity \(C_{ij}\) from \(j = 1, \dots, k = 6\) replicates
of isolates \(i = 1, \dots, n = 313\). We have estimated
\(\sigma^2_\varepsilon\) by fitting the linear model \(C_{ij} = X_i +
\varepsilon_{ij}\) in R and extracting the residual standard
error. Tuning parameter \(\alpha\) was chosen by picking an arbitrary
provisional \(\alpha_0\), seeking \(\alpha_i\) which minimizes
\(\mathbb{E}\|s_n^\alpha - s_n^{\alpha_{i-1}}\|^2\) assuming the
\(X_i\) have pdf \(s_n^{\alpha_{i-1}}\), and iterating until
convergence. The results are shown in Figure~\ref{fig:real_data},
along with a standard kernel density estimate of the \(Y_i\). This
example has a relatively small amount of measurement error, with
proportion \(p = \Var(E)/\Var(Y) \approx 0.045\). To illustrate SPeD
with greater measurement error and to see if the number of replicates
may be reduced in future experiments, we have also split the
replicates randomly into two groups (i) and (ii), and re-fit the
estimator as if there were only three available replicates. This
yields \(p \approx 0.088\) and \(p \approx 0.082\) for groups (i) and
(ii), respectively. The two modes present in the full data are blurred
to one mode in the reduced data, but our estimator does recover two
modes in one of the two reduced data settings.

\if0\blind {
\section{Acknowledgements}
The authors thank Professor Jasna Kovac for sharing with us the
\emph{B. cereus} cytotoxicity data, and Professor Kengo Kato for
several helpful conversations.
} \fi

\clearpage

\bibliography{ref.bib}
\makeatletter\@input{supplementalaux.tex}\makeatother
\end{document}


\def\spacingset#1{\renewcommand{\baselinestretch}%
{#1}\small\normalsize} \spacingset{1}


\if0\blind
{
  \title{\bf Smoothness-Penalized Deconvolution (SPeD) of a Density Estimate: Supplemental Material}
  \author{David Kent\footnote{Department of Statistics and Data Science, Cornell University} \footnote{This work was supported by the National Science Foundation under Grant AST-1814840. The opinions, findings, and conclusions, or recommendations expressed are those of the authors and do not necessarily reflect the views of the National Science Foundation.}\ \ and David Ruppert\footnotemark[1] \footnotemark[2] \footnote{School of Operations Research and Information Engineering, Cornell University}\hspace{.2cm}}
  \maketitle
} \fi

\if1\blind
{
  \title{\bf Smoothness-Penalized Deconvolution (SPeD) of a Density Estimate: Supplemental Material}
  \maketitle
} \fi

%

\newpage
\spacingset{1.9} 

\appendix
\section{Supporting Facts}

\input{sections/supporting}

\begin{prop}\label{prop:alphasmoothed}
  This \(\alpha\)-smoothed \(f\) has the following properties:
  \begin{enumerate}[label=(\roman*)]
  \item \(\tilde f^\alpha(\omega) = \tilde \varphi_\alpha(\omega) \tilde h(\omega)\),
  \item \(f^\alpha(x) = \varphi_\alpha*h(x)\),
  \item If \(f \in H^m(\mathbb{R})\), then \(\|D^mf^\alpha\| \leq \|D^mf\|\), and
  \item Under Assumption~\ref{assumption:gzeros}, \(\|f^\alpha - f\| \to 0\) as \(\alpha \to 0\).
  \end{enumerate}
\end{prop}
\input{proofs/prop_alphasmoothed}

\section{Deferred Proofs}

\stepcounter{defproof}

\begin{defproof} \label{defproof:theorem:representing}
  \input{proofs/theorem_representingsupp}
\end{defproof}

\stepcounter{defproof}

\stepcounter{defproof}

\begin{figure}\label{fig:laplace_proof}
\center
\input{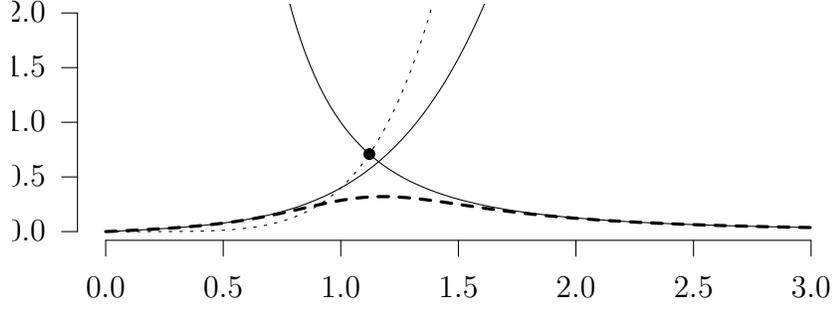}
\caption{Approach in proof of Lemma \ref{lem:systematic_bounds}, with \(m = 2, k = 2, \alpha = 1/10\). Thick dashed line is \(|\theta(\omega)|\). Thin solid lines are \(\alpha \omega^{2m-k}(1+\omega^2)^2\) and \(\omega^{-k}\). Thin dotted line is \(4\alpha\omega^{2m-k+2}\). Marked point is upper bound for \(|\theta(\omega)|\).}
\end{figure}
\begin{defproof} \label{defproof:lem:systematic_bounds}
  \input{proofs/lem_systematic_bounds}
\end{defproof}

\stepcounter{defproof}

\stepcounter{defproof}

\stepcounter{defproof}

\begin{defproof} \label{defproof:theorem:rates}
  \input{proofs/theorem_rates}
\end{defproof}

\stepcounter{defproof}

\begin{defproof} \label{defproof:prop:particularlaplace}
  \input{proofs/prop_particularlaplace}
\end{defproof}

\begin{defproof} \label{defproof:lem:Caclosed}
  \input{proofs/lem_Caclosed}
\end{defproof}

\begin{defproof} \label{defproof:lem:projf}
  \input{proofs/lem_projf}
\end{defproof}

\stepcounter{defproof}

\begin{figure}
\center
\includegraphics{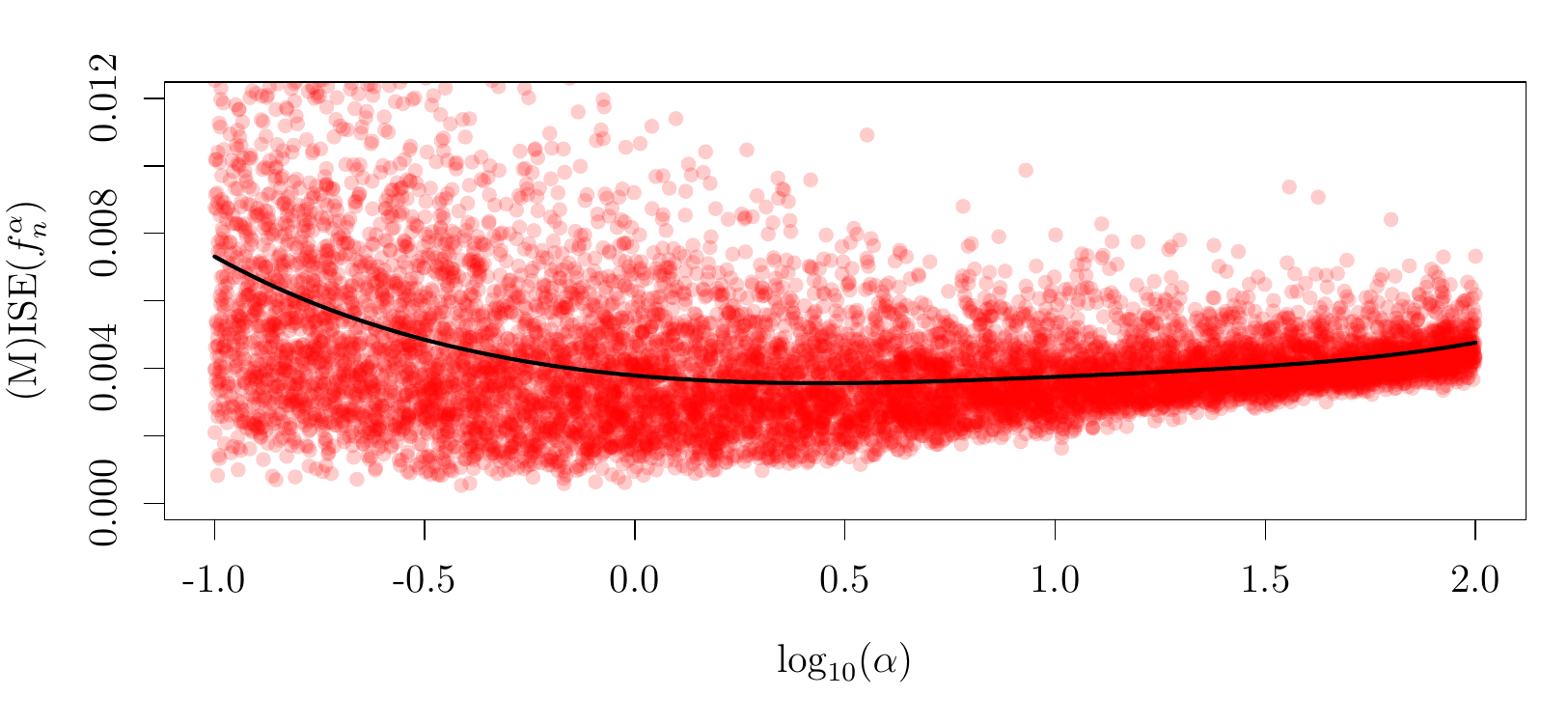}
\includegraphics{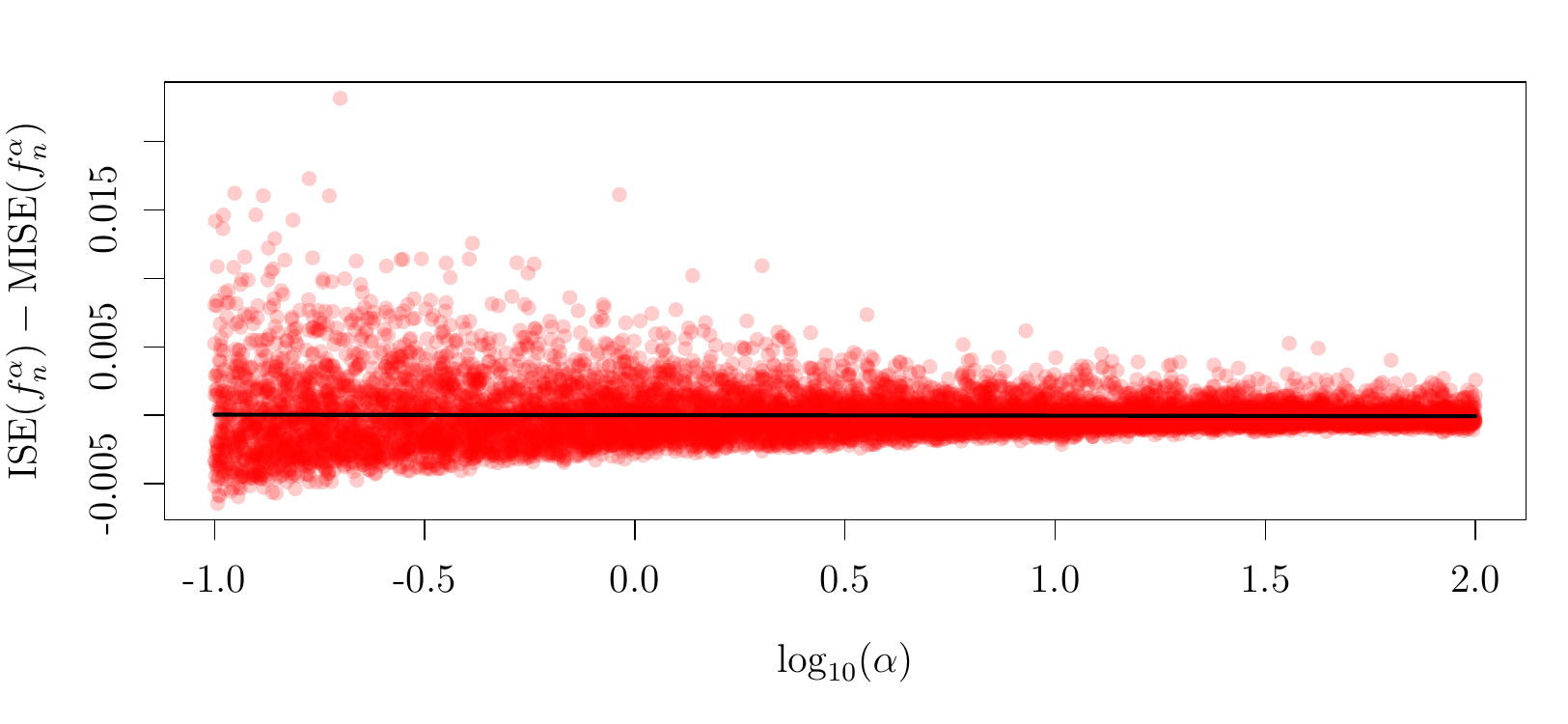}
\caption{\(n_{sim} = 10^4\) simulations from the bimodal gamma mixture
density (iv), with normal measurement errors \(E\), \(\Var(E) =
\Var(X)/10\), and \(n = 100\). The estimator was computed according to
Section~\ref{sec:computing}. In top panel, points are the integrated
squared errors, and the line is the approximate \(\mise(f_n^\alpha)\)
computed according to Equation~\ref{eqn:apprxMISE}. In bottom panel,
points are the difference between the approximated MISE and the
simulated integrated squared errors; the line is a smoothing spline
with tuning parameter chosen by generalized cross-validation.}
\label{fig:apprxMISE}
\end{figure}

\section{Technical Results}

First, we present a very slightly modified version of \cite[Theorem 2.C]{zeidlerAppliedFunctionalAnalysis1995}.

\begin{lemma}\label{lem:ritz}
Consider the problem
\begin{equation}
  u = \argmin_{v \in H^m(\mathbb{R})} \|g*v - h_n\|^2 + \alpha \|v^{(m)}\|^2
\end{equation}
and its approximation
\begin{equation}
  u_n = \argmin_{v \in X_n} \|g*v - h_n\|^2 + \alpha \|v^{(m)}\|^2
\end{equation}
If \(X_n\) is a finite-dimensional linear subspace of \(H^{m}(\mathbb{R})\), then for sufficiently small \(\alpha\),
\begin{equation}
  \|u - u_n\|^2 \leq \frac{c}{\alpha^2} \cdot \inf_{v \in X_n} [\|u - v\|^2 + \|u^{(m)} - v^{(m)}\|^2]
\end{equation}
\end{lemma}
\begin{proof}
  First note that with \(a(u,v) = \langle g*u, g*v \rangle + \alpha \langle u^{(m)},v^{(m)} \rangle\) and \(b(u) = \langle g*u,h_n \rangle\), we have \(\|g*v - h_n\|^2 + \alpha \|v^{(m)}\|^2 = a(v,v) - 2b(v) + \|h_n\|^2\). It can be easily checked that \(a(\cdot,\cdot)\) is symmetric, bilinear, and in a moment we will show that it is bounded and strongly positive with respect to the Sobolev norm \(\|u\|^2_X = \|u\|^2 + \|u^{(m)}\|^2\). Applying \cite[Theorem 2.A]{zeidlerAppliedFunctionalAnalysis1995} (and noting that \(X_n\) is a Hilbert space under the same norm), we find that \(a(u,v) = b(v)\) for all \(v \in H^{m}(\mathbb{R})\), and \(a(u_n,v) = b(v)\) for all \(v \in X_n \subset H^m(\mathbb{R})\). Substracting these yields that \(a(u - u_n,v) = 0\) for all \(v \in X_n\), and taking \(v \vcentcolon= u_n\), we also see that \(a(u - u_n,u_n) = 0\). Subtracting these last two from the identity \(a(u-u_n,u) = a(u-u_n,u)\) yields 
  \begin{equation}
    \label{eqn:aequality}
    a(u - u_n,u-u_n) = a(u-u_n,u-v)\text{ for all }v \in X_n.
  \end{equation}
  First we show that \(\frac{\alpha}{c} \|v\|^2 \leq a(v,v)\) for some \(c>0\) and sufficiently small \(\alpha>0\). Choose \(c'\geq 1\) so that \(c'(|\tilde g(\omega)|^2 + |i\omega|^{2m}) \geq 1\) for all \(\omega \in \mathbb{R}\). Then
  \begin{equation}
    \begin{aligned}
      \|v\|_X^2 &= \|v\|^2 + \|v^{(m)}\|^2\\
      &= \frac1{2\pi} \int |\tilde v(\omega)|^2\,d\omega + \|v^{(m)}\|^2\\
      &\leq \frac{c'}{2\pi} \int |\tilde g(\omega) \tilde v(\omega)|^2\,d\omega + \frac{c'}{2\pi} \int |(i\omega)^m \tilde v(\omega)|^2 \,d\omega + \|v^{(m)}\|^2\\
      &= c' \|g*v\|^2 + (1+c')\|v^{(m)}\|^2\\
      &\leq 2c'(\|g*v\|^2 + \|v^{(m)}\|^2)\\
      &\leq \frac{2c'}{\alpha} (\|g*v\|^2 + \alpha \|v^{(m)}\|^2)\\
      &= \frac{2c'}{\alpha} a(v,v)
    \end{aligned}
  \end{equation}
  for sufficiently small \(\alpha\).

  Now we show that for any \(v,w \in H^m(\mathbb{R})\), we have \(|a(v,w)| \leq \sqrt{2} \|v\|_X \|w\|_X\) for small enough \(\alpha\). We begin by applying the inequality \((a+b)^2 \leq 2a^2 + 2b^2\):
  \begin{equation}
    \begin{aligned}
      |a(v,w)|^2 = |\langle g*v,g*w\rangle + \alpha \langle v^{(m)},w^{(m)}\rangle|^2 &\leq 2\langle g*v,g*w\rangle^2 + 2\alpha^2 \langle v^{(m)},w^{(m)}\rangle^2 \\
      \text{(a)}\qquad&\leq 2(\|g*v\|^2 \|g*w\|^2 + \|v^{(m)}\|^2 \|w^{(m)}\|^2)\\
      \text{(b)}\qquad&\leq 2(\|v\|^2\|w\|^2 + \|v^{(m)}\|^2 \|w^{(m)}\|^2)\\
      &\leq 2\|v\|_X^2 \|w\|_X^2,
    \end{aligned}
  \end{equation}
  where (a) is by Cauchy Schwarz and for \(\alpha\) small enough, (b) is by Young's convolution inequality, and the final line follows from adding the positive term \(\|v\|^2\|w^{(m)}\|^2 + \|v\|^2 \|w^{(m)}\|^2\).

  Plugging in \(u - u_n\) to the first result, and then stringing them together, we find, for all \(v \in X_n\),
  \begin{equation}
    \begin{aligned}
      \left ( \frac{\alpha}{2c'} \right )^2 \|u-u_n\|_X^4 &\leq a(u-u_n,u-u_n)^2 \\
      &= a(u-u_n,u-v)^2\\
      &\leq 2\|u-u_n\|_X^2 \|u-v\|_X^2,
    \end{aligned}
  \end{equation}
  so that for all \(v \in X_n\), we have \(\|u-u_n\|_X^2 \leq \frac{c}{\alpha^2} \|u-v\|_X^2\). Thus
  \begin{equation}
    \|u-u_n\|^2 \leq \|u-u_n\|_X^2 \leq \frac{c}{\alpha^2} \inf_{v \in X_n} [\|u - v\|^2 + \|u^{(m)} - v^{(m)}\|^2],
  \end{equation}
  as needed.
\end{proof}

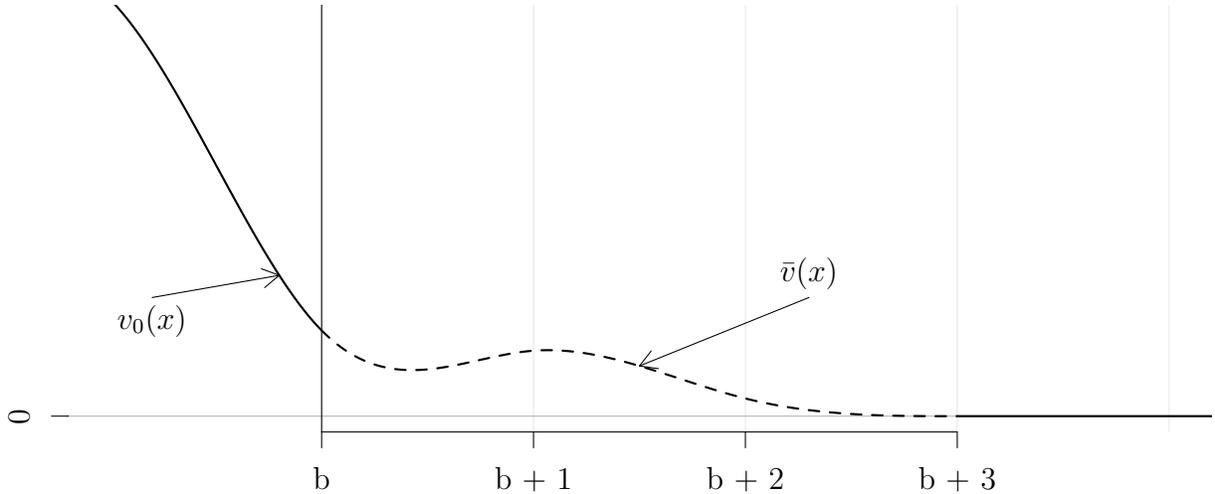
\begin{figure}
\center
\input{tikz_plots/lemma16.tex}
\caption{Picture of Lemma~\ref{lem:splineextension} for \(m = 3\). The function \(v_0(x)\), drawn as a solid line, is a spline defined on an interval \([a,b]\). The function \(\bar v(x)\), drawn as a dashed line, extends \(v_0(x)\) to a spline on \([a,b+m]\) which is equal to \(v_0(x)\) on \([a,b]\) but has vanishing derivatives of order up to \(m-1\) at \(b+m\).}
\label{fig:lemma16}
\end{figure}

To apply Lemma~\ref{lem:ritz}, we need our spline space \(X_n \subset H^m(\mathbb{R})\), but we'd like to use approximation theory that has been developed for splines on an interval \([a,b]\). We show in the following that such splines can be extended to functions with domain \(\mathbb{R}\) which belong to \(H^m(\mathbb{R})\), i.e. functions which have \(m\) square-integrable weak derivatives on \(\mathbb{R}\), and furthermore that the function and its derivatives on \(\mathbb{R}\setminus[a,b]\) are small.

In the following lemma, we show essentially that a spline on an interval \([a,b]\) can be extended to a spline on the entire real line, and that the norm of the added piece is modulated by the value of the derivatives at the endpoints \(a\) and \(b\).
\begin{lemma}\label{lem:splineextension}
  Suppose a function \(v_0(x)\) is an \(r\)th-order spline on \([a,b]\), \(r>m\). Then there is a function \(\bar v~\vcentcolon~[b,b+m] \to \mathbb{R}\) satisfying \(\bar v^{(k)}(b+m) = 0\) for \(k = 0, \dots, m-1\), which is polynomial on \([b+k,b+k+1]\) for \(k = 0, \dots, m-1\), so that the extension \(v~\vcentcolon~[a,b+m]\to\mathbb{R}\) given by \(v(x) = v_0(x)\) if \(x \in [a,b]\) and \(v(x) = \bar v(x)\) if \(x \in (b,b+m]\) is an \(r\)th order spline on \([a,b+m]\) with derivatives up to order \(m-1\) vanishing at \(b+m\). Furthermore, \(\int_b^{b+m} \bar v(x)^2\,dx~\leq~C_0 \sum_{k=0}^{r-1} v_0^{(k)}(b)^2\) and \(\int_b^{b+m} \bar v^{(m)}(x)^2\,dx~\leq~C_m \sum_{k=0}^{r-1} v_0^{(k)}(b)^2\)
\end{lemma}

\begin{proof}
  WLOG, let \(b = 0\).
  Since \(v_0(x)\) is an \(r\)th-order spline, it is represented by an \(r\)th-order polynomial near \(0\), say \(v_0(x) = P_0(x)\) on \((-\varepsilon,0]\).
  Then if we write \(\bar v(x) = P_0(x) + \sum_{i=1}^m c_i(x - (i-1))_+^{r-1}\), the function \(v(x)\) defined above will be an \(r\)th-order spline on \([a,m]\).

  First, we show th
  The requirement \(\bar v^{(k)}(m) = 0\) for \(k = 0, \dots, m-1\) can be expressed as the \(m\) linear equations in the \(c_i\):
  \begin{equation}
  \sum_{i=1}^m -\frac{(r-1)!}{(r-1-k)!} c_i (m - (i-1))_+^{r-1-k} = P_0^{(k)}(m),
  \end{equation}
  for \(k = 0, \dots, m-1\).
  Then if \(\mathbf{A}\) is the \(m \times m\) matrix with entries \(\mathbf{A}_{ji} = -\frac{(r-1)!}{(r-j)!} (m - (i-1))_+^{r-j}\) and \(\mathbf{q}\) is the \(m \times 1\) vector with entries \(\mathbf{q}_j = P_0^{(j-1)}(m)\), we see that the coefficients \(\mathbf{c} = \mathbf{A}^{-1}\mathbf{q}\) are a linear function of the first \(m\) derivatives of \(P_0(x)\) at \(x = m\).
  Now, letting \(\mathbf{p}\) be the \(r \times 1\) vector with entries \(\mathbf{p}_i = P_0^{(i-1)}(0)\), and \(\mathbf{B}\) the \(m \times r\) matrix with entries \(\mathbf{B}_{ji} = \frac{m^{i-1}}{(i-j)!}\), we see that \(\mathbf{c} = \mathbf{A}^{-1}\mathbf{B}\mathbf{p}\). Finally, we see that the coefficients of \(\bar v\) against the basis mentioned are given by \(\mathbf{C}\mathbf{p}\) where \(\mathbf{C} = [\mathbf{I}^T_r\;\;\mathbf{B}^T\mathbf{A}^{-T}]^T\).

  Now, \(\int_0^{m} \bar v^{(k)}(x)^2\,dx = \mathbf{d}^T\mathbf{G}_k\mathbf{d}\), where \(\mathbf{d}\) is the vector of coefficients against our basis, and \(\mathbf{G}_k\) is the \((m+r)\times(m+r)\) Gramian matrix of inner products of the \(k\)th derivatives of the basis functions.
  Thus we have \(\int_0^m \bar v(x)^2\,dx = \mathbf{p}^T\mathbf{C}^T \mathbf{G}_0 \mathbf{C}\mathbf{p}\) and \(\int_0^m \bar v^{(m)}(x)^2\,dx = \mathbf{p}^T\mathbf{C}^T \mathbf{G}_m \mathbf{C}\mathbf{p}\), so \(\int_0^m\bar v(x)^2\,dx \leq C_0 \sum_{i=1}^r \mathbf{p}_i^2 = C_0 \sum_{i=1}^r P_0^{(i-1)}(0)^2\) and \(\int_0^m\bar v^{(m)}(x)^2\,dx \leq  C_m \sum_{i=1}^r P_0^{(i-1)}(0)\) where \(C_0\) and \(C_m\) are the largest eigenvalues of \(\mathbf{C}^T\mathbf{G}_0\mathbf{C}\) and \(\mathbf{C}^T\mathbf{G}_m\mathbf{C}\), respectively. Recalling that \(v_0(x) = P_0(x)\) on \((-\varepsilon,0]\), this gives the result.
\end{proof}
\begin{cor}\label{cor:extend}
  Let \(v_0: [a,b] \to \mathbb{R}\) be an \(r\)th order spline with maximum grid spacing \(\Delta\). We may extend \(v_0\) to a function \(v: [a - \Delta m,b + \Delta m] \to \mathbb{R}\) with derivatives up to order \(m-1\) vanishing at \(a-\Delta m\) and \(b + \Delta m\), with \(v\) also an \(r\)th-order spline with maximum grid spacing \(\Delta\), by tacking on functions \(\overline v_-(x): [a-\Delta m,a] \to \mathbb{R}\) and \(\overline v_+(x): [b,b+\Delta m] \to \mathbb{R}\) which have the property that
  \[
    \begin{aligned}
      \int_{a - \Delta m}^a \overline v_-(x)^2 + \overline v_-^{(m)}(x)^2\,dx &+ \int_b^{b+\Delta m} \overline v_+(x)^2 + \overline v_+^{(m)}(x)^2\,dx \\
      &\leq C_- \sum_{k=0}^{r-1} \Delta^{2k+1} v_0^{(k)}(a)^2 + C_+ \sum_{k=0}^{r-1} \Delta^{2k+1} v_0^{(k)}(b)^2
    \end{aligned}
  \]
\end{cor}
\begin{proof}
  Let \(u_0(x) = v_0(\Delta x)\) on \([a/\Delta,b/\Delta]\) (an \(r\)th order spline with unit maximum grid spacing) and apply Lemma~\ref{lem:splineextension} to each end (i.e. apply to \(u_0\) and then to its reflection over \(x=0\)), yielding functions \(\overline u_-(x)\) and \(\overline u_+(x)\), which extend \(u_0\) to an \(r\)th order spline \(u\) with unit maximum grid spacing on \([a/\Delta - m, b/\Delta + m]\), and with derivatives up to order \(m-1\) vanishing at the endpoints of the interval. Then let \(\overline v_-(x) = \overline u_-(x/\Delta)\) and re-arrange the upper bound in Lemma~\ref{lem:splineextension}.
\end{proof}

\setcounter{defproof}{14}
\begin{defproof}[Proof of Theorem~\ref{theorem:splineasgood}]\label{defproof:theorem:splineasgood}
  Notice that \(X_n \subset H^m(\mathbb{R})\) if we set \(s(x) = 0\) for \(x \not\in [a - \gamma^*m,b+\gamma^*m]\) for all \(s \in X_n\). Now, Lemma~\ref{lem:ritz} yields that
  \begin{equation}
    \label{eqn:ritzbound}
    \|s_n^\alpha - f_n^\alpha\|^2 \leq \frac{c}{\alpha^2} \inf_{v \in X_n} [ \|v - f_n^\alpha\|^2 + \|D^m(v - f_n^\alpha)\|^2 ]
  \end{equation}
  We will exhibit a particular \(v^* \in X_n\) such that \(\frac{1}{\alpha^2} \mathbb{E}[ \|v^* - f_n^\alpha\|^2 + \|D^m(v^* - f_n^\alpha)\|^2  ] = o(\mathbb{E}\|f - f_n^\alpha\|^2)\)

  Consider the spline space \(\mathscr{S}_r([0,1],\Delta)\) of splines with simple knots at \(0,\Delta, 2\Delta, \dots, 1\).
  Let \(u(x) = f_n^\alpha( (b-a)x + a )\), and \(\hat u(x) = Qu(x)\) be the quasi-interpolant of \(u\) in \(\mathscr{S}_r([0,1],\Delta)\) given in \cite[Theorem 6.18]{schumaker_spline_2007}.
  We will then set \(v^*(x) = \hat u( (x-a)/(b-a) )\) for \(x \in [a,b]\), a spline with knot spacing \(\gamma = (b-a)\Delta\).
  Extend \(v^*(x)\) to \([a-\gamma^* m,b+\gamma^* m]\) by Corollary~\ref{cor:extend}, and set \(v^*(x) = 0\) for \(x \not \in [a-\gamma^* m,b + \gamma^* m]\). Note that the knot spacing \(\gamma^*\) on the extension need not match the knot spacing on \([a,b]\).
  Now we will deal with \(\|v^* - f_n^\alpha\|^2 + \|D^m(v^* - f_n^\alpha)\|^2\) by treating separately the parts on \([a,b]\) and \(\mathbb{R}\setminus [a,b]\).
  Note that
  \begin{equation}
    \begin{aligned}
      \|v^* - f_n^\alpha\|^2 &+ \|D^m(v^* - f_n^\alpha)\|^2\\
                             &\leq \|v^* - f_n^\alpha\|_{L_2([a,b])}^2 + \|D^m(v^* - f_n^\alpha)\|_{L_2([a,b])}^2\\
                             &\qquad\qquad+ 2(\|v^*\|_{L_2(\mathbb{R}\setminus[a,b])}^2 + \|D^mv^*\|_{L_2(\mathbb{R}\setminus[a,b])}^2)\\
                             &\qquad\qquad+ 2(\|f_n^\alpha\|_{L_2(\mathbb{R}\setminus[a,b])}^2 + \|D^mf_n^\alpha\|_{L_2(\mathbb{R}\setminus[a,b])}^2).
    \end{aligned}
  \end{equation}
  Thus as long as \(v^*\) is close to \(f_n^\alpha\) on \([a,b]\), and \(v^*\) and \(f_n^\alpha\) are both small outside of \([a,b]\), we will have that \(v^*\) is close to \(f_n^\alpha\) on all of \(\mathbb{R}\). We will show the following, for small enough \(\gamma\), and for \(a < 0 < b\):
  \begin{enumerate}[label=(\roman*)]
  \item \label{bounds:itemi} \(\|v^* - f_n^\alpha\|_{L_2([a,b])}^2 + \|D^m(v^* - f_n^\alpha)\|_{L_2([a,b])}^2 \leq C\alpha^{-2}\gamma^{2(r-m)}\)
  \item \label{item:itemii} \(\|v^*\|_{L_2(\mathbb{R}\setminus[a,b])}^2 + \|D^mv^*\|_{L_2(\mathbb{R}\setminus[a,b])}^2 \leq C\alpha^{-2}\gamma^*(1+\gamma)\)
  \item \label{item:itemiii}\(\|f_n^\alpha\|_{L_2(\mathbb{R}\setminus[a,b])}^2 + \|D^mf_n^\alpha\|_{L_2(\mathbb{R}\setminus[a,b])}^2 \leq \frac{C}{\alpha^2 (|a|\wedge|b|)} \int_{-\infty}^\infty |x| h_n(x)\,dx.\)
  \end{enumerate}

  To demonstrate item~\ref{bounds:itemi}, apply \cite[Theorem 6.25]{schumaker_spline_2007} with \(p = q = 2\) to yield \(\|D^k(u - Qu)\|_{L_2([0,1])} \leq C \Delta^{\sigma - k} \omega_{r-\sigma}(D^\sigma u;\Delta)_{L_2([0,1])}\), where \(m \leq \sigma \leq 2m\) is the order of a Sobolev space containing \(f_n^\alpha\). Now apply \cite[Theorem 2.59, (2.120)]{schumaker_spline_2007} to bound the modulus of smoothness, yielding \(\|D^k(u - Qu)\|_{L_2([0,1])} \leq C \Delta^{r-k} \|D^r u\|_{L_2([0,1])}\). Combining this bound with the fact that \(\|D^k(f_n^\alpha - v^*)\|_{L_2([a,b])} = (b-a)^{-(k-1)}\|D^k(u-Qu)\|_{L_2([0,1])}\) and \(\|D^r u\|_{L_2([0,1])} = (b-a)^{r-1}\|D^rf_n^\alpha\|_{L_2([a,b])}\) gives \(\|D^k(f_n^\alpha - v^*)\|_{L_2([a,b])} \leq C \gamma^{r-k} \|D^rf_n^\alpha\|_{L_2([a,b])}\). Applying Theorem~\Ref{theorem:representing}\ref{theorem:representing:apriori} brings us to \(\|D^k(f_n^\alpha - v^*)\|_{L_2([a,b])} \leq C \alpha^{-1}\gamma^{r-k}\). Squaring both sides and summing the bounds for \(k = 0\) and \(k = m\), and then noticing that for small enough \(\gamma\) we have \(\gamma^{2r} \leq C \gamma^{2(r-m)}\) gives the result.

  For item~\ref{item:itemii}, we begin with a bound for \(\|D^kv^*\|_{L_\infty([a,b])}\) and then we will apply Corollary~\ref{cor:extend}. The triangle inequality gives \(\|D^kQu\|_{L_\infty([0,1])} \leq \|D^ku\|_{L_\infty([0,1])} + \|D^k(u - Qu)\|_{L_\infty([0,1])}\). Applying to the second term the same two theorems of \cite{schumaker_spline_2007} as in the previous paragraph, but with \(p = q = \infty\), we have
  \begin{equation}
    \|D^k(u-Qu)\|_{L_\infty([0,1])} \leq C \Delta^{\sigma - k}\omega_{r-\sigma}(D^\sigma u;\Delta)_{L_\infty([0,1])} \leq C \Delta^{r - k}\|D^r u\|_{L_\infty([0,1])}.
  \end{equation}
  Now, using \(\|D^kQu\|_{L_\infty([0,1])} = (b-a)^k\|D^kv^*\|_{L_\infty([a,b])}\) and \(\|D^ku\|_{L_\infty([0,1])} = (b-a)^k \|D^kf_n^\alpha\|_{L_\infty([a,b])}\) and combining the two upper bounds, we have
  \begin{equation}
    \|D^kv^*\|_{L_\infty([a,b])} \leq \|D^k f_n^\alpha \|_{L_\infty([a,b])} + C\gamma^{(r-k)} \|D^rf_n^\alpha\|_{L_\infty([a,b])}.
  \end{equation}
  Now, for \(\gamma^* \leq M_1\),
  \begin{equation}
    \begin{aligned}
      \|v^*\|_{L_2(\mathbb{R}\setminus[a,b])}^2 + \|D^mv^*\|_{L_2(\mathbb{R}\setminus[a,b])}^2 &\leq C_1 \sum_{k=0}^{r-1} (\gamma^*)^{2k+1} [D^k v^*(a)]^2 + C_2 \sum_{k=0}^{r-1} (\gamma^*)^{2k+1} [D^k v_0^*(b)]^2\\
      &\leq C_1 \sum_{k=0}^{r-1} (\gamma^*)^{2k+1} \|D^kv^*\|_{L_\infty([a,b])}^2 + C_2 \sum_{k=0}^{r-1} (\gamma^*)^{2k+1} \|D^kv^*\|_{L_\infty([a,b])}^2\\
      &\leq C_3 \gamma^* \sum_{k=0}^{r-1}\|D^kv^*\|_{L_\infty([a,b])}^2,
    \end{aligned}
  \end{equation}
  where the first inequality is straight from Corollary~\ref{cor:extend}, the second comes from upper-bounding the values at \(a\) and \(b\) by the suprema over the entire interval, and the third follows from \(\gamma^* \leq M\), with \(C_3\) subsuming \(C_1\), \(C_2\), and the \(M_1^k\). Now we use the upper bound for \(\|D^kv^*\|_{L_\infty([a,b])}\) just derived to find, for \(\gamma < M_2\),
  \begin{equation}
    \begin{aligned}
      &\|v^*\|_{L_2(\mathbb{R}\setminus[a,b])}^2 + \|D^mv^*\|_{L_2(\mathbb{R}\setminus[a,b])}^2\\
      &\qquad\leq C_3 \gamma^* \sum_{k=0}^{r-1}\left (\|D^k f_n^\alpha \|_{L_\infty([a,b])} + C_4\gamma^{(r-k)} \|D^rf_n^\alpha\|_{L_\infty([a,b])} \right )^2\\
      &\qquad\leq C_7 \gamma^* \left (C_5\alpha^{-1} + C_6\gamma \alpha^{-1} \right )^2\\
      &\qquad\leq C_8 \frac{\gamma^*}{\alpha^2} + C_9 \frac{\gamma^*\gamma}{\alpha^2}.
    \end{aligned}
  \end{equation}

  For Item~\ref{item:itemiii}, assume \(a < 0 < b\) (which must eventually be true).
  \begin{equation}
    \begin{aligned}
      \int_{|b|}^\infty |D^kf_n^\alpha(x)|^2\,dx &\leq \frac{C}{\alpha} \int_{|b|}^\infty |D^kf_n^\alpha(x)|\,dx\\
      &= \frac{C}{\alpha} \int_{|b|}^\infty \left | \frac{1}{2\pi} \int_{-\infty}^\infty e^{i\omega x}(it)^k\tilde \varphi_\alpha(\omega)\tilde h_n(\omega)\,d\omega\right |\,dx\\
      &\leq \frac{C}{\alpha^2} \int_{|b|}^\infty \left | \frac{1}{2\pi} \int_{-\infty}^\infty e^{i\omega x}\tilde h_n(\omega)\,d\omega\right |\,dx\\
      &= \frac{C}{\alpha^2} \int_{|b|}^\infty h_n(x)\,dx\\
      &\leq \frac{C}{|b|\alpha^2} \int_{-\infty}^\infty |x| h_n(x)\,dx,
    \end{aligned}
  \end{equation}
  where the first and second inequalities are by Theorem~\ref{theorem:representing}\ref{theorem:representing:apriorisup}, and the third is by Markov's inequality. Now,
  \begin{equation}
    \begin{aligned}
      \|f_n^\alpha\|_{L_2(\mathbb{R}\setminus[a,b])}^2 + \|D^mf_n^\alpha\|_{L_2(\mathbb{R}\setminus[a,b])}^2 &= \int_b^\infty |f_n^\alpha(x)|^2\,dx + \int_b^\infty |D^mf_n^\alpha(x)|^2\,dx \\
      &\qquad + \int_{-a}^\infty |f_n^\alpha(-x)|^2\,dx + \int_{-a}^\infty |D^mf_n^\alpha(-x)|^2\,dx\\
      &\leq  \frac{4C}{\alpha^2 (|a|\wedge|b|)} \int_{-\infty}^\infty |x| h_n(x)\,dx.
    \end{aligned}
  \end{equation}

  Finally, combining Items~\ref{bounds:itemi}-~\ref{item:itemiii} with (\ref{eqn:ritzbound}) and taking expectations, we have
  \begin{equation}
    \mathbb{E}\|s_n^\alpha - f_n^\alpha\|^2 \leq C_1 \alpha^{-4}\gamma^{2(r-m)} + C_2 \alpha^{-4}\gamma^*(1+\gamma) + C_3 \alpha^{-4}(|a|\wedge|b|)^{-1} \mu_{\hat Y},
  \end{equation}
  Thus, if \(\gamma,\gamma^*,a\), and \(b\) are chosen as in the hypothesis of the Theorem, we have \(\mathbb{E}\|s_n^\alpha - f_n^\alpha\|^2 = o(r_n)\). Then the result follows from
  \begin{equation}
    \mathbb{E}\|s_n^\alpha - f\|^2 \leq 2\mathbb{E}\|f_n^\alpha - f\|^2 + 2\mathbb{E}\|s_n^\alpha - f_n^\alpha\|^2 \leq O(r_n) + o(r_n)
  \end{equation}
\end{defproof}

\pagebreak

\bibliography{ref.bib}
\makeatletter\@input{msaux.tex}\makeatother

%% file: proofs/prop_discontinuous.tex
\begin{proof}[Proof of Proposition~\ref{prop:discontinuous}]
  We will construct a sequence \(\phi_n \in \mathcal{R}(T)\) which is
a Cauchy sequence in \(L_2(\mathbb{R})\), but with the property that
for \(n \in \mathbb{N}\), we have \(\|T^{-1}(\phi_n - \phi_{n+1})\| =
1\). Once we have this sequence, we can finish the proof in the
following way. Fix \(M > 0\). Since \(\phi_n\) is Cauchy, we can
choose \(n \in \mathbb{N}\) large enough that \(\|\phi_n -
\phi_{n+1}\| < \frac{1}{M}\). Then \(u = \phi_n - \phi_{n+1}\)
satisfies
\begin{equation}
  \|T^{-1}(\phi_n - \phi_{n+1})\| > M \|\phi_n - \phi_{n+1}\|,
\end{equation}
as needed.

Now, if we can find such a sequence \(\phi_n\), we are finished. To that end, let \(\psi_n = n\mathds{1}_{[0,1/n]}\). It can be checked that \(\|\psi_n - \psi_{n+1}\| = 1\). Furthermore, the \(\psi_n\) constitute an ``approximate identity,'' so that by \citet[Theorem 8.14a]{folland_real_1999}, we have \(\|g*\psi_n - g\| \to 0\) as \(n \to \infty\). Now, let \(\phi_n = g*\psi_n = T\psi_n\). To see that \(\phi_n\) is Cauchy, apply the triangle inequality:
\begin{equation}
  \|\phi_n - \phi_m\| \leq \|\phi_n - g\| + \|\phi_m - g\| = \|g*\psi_n - g\| + \|g*\psi_m - g\|.
\end{equation}
For the other property, note that
\begin{equation}
  \|T^{-1}(\phi_n - \phi_{n+1})\| = \|T^{-1}(T\psi_n - T\psi_{n+1})\| = \|\psi_n - \psi_{n+1}\| = 1,
\end{equation}
finishing the proof.
\end{proof}


%% file: tikz_plots/approxinv.tex
\begin{tikzpicture}[yscale=0.3,xscale=1.0]
\pgfmathdeclarefunction{sinc}{1}{%
    \pgfmathparse{abs(#1)<0.01 ? int(1) : int(0)}%
    \ifnum\pgfmathresult>0 \pgfmathparse{1}\else\pgfmathparse{sin(#1 r)/#1}\fi%
    }
\draw[thick, lightgray, ->] (-3,0) -- (3.25, 0)
node[anchor=south west] {$\omega$};

\foreach \x in {-1, 1}
\draw[thick] (\x, 0.067) -- (\x, -0.067)
  node[anchor=north] {\(\x\)};

\draw[thick, lightgray, ->] (0, 0) -- (0, 10.5)
node[anchor=south west] {$y$};

\draw[thick] (-0.1, 5) -- (0.1, 5)
  node[anchor=west] {\(5\)};
\draw[thick] (-0.1, 10) -- (0.1, 10)
  node[anchor=west] {\(10\)};

\draw [black, thick, smooth, samples=100, domain=-2.2:2.2] plot(\x,{exp(\x*\x/2)});
\draw [black, dashed, smooth, samples=100, domain=-3:3] plot(\x,{exp(-\x*\x/2)/(exp(-\x*\x/2)*exp(-\x*\x/2) + 1*\x*\x*\x*\x});
\draw [black, dashed, smooth, samples=100, domain=-3:3] plot(\x,{exp(-\x*\x/2)/(exp(-\x*\x/2)*exp(-\x*\x/2) + 0.01*\x*\x*\x*\x});
\draw [black, dashed, smooth, samples=100, domain=-3:3] plot(\x,{exp(-\x*\x/2)/(exp(-\x*\x/2)*exp(-\x*\x/2) + 0.0001*\x*\x*\x*\x});
\end{tikzpicture}

%% file: proofs/theorem_representing.tex
\begin{proof}[Proof of Theorem \ref{theorem:representing}]
~\paragraph{\ref{theorem:representing:mult}:} By Theorem 3.1 of
\cite{locker_regularization_1980}, a function \(f^\alpha_n\) minimizes
the Tikhonov functional \(G^\alpha_n(f)\) if and only if
\(f^\alpha_n \in \mathcal{D}(L^*L)\) and \(f^\alpha_n\) satisfies the
Euler-Lagrange equation
\(\label{eqn:euler}
(T^*T + \alpha L^*L)f^\alpha_n = T^* h_n.
\)
By Fact \ref{fact:adjoints}, this corresponds to
\(
g \star g*f^\alpha_n + \alpha L^*Lf_n^\alpha = g \star h_n,
\)
where \(g\star u(t) = \int g(x-t) u(x)\,dx\), so taking Fourier transforms yields (see Fact \ref{fact:adjoints} for details)
\(
|\tilde g(\omega)|^2 \tilde f^\alpha_n(\omega) + \alpha \omega^{2m} \tilde f^\alpha_n(\omega) = \overline{\tilde g(\omega)} \tilde h_n(\omega),
\)
and re-arranging gives
\[
\tilde f^\alpha_n(\omega) = \frac{\overline{\tilde g(\omega)}}{|\tilde g(\omega)|^2 + \alpha \omega^{2m}} \tilde h_n(\omega) = \tilde \varphi_\alpha(\omega) \tilde h_n(\omega),
\]
as needed.
\paragraph{\ref{theorem:representing:conv} and \ref{theorem:representing:bound}:} 

It will be convenient to prove \ref{theorem:representing:bound} first. We prove the equivalent
inequality that for all \(\omega\), \(\sqrt{\alpha}|\tilde
\varphi_\alpha(\omega)| \leq C. \) We do this by demonstrating two
facts: first, that for all \(\omega\), \(\sqrt{\alpha}|\tilde
\varphi_\alpha(\omega)| < \frac12 |\omega|^{-m}\), so that if we can
bound \(\sqrt{\alpha}|\tilde \varphi_\alpha(\omega)|\) on a
neighborhood of zero, we are finished, since the bound decreases as
\(|\omega| \to \infty\). Second, we show that \(\sqrt{\alpha}|\tilde
\varphi_\alpha(\omega)| \leq \frac{\sqrt{M}}{|\tilde g(\omega)|}\), and that
on a neighborhood \(|\omega| \leq \varepsilon\) of zero, \(\tilde
g(\omega)\) is bounded away from zero: \(0 < c < |\tilde g(\omega)| \leq
1\), and take \(C = \max \{\sqrt{M}/c,\frac12\varepsilon^{-m}\}\).

For the first, apply the inequality \(x+y \geq 2\sqrt{xy}\) for
\(x,y>0\) to the denominator of \(\tilde \varphi_\alpha\):
\begin{equation}
    \sqrt{\alpha}|\tilde \varphi_\alpha(\omega)| = \sqrt{\alpha}\frac{|\tilde g(\omega)|}{|\tilde g(\omega)|^2 + \alpha \omega^{2m}} \leq \sqrt{\alpha}\frac{|\tilde g(\omega)|}{2\sqrt{\alpha |\tilde g(\omega)|^2 \omega^{2m}}} = \frac12 |\omega|^{-m}.
\end{equation}
For the second,
\begin{equation}
  \sqrt{\alpha}|\tilde \varphi_\alpha(\omega)| = \sqrt{\alpha}\frac{|\tilde g(\omega)|}{|\tilde g(\omega)|^2 + \alpha \omega^{2m}} \leq \sqrt{\alpha}\frac{|\tilde g(\omega)|}{|\tilde g(\omega)|^2} \leq \sqrt{M}|\tilde g(\omega)|^{-1},
\end{equation} where the first inequality is because \(\alpha
\omega^{2m} > 0\), and the second inequality is by the assumption that
\(\alpha < M\). Finally, to see that \(\tilde g\) is bounded away from
zero on a neighborhood of zero, recall that \(\tilde g\) is the
Fourier transform of a probability density \(g\). Thus \(\tilde g(0) =
1\), and \(\tilde g\) is continuous, proving \ref{theorem:representing:bound}.

Now, we will demonstrate that then \(\tilde \varphi_\alpha \in
L_2(\mathbb{C})\), so that the Fourier inversion in Equation
\eqref{eqn:varphi} is legitimate. By the arguments proving (iv), we
have also found a square-integrable function
\(
b(\omega) = \alpha^{-\frac12}(\mathds{1}_{|\omega| < \varepsilon}C + \frac12\mathds{1}_{|\omega| \geq \varepsilon}|\omega|^{-m}),
\) such that \(|b(\omega)| \geq |\tilde
\varphi_\alpha(\omega)|\). Thus, \(\int |\tilde
\varphi_\alpha(\omega)|^2\,d\omega \leq \int |b(\omega)|^2\,d\omega
<\infty,\) and \(\tilde \varphi_\alpha \in L_2(\mathbb{C})\).  Now,
(ii) follows from (i) and the well-known properties of the Fourier
transform.
\paragraph{\ref{theorem:representing:kernel} and \ref{theorem:representing:unitint}-\ref{theorem:representing:apriorisup}:} Deferred to Supplemental Proof~\ref{defproof:theorem:representing}.
\end{proof}


%% file: proofs/lem_randompart.tex
\begin{proof}[Proof of Lemma \ref{lem:randompart}]
By the Plancherel Theorem,
\begin{equation}\begin{aligned}
\|f_n^\alpha - f^\alpha\|^2 &= \frac{1}{2\pi}\|\tilde f^\alpha_n - \tilde f^\alpha\|^2\\
 &= \frac{1}{2\pi}\int | \varphi_\alpha(\omega) |^2 |\tilde h_n(\omega) - \tilde h(\omega) |^2\,d\omega\\
&\leq \sup_\omega |\tilde \varphi_\alpha(\omega)|^2 \|h_n - h\|^2 \leq C\|h_n - h\|^2/\alpha
\end{aligned}\end{equation}
where the second inequality is by Theorem~\ref{theorem:representing}\ref{theorem:representing:bound}. Taking expectations gives the result.
\end{proof}


%% file: proofs/cor_upper.tex
\begin{proof}[Proof of Corollary \ref{cor:upper}]
  Note that \((a+b)^2 \leq 2a^2 + 2b^2\), which can be seen by expanding \(0 \leq (a-b)^2\), adding \(a^2 + b^2\) to both sides, and re-arranging. Then the result follows from the triangle inequality and Lemma~\ref{lem:randompart}.
\end{proof}


%% file: proofs/theorem_consistency.tex
\begin{proof}[Proof of Theorem \ref{theorem:consistency}]
  This follows immediately from Corollary~\ref{cor:upper} and Proposition~\ref{prop:alphasmoothed}(iv).
\end{proof}


%% file: proofs/theorem_smoothrates.tex
\begin{proof}[Proof of Theorem \ref{theorem:smoothrates}]
Our task is to find the dependence of \(\|f^\alpha - f\|^2\) on
\(\alpha\). 
\begin{equation}\begin{aligned}
\|f^\alpha - f\|^2 &= \frac{1}{2\pi}\|\tilde f^\alpha - \tilde f\|^2\\
&= \frac{1}{2\pi} \int  \left | \frac{\alpha_n \omega^{2m}}{|\tilde g(\omega)|^2 + \alpha_n \omega^{2m}}\right |^2 ||\tilde g(\omega)|^2\omega^{-2m} \tilde \psi(\omega)|^2 \,d\omega,\\
&= \frac{\alpha_n^2}{2\pi} \int  \left | \frac{|\tilde g(\omega)|^2}{|\tilde g(\omega)|^2 + \alpha_n \omega^{2m}}\right |^2 |\tilde \psi(\omega)|^2 \,d\omega \leq \alpha_n^2 \|\psi\|^2,
\end{aligned}\end{equation}
which, combined with Corollary~\ref{cor:upper}, gives the bound. The upper bound is minimized by \(\alpha_n \propto \delta_n^\frac23\), in which case the upper bound becomes \(\mathbb{E}\|f_n - f\|^2 \leq C \delta_n^\frac43\), as needed.
\end{proof}


%% file: proofs/cor_smoothrates.tex
\begin{proof}[Proof of Corollary \ref{cor:smoothrates}]
If \(h_n\) is a kernel density estimate with optimal bandwidth, then
\(\delta_n^2 = \|h_n - h\|^2 = O(n^{-\frac45})\)
(\cite{wand_kernel_1994}, Section 2.5), and the result follows
immediately. Similarly, if \(h_n\) is a histogram with optimal binwidth, then by the same section,
\(\delta_n^2 = \|h_n - h\| = O(n^{-\frac23})\).
\end{proof}


%% file: proofs/lem_asgood.tex
\begin{proof}[Proof of Lemma \ref{lem:asgood}]
Add and subtract \(P_{\mathcal{C}_a}f\):
\begin{equation}\begin{aligned}
\|\mathring f_n^\alpha - f\| &= \|P_{\mathcal{C}_a}f_n^\alpha - P_{\mathcal{C}_a}f + P_{\mathcal{C}_a}f - f\|\\
&\leq \|P_{\mathcal{C}_a}f_n^\alpha - P_{\mathcal{C}_a}f\| + \|P_{\mathcal{C}_a}f - f\|
\leq \|f_n^\alpha - f\| + 2\mathbb{E}[|X|^\beta]a^{-\beta}.
\end{aligned}\end{equation}
A similar approach yields the exponential version.
\end{proof}


%% file: tikz_plots/infMISEs_all.tex
\begin{tikzpicture}[x=1pt,y=1pt]
\definecolor{fillColor}{RGB}{255,255,255}
\path[use as bounding box,fill=fillColor,fill opacity=0.00] (0,0) rectangle (469.75,180.67);
\begin{scope}
\path[clip] ( 30.89, 39.60) rectangle (129.32,148.20);
\definecolor{drawColor}{RGB}{0,0,0}

\path[draw=drawColor,line width= 0.4pt,line join=round,line cap=round] ( 34.53,130.47) --
	( 35.44,129.84) --
	( 36.36,129.19) --
	( 37.27,128.54) --
	( 38.18,127.88) --
	( 39.09,127.21) --
	( 40.00,126.54) --
	( 40.91,125.85) --
	( 41.82,125.17) --
	( 42.74,124.47) --
	( 43.65,123.77) --
	( 44.56,123.06) --
	( 45.47,122.35) --
	( 46.38,121.63) --
	( 47.29,120.91) --
	( 48.20,120.18) --
	( 49.12,119.45) --
	( 50.03,118.71) --
	( 50.94,117.97) --
	( 51.85,117.22) --
	( 52.76,116.47) --
	( 53.67,115.72) --
	( 54.58,114.96) --
	( 55.50,114.19) --
	( 56.41,113.43) --
	( 57.32,112.66) --
	( 58.23,111.89) --
	( 59.14,111.11) --
	( 60.05,110.33) --
	( 60.96,109.55) --
	( 61.88,108.77) --
	( 62.79,107.98) --
	( 63.70,107.19) --
	( 64.61,106.40) --
	( 65.52,105.60) --
	( 66.43,104.80) --
	( 67.34,104.00) --
	( 68.26,103.20) --
	( 69.17,102.40) --
	( 70.08,101.59) --
	( 70.99,100.78) --
	( 71.90, 99.97) --
	( 72.81, 99.16) --
	( 73.72, 98.35) --
	( 74.63, 97.54) --
	( 75.55, 96.72) --
	( 76.46, 95.90) --
	( 77.37, 95.08) --
	( 78.28, 94.26) --
	( 79.19, 93.44) --
	( 80.10, 92.61) --
	( 81.01, 91.79) --
	( 81.93, 90.96) --
	( 82.84, 90.13) --
	( 83.75, 89.31) --
	( 84.66, 88.48) --
	( 85.57, 87.64) --
	( 86.48, 86.81) --
	( 87.39, 85.98) --
	( 88.31, 85.14) --
	( 89.22, 84.31) --
	( 90.13, 83.47) --
	( 91.04, 82.64) --
	( 91.95, 81.80) --
	( 92.86, 80.96) --
	( 93.77, 80.12) --
	( 94.69, 79.28) --
	( 95.60, 78.44) --
	( 96.51, 77.60) --
	( 97.42, 76.75) --
	( 98.33, 75.91) --
	( 99.24, 75.07) --
	(100.15, 74.22) --
	(101.07, 73.38) --
	(101.98, 72.53) --
	(102.89, 71.68) --
	(103.80, 70.84) --
	(104.71, 69.99) --
	(105.62, 69.14) --
	(106.53, 68.29) --
	(107.45, 67.44) --
	(108.36, 66.59) --
	(109.27, 65.74) --
	(110.18, 64.89) --
	(111.09, 64.04) --
	(112.00, 63.19) --
	(112.91, 62.34) --
	(113.83, 61.48) --
	(114.74, 60.63) --
	(115.65, 59.78) --
	(116.56, 58.92) --
	(117.47, 58.07) --
	(118.38, 57.21) --
	(119.29, 56.36) --
	(120.20, 55.50) --
	(121.12, 54.65) --
	(122.03, 53.79) --
	(122.94, 52.93) --
	(123.85, 52.08) --
	(124.76, 51.22) --
	(125.67, 50.36);
\end{scope}
\begin{scope}
\path[clip] (  0.00,  0.00) rectangle (469.75,180.67);
\definecolor{drawColor}{RGB}{0,0,0}

\path[draw=drawColor,line width= 0.4pt,line join=round,line cap=round] ( 34.53, 39.60) -- (125.67, 39.60);

\path[draw=drawColor,line width= 0.4pt,line join=round,line cap=round] ( 34.53, 39.60) -- ( 34.53, 35.64);

\path[draw=drawColor,line width= 0.4pt,line join=round,line cap=round] ( 52.76, 39.60) -- ( 52.76, 35.64);

\path[draw=drawColor,line width= 0.4pt,line join=round,line cap=round] ( 70.99, 39.60) -- ( 70.99, 35.64);

\path[draw=drawColor,line width= 0.4pt,line join=round,line cap=round] ( 89.22, 39.60) -- ( 89.22, 35.64);

\path[draw=drawColor,line width= 0.4pt,line join=round,line cap=round] (107.45, 39.60) -- (107.45, 35.64);

\path[draw=drawColor,line width= 0.4pt,line join=round,line cap=round] (125.67, 39.60) -- (125.67, 35.64);

\node[text=drawColor,anchor=base,inner sep=0pt, outer sep=0pt, scale=  0.66] at ( 34.53, 25.34) {1};

\node[text=drawColor,anchor=base,inner sep=0pt, outer sep=0pt, scale=  0.66] at ( 52.76, 25.34) {2};

\node[text=drawColor,anchor=base,inner sep=0pt, outer sep=0pt, scale=  0.66] at ( 70.99, 25.34) {3};

\node[text=drawColor,anchor=base,inner sep=0pt, outer sep=0pt, scale=  0.66] at ( 89.22, 25.34) {4};

\node[text=drawColor,anchor=base,inner sep=0pt, outer sep=0pt, scale=  0.66] at (107.45, 25.34) {5};

\node[text=drawColor,anchor=base,inner sep=0pt, outer sep=0pt, scale=  0.66] at (125.67, 25.34) {6};

\path[draw=drawColor,line width= 0.4pt,line join=round,line cap=round] ( 30.89, 45.59) -- ( 30.89,144.18);

\path[draw=drawColor,line width= 0.4pt,line join=round,line cap=round] ( 30.89, 45.59) -- ( 26.93, 45.59);

\path[draw=drawColor,line width= 0.4pt,line join=round,line cap=round] ( 30.89, 65.31) -- ( 26.93, 65.31);

\path[draw=drawColor,line width= 0.4pt,line join=round,line cap=round] ( 30.89, 85.03) -- ( 26.93, 85.03);

\path[draw=drawColor,line width= 0.4pt,line join=round,line cap=round] ( 30.89,104.75) -- ( 26.93,104.75);

\path[draw=drawColor,line width= 0.4pt,line join=round,line cap=round] ( 30.89,124.46) -- ( 26.93,124.46);

\path[draw=drawColor,line width= 0.4pt,line join=round,line cap=round] ( 30.89,144.18) -- ( 26.93,144.18);

\node[text=drawColor,rotate= 90.00,anchor=base,inner sep=0pt, outer sep=0pt, scale=  0.66] at ( 21.38, 45.59) {-6};

\node[text=drawColor,rotate= 90.00,anchor=base,inner sep=0pt, outer sep=0pt, scale=  0.66] at ( 21.38, 65.31) {-5};

\node[text=drawColor,rotate= 90.00,anchor=base,inner sep=0pt, outer sep=0pt, scale=  0.66] at ( 21.38, 85.03) {-4};

\node[text=drawColor,rotate= 90.00,anchor=base,inner sep=0pt, outer sep=0pt, scale=  0.66] at ( 21.38,104.75) {-3};

\node[text=drawColor,rotate= 90.00,anchor=base,inner sep=0pt, outer sep=0pt, scale=  0.66] at ( 21.38,124.46) {-2};

\node[text=drawColor,rotate= 90.00,anchor=base,inner sep=0pt, outer sep=0pt, scale=  0.66] at ( 21.38,144.18) {-1};

\path[draw=drawColor,line width= 0.4pt,line join=round,line cap=round] ( 30.89, 39.60) --
	(129.32, 39.60) --
	(129.32,148.20) --
	( 30.89,148.20) --
	cycle;
\end{scope}
\begin{scope}
\path[clip] ( 26.14,  7.92) rectangle (129.32,180.67);
\definecolor{drawColor}{RGB}{0,0,0}

\node[text=drawColor,anchor=base,inner sep=0pt, outer sep=0pt, scale=  0.79] at ( 80.10,161.71) {\bfseries Normal};
\end{scope}
\begin{scope}
\path[clip] ( 30.89, 39.60) rectangle (129.32,148.20);
\definecolor{drawColor}{RGB}{223,83,107}

\path[draw=drawColor,line width= 0.8pt,dash pattern=on 4pt off 4pt ,line join=round,line cap=round] ( 34.53,132.93) --
	( 35.44,132.36) --
	( 36.36,131.78) --
	( 37.27,131.21) --
	( 38.18,130.62) --
	( 39.09,130.04) --
	( 40.00,129.45) --
	( 40.91,128.86) --
	( 41.82,128.27) --
	( 42.74,127.68) --
	( 43.65,127.08) --
	( 44.56,126.49) --
	( 45.47,125.89) --
	( 46.38,125.29) --
	( 47.29,124.69) --
	( 48.20,124.08) --
	( 49.12,123.48) --
	( 50.03,122.88) --
	( 50.94,122.28) --
	( 51.85,121.67) --
	( 52.76,121.07) --
	( 53.67,120.47) --
	( 54.58,119.87) --
	( 55.50,119.27) --
	( 56.41,118.67) --
	( 57.32,118.07) --
	( 58.23,117.47) --
	( 59.14,116.88) --
	( 60.05,116.29) --
	( 60.96,115.70) --
	( 61.88,115.11) --
	( 62.79,114.53) --
	( 63.70,113.94) --
	( 64.61,113.37) --
	( 65.52,112.79) --
	( 66.43,112.22) --
	( 67.34,111.65) --
	( 68.26,111.09) --
	( 69.17,110.53) --
	( 70.08,109.97) --
	( 70.99,109.42) --
	( 71.90,108.88) --
	( 72.81,108.34) --
	( 73.72,107.80) --
	( 74.63,107.27) --
	( 75.55,106.75) --
	( 76.46,106.23) --
	( 77.37,105.72) --
	( 78.28,105.21) --
	( 79.19,104.71) --
	( 80.10,104.22) --
	( 81.01,103.73) --
	( 81.93,103.25) --
	( 82.84,102.78) --
	( 83.75,102.31) --
	( 84.66,101.85) --
	( 85.57,101.39) --
	( 86.48,100.95) --
	( 87.39,100.51) --
	( 88.31,100.07) --
	( 89.22, 99.65) --
	( 90.13, 99.23) --
	( 91.04, 98.82) --
	( 91.95, 98.41) --
	( 92.86, 98.02) --
	( 93.77, 97.63) --
	( 94.69, 97.24) --
	( 95.60, 96.87) --
	( 96.51, 96.50) --
	( 97.42, 96.13) --
	( 98.33, 95.78) --
	( 99.24, 95.43) --
	(100.15, 95.08) --
	(101.07, 94.75) --
	(101.98, 94.42) --
	(102.89, 94.09) --
	(103.80, 93.77) --
	(104.71, 93.46) --
	(105.62, 93.16) --
	(106.53, 92.86) --
	(107.45, 92.56) --
	(108.36, 92.27) --
	(109.27, 91.99) --
	(110.18, 91.71) --
	(111.09, 91.44) --
	(112.00, 91.17) --
	(112.91, 90.90) --
	(113.83, 90.64) --
	(114.74, 90.39) --
	(115.65, 90.14) --
	(116.56, 89.90) --
	(117.47, 89.65) --
	(118.38, 89.42) --
	(119.29, 89.19) --
	(120.20, 88.96) --
	(121.12, 88.73) --
	(122.03, 88.51) --
	(122.94, 88.29) --
	(123.85, 88.08) --
	(124.76, 87.87) --
	(125.67, 87.66);
\definecolor{drawColor}{RGB}{34,151,230}

\path[draw=drawColor,line width= 0.8pt,dash pattern=on 4pt off 4pt ,line join=round,line cap=round] ( 34.53,135.44) --
	( 35.44,134.96) --
	( 36.36,134.48) --
	( 37.27,133.99) --
	( 38.18,133.51) --
	( 39.09,133.02) --
	( 40.00,132.54) --
	( 40.91,132.06) --
	( 41.82,131.57) --
	( 42.74,131.09) --
	( 43.65,130.61) --
	( 44.56,130.14) --
	( 45.47,129.66) --
	( 46.38,129.19) --
	( 47.29,128.72) --
	( 48.20,128.25) --
	( 49.12,127.79) --
	( 50.03,127.33) --
	( 50.94,126.87) --
	( 51.85,126.42) --
	( 52.76,125.97) --
	( 53.67,125.52) --
	( 54.58,125.08) --
	( 55.50,124.65) --
	( 56.41,124.22) --
	( 57.32,123.79) --
	( 58.23,123.37) --
	( 59.14,122.95) --
	( 60.05,122.54) --
	( 60.96,122.14) --
	( 61.88,121.74) --
	( 62.79,121.35) --
	( 63.70,120.96) --
	( 64.61,120.58) --
	( 65.52,120.20) --
	( 66.43,119.83) --
	( 67.34,119.47) --
	( 68.26,119.11) --
	( 69.17,118.75) --
	( 70.08,118.41) --
	( 70.99,118.07) --
	( 71.90,117.73) --
	( 72.81,117.40) --
	( 73.72,117.08) --
	( 74.63,116.76) --
	( 75.55,116.44) --
	( 76.46,116.14) --
	( 77.37,115.83) --
	( 78.28,115.54) --
	( 79.19,115.25) --
	( 80.10,114.96) --
	( 81.01,114.68) --
	( 81.93,114.40) --
	( 82.84,114.13) --
	( 83.75,113.86) --
	( 84.66,113.60) --
	( 85.57,113.34) --
	( 86.48,113.09) --
	( 87.39,112.84) --
	( 88.31,112.60) --
	( 89.22,112.36) --
	( 90.13,112.12) --
	( 91.04,111.89) --
	( 91.95,111.66) --
	( 92.86,111.44) --
	( 93.77,111.22) --
	( 94.69,111.00) --
	( 95.60,110.79) --
	( 96.51,110.58) --
	( 97.42,110.37) --
	( 98.33,110.16) --
	( 99.24,109.96) --
	(100.15,109.77) --
	(101.07,109.57) --
	(101.98,109.38) --
	(102.89,109.19) --
	(103.80,109.01) --
	(104.71,108.82) --
	(105.62,108.64) --
	(106.53,108.46) --
	(107.45,108.29) --
	(108.36,108.11) --
	(109.27,107.94) --
	(110.18,107.77) --
	(111.09,107.61) --
	(112.00,107.44) --
	(112.91,107.28) --
	(113.83,107.12) --
	(114.74,106.96) --
	(115.65,106.81) --
	(116.56,106.65) --
	(117.47,106.50) --
	(118.38,106.35) --
	(119.29,106.20) --
	(120.20,106.05) --
	(121.12,105.91) --
	(122.03,105.76) --
	(122.94,105.62) --
	(123.85,105.48) --
	(124.76,105.34) --
	(125.67,105.20);
\definecolor{drawColor}{RGB}{223,83,107}

\path[draw=drawColor,line width= 0.8pt,line join=round,line cap=round] ( 34.53,132.60) --
	( 36.39,131.31) --
	( 38.25,129.99) --
	( 40.11,128.64) --
	( 41.97,127.27) --
	( 43.83,125.87) --
	( 45.69,124.45) --
	( 47.55,123.01) --
	( 49.41,121.55) --
	( 51.27,120.08) --
	( 53.13,118.59) --
	( 54.99,117.08) --
	( 56.85,115.56) --
	( 58.71,114.03) --
	( 60.57,112.48) --
	( 62.43,110.93) --
	( 64.29,109.36) --
	( 66.15,107.79) --
	( 68.01,106.21) --
	( 69.87,104.62) --
	( 71.73,103.02) --
	( 73.59,101.42) --
	( 75.45, 99.81) --
	( 77.31, 98.20) --
	( 79.17, 96.58) --
	( 81.03, 94.96) --
	( 82.89, 93.33) --
	( 84.75, 91.71) --
	( 86.61, 90.07) --
	( 88.47, 88.44) --
	( 90.33, 86.80) --
	( 92.19, 85.16) --
	( 94.05, 83.52) --
	( 95.91, 81.88) --
	( 97.77, 80.23) --
	( 99.63, 78.59) --
	(101.49, 76.94) --
	(103.35, 75.30) --
	(105.21, 73.65) --
	(107.07, 72.00) --
	(108.93, 70.36) --
	(110.79, 68.71) --
	(112.65, 67.06) --
	(114.51, 65.42) --
	(116.37, 63.77) --
	(118.23, 62.13) --
	(120.09, 60.48) --
	(121.95, 58.84) --
	(123.81, 57.20) --
	(125.67, 55.56);
\definecolor{drawColor}{RGB}{34,151,230}

\path[draw=drawColor,line width= 0.8pt,line join=round,line cap=round] ( 34.53,135.04) --
	( 36.39,133.89) --
	( 38.25,132.72) --
	( 40.11,131.52) --
	( 41.97,130.29) --
	( 43.83,129.05) --
	( 45.69,127.79) --
	( 47.55,126.51) --
	( 49.41,125.21) --
	( 51.27,123.90) --
	( 53.13,122.57) --
	( 54.99,121.24) --
	( 56.85,119.89) --
	( 58.71,118.53) --
	( 60.57,117.16) --
	( 62.43,115.79) --
	( 64.29,114.41) --
	( 66.15,113.02) --
	( 68.01,111.63) --
	( 69.87,110.23) --
	( 71.73,108.82) --
	( 73.59,107.41) --
	( 75.45,106.00) --
	( 77.31,104.59) --
	( 79.17,103.17) --
	( 81.03,101.75) --
	( 82.89,100.32) --
	( 84.75, 98.90) --
	( 86.61, 97.47) --
	( 88.47, 96.05) --
	( 90.33, 94.62) --
	( 92.19, 93.19) --
	( 94.05, 91.76) --
	( 95.91, 90.33) --
	( 97.77, 88.89) --
	( 99.63, 87.46) --
	(101.49, 86.03) --
	(103.35, 84.60) --
	(105.21, 83.17) --
	(107.07, 81.73) --
	(108.93, 80.30) --
	(110.79, 78.87) --
	(112.65, 77.44) --
	(114.51, 76.01) --
	(116.37, 74.58) --
	(118.23, 73.15) --
	(120.09, 71.72) --
	(121.95, 70.29) --
	(123.81, 68.86) --
	(125.67, 67.44);
\end{scope}
\begin{scope}
\path[clip] (134.07, 39.60) rectangle (232.50,148.20);
\definecolor{drawColor}{RGB}{0,0,0}

\path[draw=drawColor,line width= 0.4pt,line join=round,line cap=round] (137.72,142.59) --
	(138.63,142.10) --
	(139.54,141.59) --
	(140.45,141.05) --
	(141.36,140.49) --
	(142.27,139.92) --
	(143.18,139.32) --
	(144.10,138.71) --
	(145.01,138.09) --
	(145.92,137.45) --
	(146.83,136.79) --
	(147.74,136.13) --
	(148.65,135.45) --
	(149.56,134.76) --
	(150.48,134.07) --
	(151.39,133.36) --
	(152.30,132.64) --
	(153.21,131.92) --
	(154.12,131.19) --
	(155.03,130.46) --
	(155.94,129.71) --
	(156.86,128.96) --
	(157.77,128.21) --
	(158.68,127.45) --
	(159.59,126.68) --
	(160.50,125.91) --
	(161.41,125.14) --
	(162.32,124.36) --
	(163.24,123.58) --
	(164.15,122.79) --
	(165.06,122.00) --
	(165.97,121.21) --
	(166.88,120.41) --
	(167.79,119.61) --
	(168.70,118.80) --
	(169.62,118.00) --
	(170.53,117.19) --
	(171.44,116.38) --
	(172.35,115.56) --
	(173.26,114.75) --
	(174.17,113.93) --
	(175.08,113.11) --
	(175.99,112.29) --
	(176.91,111.46) --
	(177.82,110.64) --
	(178.73,109.81) --
	(179.64,108.98) --
	(180.55,108.15) --
	(181.46,107.31) --
	(182.37,106.48) --
	(183.29,105.64) --
	(184.20,104.80) --
	(185.11,103.96) --
	(186.02,103.12) --
	(186.93,102.28) --
	(187.84,101.44) --
	(188.75,100.60) --
	(189.67, 99.75) --
	(190.58, 98.91) --
	(191.49, 98.06) --
	(192.40, 97.21) --
	(193.31, 96.36) --
	(194.22, 95.51) --
	(195.13, 94.66) --
	(196.05, 93.81) --
	(196.96, 92.96) --
	(197.87, 92.10) --
	(198.78, 91.25) --
	(199.69, 90.40) --
	(200.60, 89.54) --
	(201.51, 88.68) --
	(202.43, 87.83) --
	(203.34, 86.97) --
	(204.25, 86.11) --
	(205.16, 85.16) --
	(206.07, 84.30) --
	(206.98, 83.44) --
	(207.89, 82.58) --
	(208.81, 81.72) --
	(209.72, 80.86) --
	(210.63, 80.00) --
	(211.54, 79.14) --
	(212.45, 78.28) --
	(213.36, 77.42) --
	(214.27, 76.56) --
	(215.18, 75.70) --
	(216.10, 74.84) --
	(217.01, 73.97) --
	(217.92, 73.11) --
	(218.83, 72.25) --
	(219.74, 71.38) --
	(220.65, 70.52) --
	(221.56, 69.65) --
	(222.48, 68.79) --
	(223.39, 67.92) --
	(224.30, 67.06) --
	(225.21, 66.19) --
	(226.12, 65.33) --
	(227.03, 64.46) --
	(227.94, 63.59) --
	(228.86, 62.73);
\end{scope}
\begin{scope}
\path[clip] (  0.00,  0.00) rectangle (469.75,180.67);
\definecolor{drawColor}{RGB}{0,0,0}

\path[draw=drawColor,line width= 0.4pt,line join=round,line cap=round] (137.72, 39.60) -- (228.86, 39.60);

\path[draw=drawColor,line width= 0.4pt,line join=round,line cap=round] (137.72, 39.60) -- (137.72, 35.64);

\path[draw=drawColor,line width= 0.4pt,line join=round,line cap=round] (155.94, 39.60) -- (155.94, 35.64);

\path[draw=drawColor,line width= 0.4pt,line join=round,line cap=round] (174.17, 39.60) -- (174.17, 35.64);

\path[draw=drawColor,line width= 0.4pt,line join=round,line cap=round] (192.40, 39.60) -- (192.40, 35.64);

\path[draw=drawColor,line width= 0.4pt,line join=round,line cap=round] (210.63, 39.60) -- (210.63, 35.64);

\path[draw=drawColor,line width= 0.4pt,line join=round,line cap=round] (228.86, 39.60) -- (228.86, 35.64);

\node[text=drawColor,anchor=base,inner sep=0pt, outer sep=0pt, scale=  0.66] at (137.72, 25.34) {1};

\node[text=drawColor,anchor=base,inner sep=0pt, outer sep=0pt, scale=  0.66] at (155.94, 25.34) {2};

\node[text=drawColor,anchor=base,inner sep=0pt, outer sep=0pt, scale=  0.66] at (174.17, 25.34) {3};

\node[text=drawColor,anchor=base,inner sep=0pt, outer sep=0pt, scale=  0.66] at (192.40, 25.34) {4};

\node[text=drawColor,anchor=base,inner sep=0pt, outer sep=0pt, scale=  0.66] at (210.63, 25.34) {5};

\node[text=drawColor,anchor=base,inner sep=0pt, outer sep=0pt, scale=  0.66] at (228.86, 25.34) {6};

\path[draw=drawColor,line width= 0.4pt,line join=round,line cap=round] (134.07, 39.60) --
	(232.50, 39.60) --
	(232.50,148.20) --
	(134.07,148.20) --
	cycle;
\end{scope}
\begin{scope}
\path[clip] (129.32,  7.92) rectangle (232.50,180.67);
\definecolor{drawColor}{RGB}{0,0,0}

\node[text=drawColor,anchor=base,inner sep=0pt, outer sep=0pt, scale=  0.79] at (183.29,161.71) {\bfseries Normal Mixture};
\end{scope}
\begin{scope}
\path[clip] (134.07, 39.60) rectangle (232.50,148.20);
\definecolor{drawColor}{RGB}{223,83,107}

\path[draw=drawColor,line width= 0.8pt,dash pattern=on 4pt off 4pt ,line join=round,line cap=round] (137.72,144.51) --
	(138.63,144.15) --
	(139.54,143.79) --
	(140.45,143.42) --
	(141.36,143.04) --
	(142.27,142.67) --
	(143.18,142.29) --
	(144.10,141.91) --
	(145.01,141.53) --
	(145.92,141.14) --
	(146.83,140.76) --
	(147.74,140.37) --
	(148.65,139.99) --
	(149.56,139.61) --
	(150.48,139.23) --
	(151.39,138.85) --
	(152.30,138.48) --
	(153.21,138.11) --
	(154.12,137.74) --
	(155.03,137.38) --
	(155.94,137.02) --
	(156.86,136.67) --
	(157.77,136.32) --
	(158.68,135.97) --
	(159.59,135.63) --
	(160.50,135.30) --
	(161.41,134.97) --
	(162.32,134.64) --
	(163.24,134.33) --
	(164.15,134.01) --
	(165.06,133.71) --
	(165.97,133.40) --
	(166.88,133.11) --
	(167.79,132.82) --
	(168.70,132.53) --
	(169.62,132.25) --
	(170.53,131.98) --
	(171.44,131.71) --
	(172.35,131.44) --
	(173.26,131.18) --
	(174.17,130.93) --
	(175.08,130.68) --
	(175.99,130.44) --
	(176.91,130.20) --
	(177.82,129.96) --
	(178.73,129.73) --
	(179.64,129.51) --
	(180.55,129.29) --
	(181.46,129.07) --
	(182.37,128.86) --
	(183.29,128.65) --
	(184.20,128.44) --
	(185.11,128.24) --
	(186.02,128.04) --
	(186.93,127.85) --
	(187.84,127.66) --
	(188.75,127.47) --
	(189.67,127.29) --
	(190.58,127.10) --
	(191.49,126.93) --
	(192.40,126.75) --
	(193.31,126.58) --
	(194.22,126.41) --
	(195.13,126.24) --
	(196.05,126.08) --
	(196.96,125.92) --
	(197.87,125.76) --
	(198.78,125.60) --
	(199.69,125.44) --
	(200.60,125.29) --
	(201.51,125.14) --
	(202.43,124.99) --
	(203.34,124.85) --
	(204.25,124.70) --
	(205.16,124.56) --
	(206.07,124.42) --
	(206.98,124.28) --
	(207.89,124.14) --
	(208.81,124.01) --
	(209.72,123.87) --
	(210.63,123.74) --
	(211.54,123.61) --
	(212.45,123.48) --
	(213.36,123.35) --
	(214.27,123.23) --
	(215.18,123.10) --
	(216.10,122.98) --
	(217.01,122.86) --
	(217.92,122.74) --
	(218.83,122.62) --
	(219.74,122.50) --
	(220.65,122.38) --
	(221.56,122.26) --
	(222.48,122.15) --
	(223.39,122.03) --
	(224.30,121.92) --
	(225.21,121.81) --
	(226.12,121.70) --
	(227.03,121.59) --
	(227.94,121.48) --
	(228.86,121.37);
\definecolor{drawColor}{RGB}{34,151,230}

\path[draw=drawColor,line width= 0.8pt,dash pattern=on 4pt off 4pt ,line join=round,line cap=round] (137.72,146.51) --
	(138.63,146.27) --
	(139.54,146.04) --
	(140.45,145.80) --
	(141.36,145.57) --
	(142.27,145.34) --
	(143.18,145.12) --
	(144.10,144.89) --
	(145.01,144.67) --
	(145.92,144.46) --
	(146.83,144.25) --
	(147.74,144.04) --
	(148.65,143.83) --
	(149.56,143.63) --
	(150.48,143.43) --
	(151.39,143.23) --
	(152.30,143.04) --
	(153.21,142.85) --
	(154.12,142.67) --
	(155.03,142.49) --
	(155.94,142.31) --
	(156.86,142.14) --
	(157.77,141.97) --
	(158.68,141.80) --
	(159.59,141.64) --
	(160.50,141.48) --
	(161.41,141.33) --
	(162.32,141.17) --
	(163.24,141.02) --
	(164.15,140.88) --
	(165.06,140.74) --
	(165.97,140.60) --
	(166.88,140.46) --
	(167.79,140.32) --
	(168.70,140.19) --
	(169.62,140.06) --
	(170.53,139.94) --
	(171.44,139.81) --
	(172.35,139.69) --
	(173.26,139.57) --
	(174.17,139.46) --
	(175.08,139.34) --
	(175.99,139.23) --
	(176.91,139.12) --
	(177.82,139.01) --
	(178.73,138.90) --
	(179.64,138.80) --
	(180.55,138.70) --
	(181.46,138.59) --
	(182.37,138.49) --
	(183.29,138.40) --
	(184.20,138.30) --
	(185.11,138.20) --
	(186.02,138.11) --
	(186.93,138.02) --
	(187.84,137.93) --
	(188.75,137.83) --
	(189.67,137.75) --
	(190.58,137.66) --
	(191.49,137.57) --
	(192.40,137.49) --
	(193.31,137.40) --
	(194.22,137.32) --
	(195.13,137.24) --
	(196.05,137.15) --
	(196.96,137.07) --
	(197.87,136.99) --
	(198.78,136.91) --
	(199.69,136.84) --
	(200.60,136.76) --
	(201.51,136.68) --
	(202.43,136.61) --
	(203.34,136.53) --
	(204.25,136.46) --
	(205.16,136.38) --
	(206.07,136.31) --
	(206.98,136.24) --
	(207.89,136.17) --
	(208.81,136.10) --
	(209.72,136.03) --
	(210.63,135.96) --
	(211.54,135.89) --
	(212.45,135.82) --
	(213.36,135.75) --
	(214.27,135.68) --
	(215.18,135.62) --
	(216.10,135.55) --
	(217.01,135.48) --
	(217.92,135.42) --
	(218.83,135.35) --
	(219.74,135.29) --
	(220.65,135.23) --
	(221.56,135.16) --
	(222.48,135.10) --
	(223.39,135.04) --
	(224.30,134.97) --
	(225.21,134.91) --
	(226.12,134.85) --
	(227.03,134.79) --
	(227.94,134.73) --
	(228.86,134.67);
\definecolor{drawColor}{RGB}{223,83,107}

\path[draw=drawColor,line width= 0.8pt,line join=round,line cap=round] (137.72,144.65) --
	(139.58,143.96) --
	(141.44,143.24) --
	(143.30,142.50) --
	(145.16,141.72) --
	(147.02,140.93) --
	(148.88,140.11) --
	(150.74,139.26) --
	(152.60,138.40) --
	(154.46,137.53) --
	(156.32,136.64) --
	(158.18,135.75) --
	(160.04,134.84) --
	(161.90,133.94) --
	(163.76,133.03) --
	(165.62,132.11) --
	(167.48,131.20) --
	(169.34,130.28) --
	(171.20,129.37) --
	(173.06,128.45) --
	(174.92,127.53) --
	(176.78,126.61) --
	(178.64,125.69) --
	(180.50,124.77) --
	(182.36,123.85) --
	(184.22,122.93) --
	(186.08,122.00) --
	(187.94,121.07) --
	(189.80,120.14) --
	(191.66,119.20) --
	(193.52,118.25) --
	(195.38,117.30) --
	(197.24,116.33) --
	(199.10,115.36) --
	(200.96,114.37) --
	(202.82,113.37) --
	(204.68,112.36) --
	(206.54,111.33) --
	(208.40,110.29) --
	(210.26,109.22) --
	(212.12,108.14) --
	(213.98,107.04) --
	(215.84,105.92) --
	(217.70,104.78) --
	(219.56,103.62) --
	(221.42,102.43) --
	(223.28,101.23) --
	(225.14,100.01) --
	(227.00, 98.78) --
	(228.86, 97.52);
\definecolor{drawColor}{RGB}{34,151,230}

\path[draw=drawColor,line width= 0.8pt,line join=round,line cap=round] (137.72,146.52) --
	(139.58,146.02) --
	(141.44,145.52) --
	(143.30,145.03) --
	(145.16,144.54) --
	(147.02,144.07) --
	(148.88,143.60) --
	(150.74,143.14) --
	(152.60,142.68) --
	(154.46,142.23) --
	(156.32,141.79) --
	(158.18,141.36) --
	(160.04,140.93) --
	(161.90,140.50) --
	(163.76,140.09) --
	(165.62,139.67) --
	(167.48,139.27) --
	(169.34,138.87) --
	(171.20,138.47) --
	(173.06,138.08) --
	(174.92,137.69) --
	(176.78,137.31) --
	(178.64,136.93) --
	(180.50,136.56) --
	(182.36,136.19) --
	(184.22,135.82) --
	(186.08,135.46) --
	(187.94,135.10) --
	(189.80,134.74) --
	(191.66,134.39) --
	(193.52,134.04) --
	(195.38,133.69) --
	(197.24,133.35) --
	(199.10,133.00) --
	(200.96,132.66) --
	(202.82,132.32) --
	(204.68,131.99) --
	(206.54,131.65) --
	(208.40,131.32) --
	(210.26,130.98) --
	(212.12,130.65) --
	(213.98,130.32) --
	(215.84,130.00) --
	(217.70,129.67) --
	(219.56,129.34) --
	(221.42,129.02) --
	(223.28,128.69) --
	(225.14,128.37) --
	(227.00,128.05) --
	(228.86,127.72);
\end{scope}
\begin{scope}
\path[clip] (237.25, 39.60) rectangle (335.68,148.20);
\definecolor{drawColor}{RGB}{0,0,0}

\path[draw=drawColor,line width= 0.4pt,line join=round,line cap=round] (240.90,127.54) --
	(241.81,126.94) --
	(242.72,126.34) --
	(243.63,125.72) --
	(244.54,125.10) --
	(245.46,124.47) --
	(246.37,123.84) --
	(247.28,123.19) --
	(248.19,122.54) --
	(249.10,121.89) --
	(250.01,121.22) --
	(250.92,120.56) --
	(251.84,119.88) --
	(252.75,119.20) --
	(253.66,118.51) --
	(254.57,117.82) --
	(255.48,117.13) --
	(256.39,116.42) --
	(257.30,115.72) --
	(258.22,115.01) --
	(259.13,114.29) --
	(260.04,113.57) --
	(260.95,112.85) --
	(261.86,112.12) --
	(262.77,111.39) --
	(263.68,110.65) --
	(264.60,109.92) --
	(265.51,109.17) --
	(266.42,108.43) --
	(267.33,107.68) --
	(268.24,106.93) --
	(269.15,106.17) --
	(270.06,105.42) --
	(270.98,104.66) --
	(271.89,103.89) --
	(272.80,103.13) --
	(273.71,102.36) --
	(274.62,101.59) --
	(275.53,100.81) --
	(276.44,100.04) --
	(277.35, 99.26) --
	(278.27, 98.48) --
	(279.18, 97.70) --
	(280.09, 96.92) --
	(281.00, 96.13) --
	(281.91, 95.34) --
	(282.82, 94.55) --
	(283.73, 93.76) --
	(284.65, 92.97) --
	(285.56, 92.18) --
	(286.47, 91.38) --
	(287.38, 90.58) --
	(288.29, 89.78) --
	(289.20, 88.98) --
	(290.11, 88.18) --
	(291.03, 87.38) --
	(291.94, 86.57) --
	(292.85, 85.77) --
	(293.76, 84.96) --
	(294.67, 84.15) --
	(295.58, 83.34) --
	(296.49, 82.53) --
	(297.41, 81.72) --
	(298.32, 80.91) --
	(299.23, 80.09) --
	(300.14, 79.28) --
	(301.05, 78.46) --
	(301.96, 77.64) --
	(302.87, 76.83) --
	(303.79, 76.01) --
	(304.70, 75.19) --
	(305.61, 74.37) --
	(306.52, 73.54) --
	(307.43, 72.72) --
	(308.34, 71.90) --
	(309.25, 71.07) --
	(310.17, 70.25) --
	(311.08, 69.42) --
	(311.99, 68.60) --
	(312.90, 67.77) --
	(313.81, 66.94) --
	(314.72, 66.11) --
	(315.63, 65.29) --
	(316.54, 64.46) --
	(317.46, 63.62) --
	(318.37, 62.79) --
	(319.28, 61.96) --
	(320.19, 61.13) --
	(321.10, 60.30) --
	(322.01, 59.46) --
	(322.92, 58.63) --
	(323.84, 57.80) --
	(324.75, 56.96) --
	(325.66, 56.12) --
	(326.57, 55.29) --
	(327.48, 54.45) --
	(328.39, 53.61) --
	(329.30, 52.78) --
	(330.22, 51.94) --
	(331.13, 51.10) --
	(332.04, 50.26);
\end{scope}
\begin{scope}
\path[clip] (  0.00,  0.00) rectangle (469.75,180.67);
\definecolor{drawColor}{RGB}{0,0,0}

\path[draw=drawColor,line width= 0.4pt,line join=round,line cap=round] (240.90, 39.60) -- (332.04, 39.60);

\path[draw=drawColor,line width= 0.4pt,line join=round,line cap=round] (240.90, 39.60) -- (240.90, 35.64);

\path[draw=drawColor,line width= 0.4pt,line join=round,line cap=round] (259.13, 39.60) -- (259.13, 35.64);

\path[draw=drawColor,line width= 0.4pt,line join=round,line cap=round] (277.35, 39.60) -- (277.35, 35.64);

\path[draw=drawColor,line width= 0.4pt,line join=round,line cap=round] (295.58, 39.60) -- (295.58, 35.64);

\path[draw=drawColor,line width= 0.4pt,line join=round,line cap=round] (313.81, 39.60) -- (313.81, 35.64);

\path[draw=drawColor,line width= 0.4pt,line join=round,line cap=round] (332.04, 39.60) -- (332.04, 35.64);

\node[text=drawColor,anchor=base,inner sep=0pt, outer sep=0pt, scale=  0.66] at (240.90, 25.34) {1};

\node[text=drawColor,anchor=base,inner sep=0pt, outer sep=0pt, scale=  0.66] at (259.13, 25.34) {2};

\node[text=drawColor,anchor=base,inner sep=0pt, outer sep=0pt, scale=  0.66] at (277.35, 25.34) {3};

\node[text=drawColor,anchor=base,inner sep=0pt, outer sep=0pt, scale=  0.66] at (295.58, 25.34) {4};

\node[text=drawColor,anchor=base,inner sep=0pt, outer sep=0pt, scale=  0.66] at (313.81, 25.34) {5};

\node[text=drawColor,anchor=base,inner sep=0pt, outer sep=0pt, scale=  0.66] at (332.04, 25.34) {6};

\path[draw=drawColor,line width= 0.4pt,line join=round,line cap=round] (237.25, 39.60) --
	(335.68, 39.60) --
	(335.68,148.20) --
	(237.25,148.20) --
	cycle;
\end{scope}
\begin{scope}
\path[clip] (232.50,  7.92) rectangle (335.68,180.67);
\definecolor{drawColor}{RGB}{0,0,0}

\node[text=drawColor,anchor=base,inner sep=0pt, outer sep=0pt, scale=  0.79] at (286.47,161.71) {\bfseries Gamma};
\end{scope}
\begin{scope}
\path[clip] (237.25, 39.60) rectangle (335.68,148.20);
\definecolor{drawColor}{RGB}{223,83,107}

\path[draw=drawColor,line width= 0.8pt,dash pattern=on 4pt off 4pt ,line join=round,line cap=round] (240.90,129.82) --
	(241.81,129.28) --
	(242.72,128.75) --
	(243.63,128.21) --
	(244.54,127.67) --
	(245.46,127.13) --
	(246.37,126.58) --
	(247.28,126.03) --
	(248.19,125.48) --
	(249.10,124.93) --
	(250.01,124.38) --
	(250.92,123.82) --
	(251.84,123.27) --
	(252.75,122.71) --
	(253.66,122.16) --
	(254.57,121.61) --
	(255.48,121.05) --
	(256.39,120.50) --
	(257.30,119.95) --
	(258.22,119.40) --
	(259.13,118.85) --
	(260.04,118.30) --
	(260.95,117.75) --
	(261.86,117.21) --
	(262.77,116.67) --
	(263.68,116.14) --
	(264.60,115.60) --
	(265.51,115.08) --
	(266.42,114.55) --
	(267.33,114.03) --
	(268.24,113.51) --
	(269.15,113.00) --
	(270.06,112.49) --
	(270.98,111.99) --
	(271.89,111.49) --
	(272.80,111.00) --
	(273.71,110.51) --
	(274.62,110.03) --
	(275.53,109.56) --
	(276.44,109.09) --
	(277.35,108.62) --
	(278.27,108.17) --
	(279.18,107.72) --
	(280.09,107.27) --
	(281.00,106.83) --
	(281.91,106.40) --
	(282.82,105.98) --
	(283.73,105.56) --
	(284.65,105.15) --
	(285.56,104.75) --
	(286.47,104.35) --
	(287.38,103.96) --
	(288.29,103.58) --
	(289.20,103.20) --
	(290.11,102.83) --
	(291.03,102.47) --
	(291.94,102.11) --
	(292.85,101.77) --
	(293.76,101.42) --
	(294.67,101.09) --
	(295.58,100.76) --
	(296.49,100.43) --
	(297.41,100.11) --
	(298.32, 99.80) --
	(299.23, 99.50) --
	(300.14, 99.20) --
	(301.05, 98.90) --
	(301.96, 98.61) --
	(302.87, 98.33) --
	(303.79, 98.05) --
	(304.70, 97.78) --
	(305.61, 97.51) --
	(306.52, 97.25) --
	(307.43, 96.99) --
	(308.34, 96.74) --
	(309.25, 96.49) --
	(310.17, 96.25) --
	(311.08, 96.01) --
	(311.99, 95.78) --
	(312.90, 95.55) --
	(313.81, 95.32) --
	(314.72, 95.10) --
	(315.63, 94.88) --
	(316.54, 94.66) --
	(317.46, 94.45) --
	(318.37, 94.24) --
	(319.28, 94.04) --
	(320.19, 93.84) --
	(321.10, 93.64) --
	(322.01, 93.44) --
	(322.92, 93.25) --
	(323.84, 93.06) --
	(324.75, 92.88) --
	(325.66, 92.69) --
	(326.57, 92.51) --
	(327.48, 92.33) --
	(328.39, 92.16) --
	(329.30, 91.98) --
	(330.22, 91.81) --
	(331.13, 91.65) --
	(332.04, 91.48);
\definecolor{drawColor}{RGB}{34,151,230}

\path[draw=drawColor,line width= 0.8pt,dash pattern=on 4pt off 4pt ,line join=round,line cap=round] (240.90,132.52) --
	(241.81,132.10) --
	(242.72,131.68) --
	(243.63,131.27) --
	(244.54,130.85) --
	(245.46,130.43) --
	(246.37,130.02) --
	(247.28,129.60) --
	(248.19,129.19) --
	(249.10,128.78) --
	(250.01,128.38) --
	(250.92,127.97) --
	(251.84,127.57) --
	(252.75,127.18) --
	(253.66,126.78) --
	(254.57,126.40) --
	(255.48,126.01) --
	(256.39,125.63) --
	(257.30,125.26) --
	(258.22,124.88) --
	(259.13,124.52) --
	(260.04,124.16) --
	(260.95,123.80) --
	(261.86,123.45) --
	(262.77,123.10) --
	(263.68,122.76) --
	(264.60,122.43) --
	(265.51,122.10) --
	(266.42,121.77) --
	(267.33,121.45) --
	(268.24,121.14) --
	(269.15,120.83) --
	(270.06,120.53) --
	(270.98,120.23) --
	(271.89,119.94) --
	(272.80,119.65) --
	(273.71,119.37) --
	(274.62,119.09) --
	(275.53,118.82) --
	(276.44,118.55) --
	(277.35,118.29) --
	(278.27,118.03) --
	(279.18,117.78) --
	(280.09,117.53) --
	(281.00,117.29) --
	(281.91,117.05) --
	(282.82,116.82) --
	(283.73,116.59) --
	(284.65,116.36) --
	(285.56,116.14) --
	(286.47,115.92) --
	(287.38,115.71) --
	(288.29,115.50) --
	(289.20,115.29) --
	(290.11,115.09) --
	(291.03,114.89) --
	(291.94,114.69) --
	(292.85,114.50) --
	(293.76,114.31) --
	(294.67,114.12) --
	(295.58,113.94) --
	(296.49,113.76) --
	(297.41,113.58) --
	(298.32,113.40) --
	(299.23,113.23) --
	(300.14,113.06) --
	(301.05,112.89) --
	(301.96,112.73) --
	(302.87,112.56) --
	(303.79,112.40) --
	(304.70,112.25) --
	(305.61,112.09) --
	(306.52,111.94) --
	(307.43,111.78) --
	(308.34,111.63) --
	(309.25,111.49) --
	(310.17,111.34) --
	(311.08,111.20) --
	(311.99,111.06) --
	(312.90,110.91) --
	(313.81,110.78) --
	(314.72,110.64) --
	(315.63,110.50) --
	(316.54,110.37) --
	(317.46,110.24) --
	(318.37,110.11) --
	(319.28,109.98) --
	(320.19,109.85) --
	(321.10,109.72) --
	(322.01,109.60) --
	(322.92,109.47) --
	(323.84,109.35) --
	(324.75,109.23) --
	(325.66,109.11) --
	(326.57,108.99) --
	(327.48,108.87) --
	(328.39,108.76) --
	(329.30,108.64) --
	(330.22,108.53) --
	(331.13,108.41) --
	(332.04,108.30);
\definecolor{drawColor}{RGB}{223,83,107}

\path[draw=drawColor,line width= 0.8pt,line join=round,line cap=round] (240.90,129.74) --
	(242.76,128.58) --
	(244.62,127.40) --
	(246.48,126.20) --
	(248.34,124.97) --
	(250.20,123.73) --
	(252.06,122.47) --
	(253.92,121.19) --
	(255.78,119.90) --
	(257.64,118.60) --
	(259.50,117.29) --
	(261.36,115.97) --
	(263.22,114.65) --
	(265.08,113.32) --
	(266.94,111.98) --
	(268.80,110.65) --
	(270.66,109.31) --
	(272.52,107.97) --
	(274.38,106.64) --
	(276.24,105.31) --
	(278.10,103.98) --
	(279.96,102.65) --
	(281.82,101.33) --
	(283.68,100.02) --
	(285.54, 98.72) --
	(287.40, 97.42) --
	(289.26, 96.13) --
	(291.12, 94.85) --
	(292.98, 93.59) --
	(294.84, 92.33) --
	(296.70, 91.09) --
	(298.56, 89.86) --
	(300.42, 88.64) --
	(302.28, 87.44) --
	(304.14, 86.25) --
	(306.00, 85.08) --
	(307.86, 83.92) --
	(309.72, 82.77) --
	(311.58, 81.65) --
	(313.44, 80.53) --
	(315.30, 79.44) --
	(317.16, 78.36) --
	(319.02, 77.30) --
	(320.88, 76.25) --
	(322.74, 75.22) --
	(324.60, 74.21) --
	(326.46, 73.22) --
	(328.32, 72.24) --
	(330.18, 71.27) --
	(332.04, 70.33);
\definecolor{drawColor}{RGB}{34,151,230}

\path[draw=drawColor,line width= 0.8pt,line join=round,line cap=round] (240.90,132.39) --
	(242.76,131.43) --
	(244.62,130.45) --
	(246.48,129.47) --
	(248.34,128.47) --
	(250.20,127.46) --
	(252.06,126.44) --
	(253.92,125.42) --
	(255.78,124.40) --
	(257.64,123.38) --
	(259.50,122.35) --
	(261.36,121.33) --
	(263.22,120.31) --
	(265.08,119.30) --
	(266.94,118.29) --
	(268.80,117.29) --
	(270.66,116.30) --
	(272.52,115.31) --
	(274.38,114.34) --
	(276.24,113.38) --
	(278.10,112.42) --
	(279.96,111.48) --
	(281.82,110.55) --
	(283.68,109.64) --
	(285.54,108.73) --
	(287.40,107.84) --
	(289.26,106.97) --
	(291.12,106.10) --
	(292.98,105.25) --
	(294.84,104.41) --
	(296.70,103.59) --
	(298.56,102.78) --
	(300.42,101.98) --
	(302.28,101.19) --
	(304.14,100.42) --
	(306.00, 99.66) --
	(307.86, 98.91) --
	(309.72, 98.18) --
	(311.58, 97.46) --
	(313.44, 96.74) --
	(315.30, 96.04) --
	(317.16, 95.35) --
	(319.02, 94.68) --
	(320.88, 94.01) --
	(322.74, 93.35) --
	(324.60, 92.71) --
	(326.46, 92.07) --
	(328.32, 91.44) --
	(330.18, 90.82) --
	(332.04, 90.22);
\end{scope}
\begin{scope}
\path[clip] (340.44, 39.60) rectangle (438.87,148.20);
\definecolor{drawColor}{RGB}{0,0,0}

\path[draw=drawColor,line width= 0.4pt,line join=round,line cap=round] (344.08,121.24) --
	(344.99,120.88) --
	(345.90,120.53) --
	(346.82,120.19) --
	(347.73,119.84) --
	(348.64,119.49) --
	(349.55,119.14) --
	(350.46,118.79) --
	(351.37,118.42) --
	(352.28,118.03) --
	(353.20,117.60) --
	(354.11,117.10) --
	(355.02,116.55) --
	(355.93,115.96) --
	(356.84,115.34) --
	(357.75,114.70) --
	(358.66,114.04) --
	(359.58,113.36) --
	(360.49,112.68) --
	(361.40,111.98) --
	(362.31,111.27) --
	(363.22,110.55) --
	(364.13,109.83) --
	(365.04,109.10) --
	(365.96,108.36) --
	(366.87,107.61) --
	(367.78,106.86) --
	(368.69,106.11) --
	(369.60,105.35) --
	(370.51,104.59) --
	(371.42,103.82) --
	(372.34,103.05) --
	(373.25,102.28) --
	(374.16,101.50) --
	(375.07,100.72) --
	(375.98, 99.94) --
	(376.89, 99.15) --
	(377.80, 98.36) --
	(378.71, 97.57) --
	(379.63, 96.78) --
	(380.54, 95.98) --
	(381.45, 95.19) --
	(382.36, 94.39) --
	(383.27, 93.58) --
	(384.18, 92.78) --
	(385.09, 91.98) --
	(386.01, 91.17) --
	(386.92, 90.36) --
	(387.83, 89.55) --
	(388.74, 88.74) --
	(389.65, 87.92) --
	(390.56, 87.11) --
	(391.47, 86.29) --
	(392.39, 85.48) --
	(393.30, 84.66) --
	(394.21, 83.84) --
	(395.12, 83.01) --
	(396.03, 82.19) --
	(396.94, 81.37) --
	(397.85, 80.54) --
	(398.77, 79.72) --
	(399.68, 78.89) --
	(400.59, 78.06) --
	(401.50, 77.23) --
	(402.41, 76.40) --
	(403.32, 75.57) --
	(404.23, 74.74) --
	(405.15, 73.91) --
	(406.06, 73.07) --
	(406.97, 72.24) --
	(407.88, 71.40) --
	(408.79, 70.57) --
	(409.70, 69.73) --
	(410.61, 68.89) --
	(411.53, 68.05) --
	(412.44, 67.21) --
	(413.35, 66.37) --
	(414.26, 65.53) --
	(415.17, 64.69) --
	(416.08, 63.85) --
	(416.99, 63.00) --
	(417.90, 62.16) --
	(418.82, 61.31) --
	(419.73, 60.47) --
	(420.64, 59.62) --
	(421.55, 58.78) --
	(422.46, 57.93) --
	(423.37, 57.08) --
	(424.28, 56.24) --
	(425.20, 55.39) --
	(426.11, 54.54) --
	(427.02, 53.69) --
	(427.93, 52.84) --
	(428.84, 51.99) --
	(429.75, 51.14) --
	(430.66, 50.29) --
	(431.58, 49.44) --
	(432.49, 48.58) --
	(433.40, 47.73) --
	(434.31, 46.88) --
	(435.22, 46.02);
\end{scope}
\begin{scope}
\path[clip] (  0.00,  0.00) rectangle (469.75,180.67);
\definecolor{drawColor}{RGB}{0,0,0}

\path[draw=drawColor,line width= 0.4pt,line join=round,line cap=round] (344.08, 39.60) -- (435.22, 39.60);

\path[draw=drawColor,line width= 0.4pt,line join=round,line cap=round] (344.08, 39.60) -- (344.08, 35.64);

\path[draw=drawColor,line width= 0.4pt,line join=round,line cap=round] (362.31, 39.60) -- (362.31, 35.64);

\path[draw=drawColor,line width= 0.4pt,line join=round,line cap=round] (380.54, 39.60) -- (380.54, 35.64);

\path[draw=drawColor,line width= 0.4pt,line join=round,line cap=round] (398.77, 39.60) -- (398.77, 35.64);

\path[draw=drawColor,line width= 0.4pt,line join=round,line cap=round] (416.99, 39.60) -- (416.99, 35.64);

\path[draw=drawColor,line width= 0.4pt,line join=round,line cap=round] (435.22, 39.60) -- (435.22, 35.64);

\node[text=drawColor,anchor=base,inner sep=0pt, outer sep=0pt, scale=  0.66] at (344.08, 25.34) {1};

\node[text=drawColor,anchor=base,inner sep=0pt, outer sep=0pt, scale=  0.66] at (362.31, 25.34) {2};

\node[text=drawColor,anchor=base,inner sep=0pt, outer sep=0pt, scale=  0.66] at (380.54, 25.34) {3};

\node[text=drawColor,anchor=base,inner sep=0pt, outer sep=0pt, scale=  0.66] at (398.77, 25.34) {4};

\node[text=drawColor,anchor=base,inner sep=0pt, outer sep=0pt, scale=  0.66] at (416.99, 25.34) {5};

\node[text=drawColor,anchor=base,inner sep=0pt, outer sep=0pt, scale=  0.66] at (435.22, 25.34) {6};

\path[draw=drawColor,line width= 0.4pt,line join=round,line cap=round] (340.44, 39.60) --
	(438.87, 39.60) --
	(438.87,148.20) --
	(340.44,148.20) --
	cycle;
\end{scope}
\begin{scope}
\path[clip] (335.68,  7.92) rectangle (438.87,180.67);
\definecolor{drawColor}{RGB}{0,0,0}

\node[text=drawColor,anchor=base,inner sep=0pt, outer sep=0pt, scale=  0.79] at (389.65,161.71) {\bfseries Gamma Mixture};
\end{scope}
\begin{scope}
\path[clip] (340.44, 39.60) rectangle (438.87,148.20);
\definecolor{drawColor}{RGB}{223,83,107}

\path[draw=drawColor,line width= 0.8pt,dash pattern=on 4pt off 4pt ,line join=round,line cap=round] (344.08,122.90) --
	(344.99,122.56) --
	(345.90,122.23) --
	(346.82,121.90) --
	(347.73,121.57) --
	(348.64,121.24) --
	(349.55,120.91) --
	(350.46,120.57) --
	(351.37,120.22) --
	(352.28,119.86) --
	(353.20,119.49) --
	(354.11,119.10) --
	(355.02,118.71) --
	(355.93,118.30) --
	(356.84,117.88) --
	(357.75,117.46) --
	(358.66,117.03) --
	(359.58,116.60) --
	(360.49,116.17) --
	(361.40,115.74) --
	(362.31,115.30) --
	(363.22,114.87) --
	(364.13,114.43) --
	(365.04,114.00) --
	(365.96,113.57) --
	(366.87,113.15) --
	(367.78,112.73) --
	(368.69,112.31) --
	(369.60,111.89) --
	(370.51,111.48) --
	(371.42,111.08) --
	(372.34,110.68) --
	(373.25,110.28) --
	(374.16,109.90) --
	(375.07,109.51) --
	(375.98,109.14) --
	(376.89,108.77) --
	(377.80,108.40) --
	(378.71,108.04) --
	(379.63,107.69) --
	(380.54,107.34) --
	(381.45,107.00) --
	(382.36,106.67) --
	(383.27,106.34) --
	(384.18,106.02) --
	(385.09,105.71) --
	(386.01,105.40) --
	(386.92,105.09) --
	(387.83,104.80) --
	(388.74,104.51) --
	(389.65,104.22) --
	(390.56,103.94) --
	(391.47,103.67) --
	(392.39,103.40) --
	(393.30,103.14) --
	(394.21,102.88) --
	(395.12,102.62) --
	(396.03,102.38) --
	(396.94,102.13) --
	(397.85,101.89) --
	(398.77,101.66) --
	(399.68,101.43) --
	(400.59,101.21) --
	(401.50,100.99) --
	(402.41,100.77) --
	(403.32,100.56) --
	(404.23,100.35) --
	(405.15,100.14) --
	(406.06, 99.94) --
	(406.97, 99.74) --
	(407.88, 99.55) --
	(408.79, 99.36) --
	(409.70, 99.17) --
	(410.61, 98.98) --
	(411.53, 98.80) --
	(412.44, 98.62) --
	(413.35, 98.45) --
	(414.26, 98.27) --
	(415.17, 98.10) --
	(416.08, 97.94) --
	(416.99, 97.77) --
	(417.90, 97.61) --
	(418.82, 97.45) --
	(419.73, 97.29) --
	(420.64, 97.13) --
	(421.55, 96.98) --
	(422.46, 96.83) --
	(423.37, 96.68) --
	(424.28, 96.53) --
	(425.20, 96.38) --
	(426.11, 96.24) --
	(427.02, 96.10) --
	(427.93, 95.95) --
	(428.84, 95.82) --
	(429.75, 95.68) --
	(430.66, 95.54) --
	(431.58, 95.41) --
	(432.49, 95.28) --
	(433.40, 95.15) --
	(434.31, 95.02) --
	(435.22, 94.89);
\definecolor{drawColor}{RGB}{34,151,230}

\path[draw=drawColor,line width= 0.8pt,dash pattern=on 4pt off 4pt ,line join=round,line cap=round] (344.08,124.34) --
	(344.99,124.03) --
	(345.90,123.72) --
	(346.82,123.42) --
	(347.73,123.12) --
	(348.64,122.83) --
	(349.55,122.55) --
	(350.46,122.27) --
	(351.37,122.00) --
	(352.28,121.74) --
	(353.20,121.47) --
	(354.11,121.22) --
	(355.02,120.97) --
	(355.93,120.72) --
	(356.84,120.48) --
	(357.75,120.25) --
	(358.66,120.01) --
	(359.58,119.78) --
	(360.49,119.56) --
	(361.40,119.34) --
	(362.31,119.12) --
	(363.22,118.90) --
	(364.13,118.69) --
	(365.04,118.48) --
	(365.96,118.27) --
	(366.87,118.07) --
	(367.78,117.87) --
	(368.69,117.68) --
	(369.60,117.48) --
	(370.51,117.29) --
	(371.42,117.10) --
	(372.34,116.92) --
	(373.25,116.74) --
	(374.16,116.56) --
	(375.07,116.39) --
	(375.98,116.22) --
	(376.89,116.05) --
	(377.80,115.88) --
	(378.71,115.72) --
	(379.63,115.56) --
	(380.54,115.40) --
	(381.45,115.25) --
	(382.36,115.10) --
	(383.27,114.95) --
	(384.18,114.80) --
	(385.09,114.66) --
	(386.01,114.52) --
	(386.92,114.38) --
	(387.83,114.24) --
	(388.74,114.11) --
	(389.65,113.98) --
	(390.56,113.85) --
	(391.47,113.72) --
	(392.39,113.60) --
	(393.30,113.47) --
	(394.21,113.35) --
	(395.12,113.23) --
	(396.03,113.11) --
	(396.94,113.00) --
	(397.85,112.88) --
	(398.77,112.77) --
	(399.68,112.66) --
	(400.59,112.55) --
	(401.50,112.44) --
	(402.41,112.33) --
	(403.32,112.23) --
	(404.23,112.12) --
	(405.15,112.02) --
	(406.06,111.92) --
	(406.97,111.81) --
	(407.88,111.71) --
	(408.79,111.62) --
	(409.70,111.52) --
	(410.61,111.42) --
	(411.53,111.33) --
	(412.44,111.23) --
	(413.35,111.14) --
	(414.26,111.05) --
	(415.17,110.96) --
	(416.08,110.86) --
	(416.99,110.78) --
	(417.90,110.69) --
	(418.82,110.60) --
	(419.73,110.51) --
	(420.64,110.42) --
	(421.55,110.34) --
	(422.46,110.25) --
	(423.37,110.17) --
	(424.28,110.09) --
	(425.20,110.00) --
	(426.11,109.92) --
	(427.02,109.84) --
	(427.93,109.76) --
	(428.84,109.68) --
	(429.75,109.60) --
	(430.66,109.52) --
	(431.58,109.44) --
	(432.49,109.36) --
	(433.40,109.29) --
	(434.31,109.21) --
	(435.22,109.13);
\definecolor{drawColor}{RGB}{223,83,107}

\path[draw=drawColor,line width= 0.8pt,line join=round,line cap=round] (344.08,122.55) --
	(345.94,121.82) --
	(347.80,121.11) --
	(349.66,120.44) --
	(351.52,119.80) --
	(353.38,119.19) --
	(355.24,118.60) --
	(357.10,118.03) --
	(358.96,117.41) --
	(360.82,116.63) --
	(362.68,115.71) --
	(364.54,114.72) --
	(366.40,113.68) --
	(368.26,112.61) --
	(370.12,111.52) --
	(371.98,110.42) --
	(373.84,109.31) --
	(375.70,108.21) --
	(377.56,107.10) --
	(379.42,105.99) --
	(381.28,104.89) --
	(383.14,103.79) --
	(385.00,102.71) --
	(386.86,101.63) --
	(388.72,100.56) --
	(390.58, 99.51) --
	(392.44, 98.46) --
	(394.30, 97.43) --
	(396.16, 96.42) --
	(398.02, 95.41) --
	(399.88, 94.42) --
	(401.74, 93.45) --
	(403.60, 92.49) --
	(405.46, 91.54) --
	(407.32, 90.61) --
	(409.18, 89.69) --
	(411.04, 88.78) --
	(412.90, 87.89) --
	(414.76, 87.02) --
	(416.62, 86.15) --
	(418.48, 85.30) --
	(420.34, 84.47) --
	(422.20, 83.64) --
	(424.06, 82.83) --
	(425.92, 82.03) --
	(427.78, 81.25) --
	(429.64, 80.47) --
	(431.50, 79.71) --
	(433.36, 78.96) --
	(435.22, 78.21);
\definecolor{drawColor}{RGB}{34,151,230}

\path[draw=drawColor,line width= 0.8pt,line join=round,line cap=round] (344.08,123.97) --
	(345.94,123.23) --
	(347.80,122.53) --
	(349.66,121.86) --
	(351.52,121.22) --
	(353.38,120.62) --
	(355.24,120.06) --
	(357.10,119.54) --
	(358.96,119.07) --
	(360.82,118.64) --
	(362.68,118.24) --
	(364.54,117.88) --
	(366.40,117.56) --
	(368.26,117.26) --
	(370.12,116.99) --
	(371.98,116.72) --
	(373.84,116.44) --
	(375.70,116.09) --
	(377.56,115.68) --
	(379.42,115.24) --
	(381.28,114.77) --
	(383.14,114.28) --
	(385.00,113.79) --
	(386.86,113.28) --
	(388.72,112.76) --
	(390.58,112.24) --
	(392.44,111.72) --
	(394.30,111.20) --
	(396.16,110.68) --
	(398.02,110.16) --
	(399.88,109.64) --
	(401.74,109.12) --
	(403.60,108.60) --
	(405.46,108.09) --
	(407.32,107.58) --
	(409.18,107.08) --
	(411.04,106.58) --
	(412.90,106.08) --
	(414.76,105.58) --
	(416.62,105.09) --
	(418.48,104.60) --
	(420.34,104.11) --
	(422.20,103.64) --
	(424.06,103.17) --
	(425.92,102.70) --
	(427.78,102.24) --
	(429.64,101.78) --
	(431.50,101.32) --
	(433.36,100.86) --
	(435.22,100.41);
\end{scope}
\begin{scope}
\path[clip] (  0.00,  0.00) rectangle (469.75,180.67);
\definecolor{drawColor}{RGB}{0,0,0}

\node[text=drawColor,anchor=base,inner sep=0pt, outer sep=0pt, scale=  1.00] at (232.50,  5.54) {\(\log_{10} n\)};

\node[text=drawColor,rotate= 90.00,anchor=base,inner sep=0pt, outer sep=0pt, scale=  1.00] at ( 10.30, 94.30) {\(\inf_\alpha \log_{10}\) MISE\((f_n^\alpha)\)};
\end{scope}
\end{tikzpicture}

%% file: tables/MISE_min_n.tex
\begin{tabular}{ lrrrr }
    \multicolumn{1}{c}{}&
    \multicolumn{2}{c}{Sample size \(n = 100\)}&
    \multicolumn{2}{c}{\(n = 1,000\)}\\
    \multicolumn{1}{l}{\(p\)}&
    \multicolumn{1}{c}{SPDE}&
    \multicolumn{1}{c}{DKE}&
    \multicolumn{1}{c}{SPDE}&
    \multicolumn{1}{c}{DKE}\\
    \multicolumn{5}{l}{(i) Standard normal density}\\
    10\% & 146 (102) & 243 (156) & 1,525 (1,001) & 7,386 (2,931)\\
    30\% & 303 (204) & 1,761 (788) & 4,170 (2,239) & \(> 10^6\) (\(> 10^6\))\\
    50\% & 1,221 (747) & 924,510 (103,089) & 34,945 (15,566) & \(> 10^6\) (\(> 10^6\))\\
    \multicolumn{5}{l}{(ii) Normal mixture density}\\
    10\% & 629 (556) & 1,798 (1,415) & 32,715 (22,395) & \(> 10^6\) (\(> 10^6\))\\
    30\% & 236,587 (167,254) & \(> 10^6\) (\(> 10^6\)) & \(> 10^6\) (\(> 10^6\)) & \(> 10^6\) (\(> 10^6\))\\
    50\% & \(> 10^6\) (\(> 10^6\)) & \(> 10^6\) (\(> 10^6\)) & \(> 10^6\) (\(> 10^6\)) & \(> 10^6\) (\(> 10^6\))\\
    \multicolumn{5}{l}{(iii) Gamma(4) density}\\
    10\% & 179 (140) & 266 (197) & 2,548 (1,721) & 17,342 (7,863)\\
    30\% & 695 (499) & 8,016 (3,620) & 42,254 (21,561) & \(> 10^6\) (\(> 10^6\))\\
    50\% & 9,451 (5,551) & \(> 10^6\) (\(> 10^6\)) & \(> 10^6\) (\(> 10^6\)) & \(> 10^6\) (\(> 10^6\))\\
    \multicolumn{5}{l}{(iv) Gamma mixture density}\\
    10\% & 284 (270) & 300 (282) & 7,963 (6,020) & 388,770 (151,942)\\
    30\% & 5,521 (4,992) & 53,740 (41,039) & \(> 10^6\) (\(> 10^6\)) & \(> 10^6\) (\(> 10^6\))\\
    50\% & \(> 10^6\) (\(> 10^6\)) & \(> 10^6\) (\(> 10^6\)) & \(> 10^6\) (\(> 10^6\)) & \(> 10^6\) (\(> 10^6\))\\
\end{tabular}

%% file: sections/supporting.tex
\begin{fact}\label{fact:injective}
\(T\) is injective as long as the Fourier transform of \(g\) is a.e. non-vanishing.
\end{fact}
\begin{proof}[Proof of Fact~\ref{fact:injective}]
Suppose \(u,v \in L_2(\mathbb{R})\), and \(\|Tu - Tv\| = \|g*(u-v)\| =
0\). Then by the Plancherel Theorem and the fact that Fourier
transforms reduce convolution to multiplication, \(\|\tilde g
(\widetilde{u-v})\| = 0\). Since \(\hat g\) is a.e. non-vanishing, and
by another application of the Plancherel Theorem, it follows that
\(\|u-v\| = 0\), as needed.
\end{proof}

\begin{fact}\label{fact:adjoints}
  With \(Lv = v^{(m)}\), we have that \(L\) is self-adjoint if \(m\) is even, and skew-adjoint if \(m\) is odd:
  \begin{equation}
    L^* = (-1)^mL.
  \end{equation}
  We also have that the Fourier transform reduces \(L^*L\) to multiplication by \(\omega^{2m}\):
  \begin{equation}
    \mathcal{F}[L^*Lv](\omega) = \omega^{2m}\tilde v(\omega)
  \end{equation}
  Finally, we have that with \(Tv = g*v\), the adjoint of \(T\) is cross-correlation with \(g\), defined by \(T^*v(x) = g \star v(x) = \int g(x-t) u(x)\,dx\), and \(\mathcal{F}[T^*v](\omega) = \overline{\tilde g(\omega)} \tilde v(\omega)\).
\end{fact}
\begin{proof}[Proof of Fact \ref{fact:adjoints}]
  For the first, for \(u,v \in \mathcal{D}(L)\),
  \begin{equation}
    \begin{aligned}
      \langle u, Lv \rangle &= \int u(x) v^{(m)}(x)\,dx\\
      &= (2\pi)^{-1} \int \tilde u(\omega) \overline{(i\omega)^m \tilde v(\omega)} \,d\omega\\
      &= (2\pi)^{-1} \int (-1)^m (i\omega)^m \tilde u(\omega) \overline{\tilde v(\omega)} \,d\omega\\
      &= \int (-1)^m u^{(m)}(x)v(x)\,dx\\
      &= \langle (-1)^mLu,v \rangle.
    \end{aligned}
  \end{equation}
  For the second, we have, for \(v \in \mathcal{D}(L^*L)\),
  \begin{equation}
    \begin{aligned}
      L^*Lv(x) &= (-1)^m v^{(2m)}(x),
    \end{aligned}
  \end{equation}
  so that
  \begin{equation}
    \begin{aligned}
      \mathcal{F}[L^*Lv](\omega) &= (-1)^m (i\omega)^{2m} \tilde v(\omega)\\
      &= (-1)^mi^{2m}\omega^{2m} \tilde v(\omega)\\
      &= \omega^{2m} \tilde v(\omega),
    \end{aligned}
  \end{equation}
  since \((-1)^mi^{2m} = 1\) for integer \(m\).

  Finally, changing the order of integration, we get
  \begin{equation}\begin{aligned}
      \langle u, Tv \rangle &= \langle u, g*v \rangle\\
      &= \int u(x) \int v(t) g(x-t)\,dt\,dx\\
      &= \int \left [ \int g(x-t)u(x)\,dx\right]v(t) \,dt\\
      &=\langle g \star u,v \rangle,
  \end{aligned}\end{equation}
so that \(T^*u = g \star u\). Since cross-correlation with \(g(x)\) is convolution with \(g(-x)\), we have that \(\mathcal{F}[g\star v](\omega) = \overline{\tilde g(\omega)} \tilde v(\omega)\).
\end{proof}

\begin{fact}\label{fact:apprxMISE}
  If \(\tilde f_n^\alpha(\omega) = \tilde \varphi_\alpha(\omega) \tilde P_n(\omega)\), then
  \begin{equation}
    \mathbb{E} \|f_n^\alpha- f\|^2 = \frac1{2\pi} \left [ \int | \tilde \varphi_\alpha(\omega) \tilde g(\omega) - 1 |^2 | \tilde f(\omega) |^2\,dt + \frac1n \int |\tilde \varphi_\alpha(\omega)|^2(1 - |\tilde g(\omega)\tilde f(\omega)|^2) d\omega \right ]
  \end{equation}
\end{fact}
\begin{proof}
  The mean integrated squared error (MISE) or risk of \(f_n^\alpha\) is given by
  \[
    \mathbb{E} \int (f_n^\alpha(x) - f(x))^2\,dx = \int \Bias(f_n^\alpha(x))^2 + \Var(f_n^\alpha(x))\,dx,
  \]
  where \(\Bias(f_n^\alpha(x)) = \mathbb{E}[f_n^\alpha(x)] - f(x)\). The expectation and variance here are of course with respect to the distribution of the \(Y_j\). It is not hard to see that \(\mathbb{E}[f_n^\alpha(x)] = \mathbb{E}[\varphi_\alpha(x-Y_1)] = \varphi_\alpha*g*f(x)\). Then the Plancherel Theorem yields that
  \[
    \int \Bias(f_n^\alpha(x))^2\,dx = \int \left ( \varphi_\alpha*g*f(x) - f(x) \right )^2\,dx = \frac1{2\pi} \int | \tilde \varphi_\alpha(\omega) \tilde g(\omega) - 1 |^2 | \tilde f(\omega) |^2\,d\omega.
  \]

  For the variance, note that \(\Var(f_n^\alpha(x)) = \Var \left ( \frac1n \sum_{j=1}^n \varphi_\alpha(x-Y_j) \right ) = \frac1n \Var(\varphi_\alpha(x-Y))\). Now,
  \begin{equation}
    \begin{aligned}
     \int \Var(\varphi_\alpha(x-Y))\,dx &= \int \mathbb{E}[\varphi_\alpha(x-Y)^2] - \mathbb{E}[\varphi_\alpha(x-Y)]^2\,dx\\
    &= \int \int \varphi_\alpha(x-y)^2 g*f(y)\,dy\,dx - \int \varphi_\alpha*g*f(x)^2\,dx\\
    &= \int \varphi_\alpha(x)^2 \,dx - \int \varphi_\alpha*g*f(x)^2\,dx\\
    &= \frac1{2\pi}\int |\tilde \varphi_\alpha(\omega)|^2 \,dt - \frac1{2\pi}\int |\tilde \varphi_\alpha(\omega) \tilde g(\omega) \tilde f(\omega)|^2\,d\omega. 
    \end{aligned}
  \end{equation}
  Combining these and re-arranging gives the result.
\end{proof}

\begin{fact}\label{fact:lambertW}
Consider the map \(x \mapsto xe^x\) from \(\mathbb{R}^+ \to
\mathbb{R}^+\). This map is bijective, and its inverse is the
principal branch of the Lambert W function restricted to
\(\mathbb{R}^+\). We will denote the principal branch of the Lambert W
function by \(W(\cdot)\), so that \(W(xe^x) = x\).  The following
facts about the Lambert W function will be useful, all from
\cite{corless_lambert_1996}.
\begin{enumerate}[label=(\roman*)]
\item\label{fact:lambertW:logW} \(\log W(x) = \log(x) - W(x)\)
\item\label{fact:lambertW:eW} \(e^{W(x)} = xW(x)^{-1}\) for \(x>0\)
\item\label{fact:lambertW:Wxlogx} \(W(x\log(x)) = \log(x)\) for \(x > \frac{1}{e}\)
\item\label{fact:lambertW:asympt} \(W(x) \sim \log(x)\) as \(x \to \infty\)
\end{enumerate}
\end{fact}


%% file: proofs/prop_alphasmoothed.tex
\begin{proof}[Proof of Proposition~\ref{prop:alphasmoothed}]
  Properties (i) and (ii) follow from the same argument as in Theorem~\ref{theorem:representing}. For (iii), observe that by the definition of \(f^\alpha\), we have that
  \begin{equation}
    G(f^\alpha;h,\alpha) \leq G(f;h,\alpha).
  \end{equation}
  But \(G(f;h,\alpha) = \|g*f - h\|^2 + \alpha \|D^mf\|^2 = \alpha\|D^mf\|^2\), so that
  \begin{equation}
    \|g*f^\alpha - h\|^2 + \alpha \|D^mf^\alpha\|^2 \leq \alpha \|D^mf\|^2,
  \end{equation}
  and the result follows from the non-negativity of \(\|g*f^\alpha - h\|^2\).

  For (iv), the Plancherel Theorem yields that
  \begin{equation}\begin{aligned}
      \| f^{\alpha} - f\|^2 &= \frac{1}{2\pi} \int  \left | \tilde \varphi_\alpha(\omega) \tilde g(\omega) \tilde f(\omega) - \tilde f(\omega) \right |^2\,d\omega\\
      &= \frac{1}{2\pi} \int  \left | \frac{\alpha \omega^{2m}}{|\tilde g(\omega)|^2 + \alpha \omega^{2m}} \right |^2|\tilde f(\omega)|^2\,d\omega
    \end{aligned}\end{equation}
where the first equality is from the Plancherel Theorem.  The first
term in the integral is positive and less than one, so the integrand
is bounded above by \(|\tilde f(\omega)|^2\), which is integrable by
the assumption that \(f \in L_2(\mathbb{R})\). Now let \(A = \{\omega
| \tilde g(\omega) \neq 0\}\). We have assumed that \(A^c\) has
measure zero, and for each \(\omega
\in A\), \(\frac{\alpha \omega^{2m}}{|\tilde g(\omega)|^2 + \alpha
\omega^{2m}} \to 0\). Thus by the dominated convergence theorem,
\begin{equation}\begin{aligned}
    \lim_{\alpha \to 0} \|f^\alpha - f\|^2 = \lim_{\alpha \to 0} \frac{1}{2\pi} \int  \left | \frac{\alpha \omega^{2m}}{|\tilde g(\omega)|^2 + \alpha \omega^{2m}} \right |^2|\tilde f(\omega)|^2\,d\omega = 0,
\end{aligned}\end{equation}
as needed.
\end{proof}


%% file: proofs/theorem_representingsupp.tex
Below are the proofs of Theorem~\ref{theorem:representing}\ref{theorem:representing:kernel}-\ref{theorem:representing:apriorisup}.
\paragraph{\ref{theorem:representing:kernel}:} Suppose 
\begin{equation}
h_n(x) = n^{-1} \sum_{i=1}^n K_\nu(x-Y_i) = (\nu n)^{-1} \sum_{i=1}^n K(\nu^{-1}(x-Y_i))
\end{equation}
is a kernel density estimate of \(h\). Then letting \(\tilde
P_n(\omega) = \frac1n \sum_{i=1}^n e^{-i\omega Y_i}\) be the Fourier
transform of the empirical distribution and \(\tilde K_{\alpha,\nu} = \tilde
\varphi_\alpha \tilde K_\nu\)
\begin{equation}
\tilde f_n^\alpha(\omega) = \tilde \varphi_\alpha(\omega) \tilde K_\nu(\omega) \tilde P_n(\omega) = \tilde K_{\alpha,\nu}(\omega) \tilde P_n(\omega).
\end{equation}
Taking inverse transforms yields
\begin{equation}
f_n^\alpha(x) = \frac1n\sum_{i=1}^n K_{\alpha,\nu}(x-Y_i).
\end{equation}
\paragraph{\ref{theorem:representing:unitint}:} This follows from the fact that \(\tilde \varphi_\alpha(0) = 1\), in which case \(\int \varphi_\alpha(x)\,dx = 1\), and the fact that \(\int u*v = \left (\int u  \right ) \left ( \int v \right )\).
\paragraph{\ref{theorem:representing:apriori}:} Since the Fourier transform of \(D^k f_n^\alpha\) is \(\mathcal{F}[D^k f_n^\alpha](\omega) = (i\omega)^k \tilde \varphi_\alpha(\omega) \tilde h_n(\omega)\), we have \(\|D^k f_n^\alpha\|^2 = (2\pi)^{-1} \int |\omega^k \tilde \varphi_\alpha(\omega) \tilde h_n(\omega)|^2\,dt \leq (2\pi)^{-1} \int |\omega^k \tilde \varphi_\alpha(\omega)|^2\,dt\). Now, dropping the second term in the denominator of \(\tilde \varphi_\alpha(\omega)\), \(|\omega^k \varphi_\alpha(\omega)| = \frac{|\omega|^k|\tilde g(\omega)|}{|\tilde g(\omega)|^2 + \alpha |\omega|^{2m}} \leq |\omega|^k|\tilde g(\omega)|^{-1} \leq \frac1\alpha M|\omega|^k|\tilde g(\omega)|^{-1}\) Similarly, \(|\omega^k \tilde \varphi_\alpha(\omega)| \leq \frac1\alpha |\omega|^{k-2m}|\tilde g(\omega)|\). Thus
\begin{equation}
  \begin{aligned}
    \|D^k f_n^\alpha\|^2 &\leq (2\pi)^{-1} \int |\omega^k \tilde \varphi_\alpha(\omega)|^2\,dt\\
                         &\leq \frac1{\alpha^2} (2\pi)^{-1} \int \max\{M^2|\omega|^{2k} |\tilde g(\omega)|^{-2},|\omega|^{2(k-2m)}|\tilde g(\omega)|^2\}\,dt\\
                           &= C_k\alpha^{-2}
  \end{aligned}
\end{equation}
To see that the integral is finite, notice that the first term in the maximum is integrable in a neighborhood of zero, while the second term is integrable outside a neighborhood of zero as long as \(k < 2m\) or \(k = 2m\) and \(|\tilde g(\omega)|^2\) is integrable. Take a maximum over the \(C_k\) for a constant that holds for all \(0 \leq k < 2m\) (or including \(2m\) if \(g\) is square-integrable).
\paragraph{\ref{theorem:representing:apriorisup}:} Similar to the previous result, we have
\begin{equation}
  \begin{aligned}
    |D^k f_n^\alpha(x)| &= \frac1{2\pi} \left | \int e^{i \omega x} (i \omega)^k \tilde \varphi_\alpha(\omega) \tilde h_n(\omega) \,d\omega \right |\\
                        &\leq C \int |\omega^k \tilde \varphi_\alpha(\omega)|\,d\omega\\
                        &\leq C \alpha^{-1} \int \max\{M|\omega|^k |\tilde g(\omega)|^{-1},|\omega|^{k-2m}|\tilde g(\omega)|\}\,d\omega\\ 
                        &\leq C_k \alpha^{-1},
  \end{aligned}
\end{equation}
  where again we notice the final integral is finite because the first term in the maximum is integrable near zero, and the second term is integrable away from zero under the conditions on \(k\) and \(\tilde g\).


%% file: proofs/lem_systematic_bounds.tex
\begin{proof}[Proof of Lemma \ref{lem:systematic_bounds}]
  Observe that in each of these cases, \(g\) is even, so that \(|\tilde g|^2 = \tilde g^2\). We first introduce some notions that will be helpful in all three cases, and then we dive into the particulars. Note that
  \begin{equation}
    \begin{aligned}
      \tilde f^\alpha - \tilde f &= \tilde \varphi_\alpha \tilde h - \tilde f\\
      &= \tilde \varphi_\alpha \tilde g \tilde f - \tilde f\\
      &= (\tilde \varphi_\alpha \tilde g - 1)\tilde f.
    \end{aligned}
  \end{equation}
  Now,
  \begin{equation}
    \begin{aligned}
      \tilde \varphi_\alpha(\omega) \tilde g(\omega) - 1 &= \frac{\tilde g(\omega)^2}{\tilde g(\omega)^2 + \alpha \omega^{2m}} - \frac{\tilde g(\omega)^2 + \alpha \omega^{2m}}{\tilde g(\omega)^2 + \alpha \omega^{2m}}\\
      &= \frac{\alpha \omega^{2m}}{\tilde g(\omega)^2 + \alpha \omega^{2m}},
    \end{aligned}
  \end{equation}
  so if we define
  \begin{equation}
    \theta(\omega) \vcentcolon = \frac{\alpha \omega^{2m-k}}{\tilde g(\omega)^2 + \alpha \omega^{2m}},
  \end{equation}
  then we can bound the systematic error in the following way:
  
  \begin{equation}\begin{aligned}
      \|f^\alpha - f\|^2 &= (2\pi)^{-1} \|\tilde f^\alpha -\tilde f\|^2\\
      &= (2\pi)^{-1} \int \left [ \frac{\alpha \omega^{2m}}{\tilde g(\omega)^2 + \alpha \omega^{2m}} \right ]^2 |\tilde f(\omega)|^2 \,d\omega\\
      &= (2\pi)^{-1} \int |\theta(\omega)|^2 |\omega^k\tilde f(\omega)|^2 \,d\omega\\
      &\leq C \sup_{\omega>0} |\theta(\omega)|^2,
    \end{aligned}\end{equation}
  where we can restrict the domain of the supremum to strictly positive \(\omega\) because \(|\theta(\omega)|^2\) is an even, non-negative function and \(|\theta(0)|^2 = 0\).
  \paragraph{(i) (Normal errors)} For normal errors, we have \(\tilde g(\omega) = e^{-\omega^2/2}\), so that
  \begin{equation}
    |\theta(\omega)| = \frac{\alpha \omega^{2m-k}}{e^{-\omega^2} + \alpha \omega^{2m}}
  \end{equation}
  Note that 
\begin{equation}
  \theta(\omega) \leq \alpha \omega^{2m-k}e^{\omega^2}
  \qquad\text{ and }\qquad
  \theta(\omega) \leq \omega^{-k},
\end{equation}
where the upper bounds are increasing and decreasing, respectively, with \(\omega\). Thus
\[
  \sup_{\omega > 0} \theta(\omega) \leq \omega_0^{-k}
\]
with \(\omega_0\) being where the upper bounds intersect, i.e. \(\omega_0>0\) satisfies
\[
  \alpha \omega_0^{2m-k} e^{\omega_0^2} = \omega_0^{-k}.
\]
Taking \(m\)\textsuperscript{th} roots and rearranging, we find
\begin{equation}
(\omega_0^{2}/m) e^{\omega_0^2/m} = 1/(m\alpha^{1/m}),
\end{equation}
so that
\begin{equation}
  \omega_0^2/m = W(1/(m\alpha^{1/m})),
\end{equation}
whence
\begin{equation}
  \omega_0 = \left ( mW(1/(m\alpha^{1/m})) \right )^{\frac12}.
\end{equation}
Thus we have that
\begin{equation}
  \sup_{\omega > 0} |\theta(\omega)| \leq \left ( mW(1/(m\alpha^{1/m})) \right )^{-\frac{k}{2}},
\end{equation}
from which it follows that
\begin{equation}
  \sup_{\omega > 0} |\theta(\omega)|^2 \leq \left ( mW(1/(m\alpha^{1/m})) \right )^{-k},
\end{equation}
giving the result.
\paragraph{(ii) (Cauchy errors)} For Cauchy errors, we have \(\tilde g(\omega) = e^{-|\omega|}\), so
  \begin{equation}
    \theta(\omega) = \frac{\alpha \omega^{2m-k}}{e^{-2|\omega|} + \alpha\omega^{2m}}
  \end{equation}
  and seek \(\sup_{\omega>0}|\theta(\omega)|\). Now we note that for \(\omega > 0\),
  \begin{equation}
    \theta(\omega) \leq \alpha \omega^{2m-k}e^{2\omega}\qquad\text{and}\qquad\theta(\omega) \leq \omega^{-k},
  \end{equation}
  which are increasing and decreasing, respectively, with \(\omega\). Thus \(\sup_{\omega \geq 0} \theta(\omega) \leq \omega_0^{-k}\) with \(\omega_0\) the intersection of the two upper bounds, satisfying \(\alpha \omega_0^{2m-k}e^{2\omega_0} = \omega_0^{-k}\). Re-arranging that equality, we find that
  \begin{equation}
    (\omega_0/m)e^{\omega_0/m} = m^{-1}\alpha^{-\frac{1}{2m}},
  \end{equation}
  so that
  \begin{equation}
    \omega_0 = m W\left(m^{-1} \alpha^{-\frac{1}{2m}}\right).
  \end{equation}
  Thus
  \begin{equation}\begin{aligned}
      \|f^\alpha - f\|^2 &\leq C \sup_\omega \left [ \frac{\alpha \omega^{2m-k}}{e^{-2|\omega|} + \alpha\omega^{2m}} \right ]^2\\
      &\leq  \frac{C}{m^{2k} W(m^{-1}\alpha^{-\frac{1}{2m}})^{2k}}.
  \end{aligned}\end{equation}
\paragraph{(iii) (Laplace errors)} For Laplace errors, we have \(\tilde g(\omega) = (1+\omega^2)^{-1}\), so that
  \begin{equation}
    \theta(\omega) = \frac{\alpha \omega^{2m-k}}{(1+\omega^2)^{-2} + \alpha \omega^{2m}},
  \end{equation}
  i.e.
  \begin{equation}
    \theta(\omega) = \frac{\alpha \omega^{2m-k}(1+\omega^2)^2}{1 + \alpha \omega^{2m}(1+\omega^2)^2},
  \end{equation}
  Since both terms in the denominator are positive, we have that
  \begin{equation}
    \theta(\omega) \leq \alpha \omega^{2m-k}(1+\omega^2)^2\qquad\text{and}\qquad \theta(\omega) \leq \omega^{-k},
  \end{equation}
  where the upper bounds are increasing and decreasing monotonically to \(\infty\) and \(0\), respectively, so that they must intersect exactly once, and the value at the intersection is an upper bound for \(\sup_{\omega} |\theta(\omega)|\). For \(\alpha\) small enough, this intersection occurs at some \(\omega \geq 1\), and for \(\omega \geq 1\), we can simplify one of the upper bounds, so that
  \begin{equation}
    \theta(\omega) \leq \alpha \omega^{2m-k}(1+\omega^2)^2 \leq 4\alpha \omega^{2m-k+4}\qquad\text{and}\qquad \theta(\omega) \leq \omega^{-k}.
  \end{equation}
  Now the intersection of the upper bounds \(4\alpha \omega_0^{2m-k+4} = \omega_0^{-k}\) is found to occur at 
  \begin{equation}
    \omega_0 = \left ( \frac{1}{4\alpha} \right )^{\frac{1}{2m+4}},
  \end{equation}
  so that
  \begin{equation}
    \sup_\omega |\theta(\omega)| \leq \left ( \frac{1}{4\alpha} \right )^{-\frac{k}{2m+4}}
  \end{equation}
  and
  \begin{equation}
    \sup_\omega \theta(\omega)^2 \leq \left ( \frac{1}{4\alpha} \right )^{-\frac{k}{m+2}}
  \end{equation}
  Finally, this yields that
  \begin{equation}
    \|f^\alpha - f\|^2 \leq C\left ( \frac{1}{4\alpha} \right )^{-\frac{k}{m+2}},
  \end{equation}
  as needed.
\end{proof}


%% file: proofs/theorem_rates.tex
\begin{proof} In the first two cases, we will combine Corollary~\ref{cor:upper} and Lemma~\ref{lem:systematic_bounds} to find an upper bound for \(\mathbb{E}\|f_n^\alpha - f\|^2\). For Laplace errors, we will use Lamma~\ref{lem:systematic_bounds}, but we will need to sharpen the first term in Corollary~\ref{cor:upper}.
  \paragraph{(i) (Normal errors)} 
  From Corollary~\ref{cor:upper} and Lemma~\ref{lem:systematic_bounds}, we have that
  \begin{equation}\label{eqn:normalupperbound}
    \mathbb{E} \|f_n^{\alpha_n} - f\|^2 \leq C_1\delta_n^2/\alpha_n + \frac{C_2}{m^kW(1/(m\alpha^{1/m}))^k}
  \end{equation}
  We will use \(\alpha_n = \delta_n^2 W(\delta_n^{-\frac{2}{k}})^{k}\), which comes from the heuristic that the two summands in the upper bound ought to decrease at the same rate, i.e. by solving
  \begin{equation}
    \frac{\delta_n^2}{\alpha_n} = \log(\alpha^{-\frac{1}{m}})^{-k},
  \end{equation}
  where the \(\log(\cdot)\) is used in lieu of \(W(\cdot)\) as a
  simplification justified by the fact that \(\log(x) \sim W(x)\) as \(x
  \to \infty\). After solving for \(\alpha_n\), constants are ignored
  with wild abandon to yield a simpler \(\alpha_n\).

  We show that with the above choice of \(\alpha_n\), the bound in
  Equation~\eqref{eqn:normalupperbound} is asymptotically equivalent to
  \(C[\log \delta_n]^{-k}\). Plugging \(\alpha_n = \delta_n^2
  W(\delta_n^{-\frac2k})^{k}\) into the first term of
  Equation~\eqref{eqn:normalupperbound}, we get
  \begin{equation}\begin{aligned}
      C_1 \delta_n^2/\alpha_n &= C_1 W(\delta_n^{-\frac{2}{k}})^{-k}.
    \end{aligned}\end{equation}
  For the second, we have (equivalences are as \(\alpha_n \to 0\) or \(\delta_n \to 0\))
  \begin{equation}\begin{aligned}
      \frac{C_2}{m^kW(1/(m\alpha^{1/m}))^k}  &\sim C_2m^{-k}\left [\log(m^{-1}\alpha^{-\frac{1}{m}})\right ]^{-k}\\
      &= C_2\left [\log(m^{-m} \alpha_n^{-1})\right ]^{-k}\\
      &= C_2\left [\log(m^{-m} \delta_n^{-2} W(\delta_n^{-\frac{2}{k}})^{-k}) \right ]^{-k}\\
      &= C_2\left [ -m\log(m) - 2\log(\delta_n) - k \log( W(\delta_n^{-\frac{2}{k}}))\right ]^{-k}\\
      &= C_2\left [ -m\log(m) - 2\log(\delta_n) - k \left \{ -\frac{2}{k} \log (\delta_n) - W(\delta_n^{-\frac{2}{k}}) \right \}\right ]^{-k}\\
      &=  \frac{C_2}{\left [-m\log(m)+ k W(\delta_n^{-\frac{2}{k}})  \right ]^k}\\
      &\sim C_3W(\delta_n^{-\frac{2}{k}})^{-k},\\
    \end{aligned}\end{equation}
  where the fifth line follows from Fact~\ref*{fact:lambertW}\ref{fact:lambertW:logW}.
  The above leads us to
  \begin{equation}\begin{aligned}
      C_1\delta_n^2/\alpha_n + \frac{C_2}{m^kW(1/(m\alpha^{1/m}))^k} &= O(W(\delta_n^{-\frac{2}{k}})^{-k})\\
      &=O(\log(\delta_n^{-\frac{2}{k}})^{-k})\\
      &=O(\log(\delta_n^{-1})^{-k})\\
    \end{aligned}\end{equation}
  giving the result.
  \paragraph{(ii) (Cauchy errors)} We find our \(\alpha_n\) again by the heuristic of solving
  \begin{equation}
    \delta_n^2/\alpha_n = \log(\alpha_n^{-\frac{1}{2m}})^{-2k}.
  \end{equation}
  This yields (again, flagrantly dropping constants wherever they appear) \(\alpha_n = \delta_n^2 W(\delta^{-\frac{1}{k}})^{2k}\).
  From Corollary~\ref{cor:upper} and Lemma~\ref{lem:systematic_bounds}, we find that
  \begin{equation}\label{eqn:cauchyupperbound}
    \mathbb{E} \|f_n^{\alpha_n} - f\|^2 \leq C_1\delta_n^2/\alpha_n + \frac{C_2}{m^{2k}W(m^{-1}\alpha_n^{-\frac{1}{2m}})^{2k}}
  \end{equation}
  The first term, with our choice of \(\alpha_n\), yields
  \begin{equation}
    C_1 \delta_n^2/\alpha_n = C_1 W(\delta_n^{-\frac{1}{k}})^{-2k}.
  \end{equation}
  The second term is
  \begin{equation}\begin{aligned}
    C_2m^{-2k}W(m^{-1}\alpha^{-\frac{1}{2m}})^{-2k} &\sim C_2m^{-2k}\log(m^{-1}\alpha^{-\frac{1}{2m}})^{-2k}\\
    &= C_2m^{-2k}\log(m^{-1}\delta_n^{-\frac{1}{m}} W(\delta_n^{-\frac{1}{k}})^{-2k})\\
    &= C_2m^{-2k}[\log(m^{-1}) - \frac{1}{m} \log(\delta_n) - \frac{k}{m}\log(W(\delta_n^{-\frac{1}{k}}))]^{-2k}\\
    &= C_2m^{-2k}[\log(m^{-1}) + \frac{k}{m}W(\delta_n^{-\frac{1}{k}})]^{-2k}\\
    &\sim C_3[W(\delta_n^{-\frac{1}{k}})]^{-2k}.
  \end{aligned}\end{equation}
Combining the two terms, we find that
\begin{equation}\begin{aligned}
    \mathbb{E} \|f_n^{\alpha_n} - f\|^2 &= O(W(\delta_n^{-\frac{1}{k}})^{-2k})\\
    &= O(\log(\delta_n^{-\frac{1}{k}})^{-2k})\\
    &= O(\log(\delta_n^{-1})^{-2k})\\
\end{aligned}\end{equation}
  \paragraph{(iii) (Laplace errors)}
  We first need to sharpen the upper bound in
  Theorem~\ref{theorem:representing}\ref{theorem:representing:bound}.  We have that \(\tilde
  \varphi_\alpha(\omega) \leq 1 + \omega^2\) and
  \(\varphi_\alpha(\omega) \leq \frac12 \alpha^{\frac12} \omega^{-m}\),
  which are monotone increasing and decreasing, respectively, so
  \(\tilde \varphi_\alpha(\omega)\) is upper-bounded by the intersection
  of the two. For \(\alpha\) sufficiently small (specifically, \(\alpha
  < \frac{1}{16}\)), that intersection is at a \(\omega \geq 1\), and
  for such \(\omega\) we have \(\tilde \varphi_\alpha(\omega) \leq 1 +
  \omega^2 \leq 2\omega^2\). Solving \(2\omega^2 = \frac12
  \alpha^{\frac12}\omega^{-m}\) yields \(\omega_0 =
  (16\alpha)^{-\frac{1}{2(m+2)}}\), so that \(\tilde
  \varphi_\alpha(\omega) \leq 2(16\alpha)^{-\frac{1}{m+2}}\). Using
  the same argument as the proof of Lemma~\ref{lem:randompart} yields
  that \(\mathbb{E}\|f_n^\alpha - f^\alpha\|^2 \leq C \delta_n^2
  \alpha^\frac{-2}{m+2}\). Now, from \(\mathbb{E}\|f_n^\alpha - f\|^2
  \leq 2 \mathbb{E}\|f_n^\alpha - f^\alpha\|^2 + 2\|f^\alpha - f\|^2\),
  Lemma~\ref{lem:systematic_bounds}, and the preceding, we have that,
  for constants \(C\) and \(D\),
  \begin{equation}\label{eqn:laplace_err_bound}
    \mathbb{E}\|f_n^\alpha - f\|^2 \leq C\delta_n^2 \alpha^{\frac{-2}{m+2}} + D\alpha_n^{\frac{k}{m+2}}
  \end{equation}

  We find our \(\alpha_n\) again by the heuristic of solving
  \(\delta_n^2\alpha_n^{\frac{-2}{m+2}} = \alpha_n^{\frac{k}{m+2}}\), 
  which yields \(\alpha_n = \delta_n^{\frac{2(m+2)}{k+2}}\). Plugging
  this back into Equation~\eqref{eqn:laplace_err_bound} results in, we find
  \begin{equation}\begin{aligned}
      \|f_n^\alpha - f\|^2 &\leq C\delta_n^2 \alpha^{\frac{-2}{m+2}} + D\alpha_n^{\frac{k}{m+2}}\\
      &= (C+D) \delta_n^{\frac{2k}{k+2}},
    \end{aligned}\end{equation}
  giving the result.
\end{proof}


%% file: proofs/prop_particularlaplace.tex
\begin{proof}[Proof of Proposition \ref{prop:particularlaplace}] Let
  \(f_n^\lambda\) denote the kernel estimator of \(f\) described in
  Equation (4) of \cite{zhang_fourier_1990}, or Equation~\ref{eqn:dkde}
  here, using the kernel \(k(\cdot)\) and bandwidth \(\lambda\). We will
  prove this proposition by relating the smoothness-penalized estimator
  \(f_n^\alpha\) to the deconvoluting kernel estimator
  \(f_n^\lambda\). Specifically, we will show that the upper bound for
  \(\mathbb{E}\|f_n^\lambda - f\|^2\) demonstrated in the proof of
  Theorem 1 of \cite{zhang_fourier_1990} essentially holds also for
  \(\mathbb{E}\|f_n^\alpha - f\|^2\), but with an added \(2\cdot
  4^{\frac1{m+2}}C\lambda^{-1}\alpha^{\frac1{m+2}}\). Since it is
  shown in Example 3 of \cite{zhang_fourier_1990} that
  \(\mathbb{E}\|f_n^\lambda - f\|^2 = O(n^{-\frac27})\), this is
  sufficient to prove Proposition~\ref{prop:particularlaplace} with the
  specified choice of \(\alpha\).

  By Theorem~\ref{theorem:representing}\ref{theorem:representing:kernel}, we can write
  \(f_n^\alpha\) as a kernel-like estimator
  \begin{equation}
    f_n^\alpha(x) = \frac1n \sum_{j=1}^n K_{\alpha,\lambda}(Y_j,x)
  \end{equation}
  with kernel \(K_{\alpha,\lambda}(Y,x) = \frac1{2\pi}\int \tilde
  \varphi_\alpha(\omega) \tilde k(\lambda \omega)
  e^{i\omega(x-Y)}\,d\omega.\) Then we have
  \begin{equation}
    \begin{aligned}
      \mathbb{E}\|f_n^\alpha - f\|^2 &= \mathbb{E}\int (f_n^\alpha(x) - f(x))^2\,dx\\
      &= \int \Var(f_n^\alpha(x))\,dx + \int (\mathbb{E}[f_n^\alpha(x)] - f(x))^2\,dx\\
      &= \frac1n\int \Var(K_{\alpha,\lambda}(Y,x))\,dx + \int (\mathbb{E}[K_{\alpha,\lambda}(Y,x)] - f(x))^2\,dx\\
    \end{aligned}
  \end{equation}
  Now, for the variance we find
  \begin{equation}
    \begin{aligned}
      \frac{2\pi}n\int \Var(K_{\alpha,\lambda}(Y,x))\,dx &\leq \frac{2\pi}n\int \mathbb{E}K_{\alpha,\lambda}(Y,x)^2\,dx\\
      &= \frac1n\mathbb{E}\int \tilde K_{\alpha,\lambda}(Y,x)^2\,dx\\
      &= \frac1n\mathbb{E}\int |\tilde \varphi_\alpha(\omega) \tilde k(\lambda\omega)e^{-i\omega Y}|^2\,dx\\
      &\leq \frac1n\int |\tilde \varphi_\alpha(\omega) \tilde k(\lambda\omega)|^2\,dx\\
      &\leq \frac1n\int |\tilde k(\lambda\omega)/\tilde g(\omega)|^2\,dx,
    \end{aligned}
  \end{equation}
  where the final line follows from \(|\tilde \varphi_\alpha(\omega)
  \leq 1/\tilde g(\omega)|\) for all \(\omega\), and the final line is
  identical to the upper bound for the variance of \(f_n^\lambda\) in \cite{zhang_fourier_1990}.

  For the bias,
  \begin{equation}
    \begin{aligned}
      2\pi \int (\mathbb{E}[K_{\alpha,\lambda}(Y,x)] - f(x))^2\,dx
      &\leq 4\pi \int (\mathbb{E}[K_{\alpha,\lambda}(Y,x)] - \mathbb{E}[f_n^\lambda(x)])^2 \\
      &\qquad+ 4\pi \int (\mathbb{E}[f_n^\lambda(x)] - f(x))^2\,dx
    \end{aligned}
  \end{equation}
  The second term here is twice the integrated squared bias of
  \(f_n^\lambda\). We need only find a good bound for the first term:
  \begin{equation}
    \begin{aligned}
      4\pi \int (\mathbb{E}[K_{\alpha,\lambda}(Y,x)] - \mathbb{E}[f_n^\lambda(x)])^2\,d\omega &= 2 \int |\mathbb{E}[\tilde K_{\alpha,\lambda}(Y,\omega)] - \mathbb{E}[\tilde f_n^\lambda(\omega)]|^2\,d\omega\\
      &= 2\int |\tilde \varphi_\alpha(\omega)\tilde k(\lambda \omega) \tilde f(\omega) - \tilde k(\lambda \omega)\tilde f(\omega)|^2\,d\omega\\
      &= 2\int \left |\frac{\alpha\omega^{2m-1}}{|\tilde g(\omega)|^2 + \alpha\omega^{2m}}\right|^2|\tilde k(\lambda \omega) \omega\tilde f(\omega)|^2\,d\omega\\
      &\leq 2\int \left |\frac{\alpha\omega^{2m-1}}{|\tilde g(\omega)|^2 + \alpha\omega^{2m}}\right|^2|\omega\tilde f(\omega)|^2\,d\omega\\
      &\leq 2C(4\alpha)^{\frac{1}{m+2}}.
    \end{aligned}
  \end{equation}
  The penultimate inequality follows from the fact that \(k(\cdot)\)
  is a pdf, so that \(\tilde k(\omega) \leq 1\) for all
  \(\omega\). The final inequality follows from the fact that
  \(\sup_\omega\left |\frac{\alpha\omega^{2m-1}}{|\tilde g(\omega)|^2
      + \alpha\omega^{2m}}\right|^2 \leq (4\alpha)^{\frac1{m+2}}\),
  demonstrated in the proof of
  Lemma~\ref*{lem:systematic_bounds}\ref{lem:systematic_bounds:laplace}. Putting
  it all together, we have
  \begin{equation}
    \begin{aligned}
    \mathbb{E}\|f_n^\alpha - f\|^2 &\leq \frac{1}{2\pi n}\int |\tilde k(\lambda\omega)/\tilde g(\omega)|^2\,dx \\
    &\qquad +2\int (\mathbb{E}[f_n^\lambda(x)] - f(x))^2\,dx\\
    &\qquad +\frac{C}{2\pi}(4\alpha)^{\frac{1}{m+2}}.
    \end{aligned}
  \end{equation}

  Theorem 1 and Example 3 of \cite{zhang_fourier_1990} show that the
  first two terms are \(O(n^{-\frac27})\), and with \(\alpha =
  O(n^{-\frac{2(m+2)}{7}})\), the third term is as well.

\end{proof}


%% file: proofs/lem_Caclosed.tex
\begin{proof}[Proof of Lemma \ref{lem:Caclosed}]
Clearly \(\mathcal{C}_a\) is convex. To see that it is closed, we must show
that any limit point of \(\mathcal{C}_a\) is also an element of \(\mathcal{C}_a\). To
that end, suppose \(\{\psi_n\}\subset \mathcal{C}_a\) and that there is some
\(\psi \in L_2(\mathbb{R})\) with \(\|\psi_n - \psi\| \to 0\). Clearly
\(\psi\) is non-negative and has support contained in \([-a,a]\). We
need only show that \(\int \psi = 1\). Consider
\begin{equation}\begin{aligned}
\left | 1 - \int \psi \right | &= \left | \int (\psi_n - \psi) \right |\\
&\leq \int |\psi_n - \psi|\\
&= \int |\mathds{1}_{[-a,a]} (\psi_n - \psi)|\\
&\leq \|\mathds{1}_{[-a,a]}\|\cdot\|\psi_n - \psi\|\\
&= 2a\cdot\|\psi_n - \psi\| \to 0,
\end{aligned}\end{equation}
which gives the result.
\end{proof}

%% file: proofs/lem_projf.tex
\begin{proof}[Proof of Lemma \ref{lem:projf}]
Let \(T_a = \int \mathds{1}_{[-a,a]^c}f\) be the quantity of mass \(f\) puts
outside of \([-a,a]\), and note that
\begin{equation}
T_a = \mathbb{P}(|X| > a) = \mathbb{P}(|X|^\beta > a^\beta) \leq \mathbb{E}[|X|^\beta]a^{-\beta}
\end{equation}
by Markov's inequality. We introduce an element of \(\mathcal{C}_a\) which is close to \(f\):
\begin{equation}
\psi_a = \mathds{1}_{[-a,a]}(f + T_a/2a).
\end{equation}
Then
\begin{equation}\begin{aligned}
\|\psi_a - f\| &= \|\mathds{1}_{[-a,a]}(f + T_a/2a) - \mathds{1}_{[-a,a]}f - \mathds{1}_{[-a,a]^c}f\|\\
&= \|\mathds{1}_{[-a,a]}T_a/2a - \mathds{1}_{[-a,a]^c}f\|\\
&\leq \|\mathds{1}_{[-a,a]}T_a/2a\| + \|\mathds{1}_{[-a,a]^c}f\|\\
&= T_a + \|\mathds{1}_{[-a,a]^c}f\|.
\end{aligned}\end{equation}
Now,
\begin{equation}\begin{aligned}
\|\mathds{1}_{[-a,a]^c}f\| &= \int \mathds{1}_{[-a,a]^c} f^2\\
&\leq \int \mathds{1}_{[-a,a]^c} f = T_a\\
\end{aligned}\end{equation}
for \(a\) large enough that \(f(t) < 1\) for all \(|t| > a\).
Thus \(\|P_{\mathcal{C}_a}f - f\| \leq \|\psi_a - f\| \leq 2T_a \leq 2\mathbb{E}[|X|^\beta]a^{-\beta}\), as needed. A similar approach yields the exponential version.
\end{proof}

%% file: tikz_plots/lemma16.tex
\begin{tikzpicture}[x=1pt,y=1pt]
\definecolor{fillColor}{RGB}{255,255,255}
\path[use as bounding box,fill=fillColor,fill opacity=0.00] (0,0) rectangle (469.75,216.81);
\begin{scope}
\path[clip] ( 24.60, 30.60) rectangle (457.15,192.21);
\definecolor{drawColor}{RGB}{0,0,0}

\path[draw=drawColor,line width= 0.8pt,line join=round,line cap=round] ( 40.62,194.37) --
	( 40.70,194.29) --
	( 40.78,194.21) --
	( 40.86,194.12) --
	( 40.94,194.04) --
	( 41.02,193.96) --
	( 41.10,193.88) --
	( 41.18,193.80) --
	( 41.26,193.71) --
	( 41.34,193.63) --
	( 41.42,193.55) --
	( 41.50,193.47) --
	( 41.58,193.38) --
	( 41.66,193.30) --
	( 41.74,193.21) --
	( 41.82,193.13) --
	( 41.90,193.05) --
	( 41.98,192.96) --
	( 42.06,192.88) --
	( 42.14,192.79) --
	( 42.22,192.71) --
	( 42.30,192.62) --
	( 42.38,192.53) --
	( 42.46,192.45) --
	( 42.54,192.36) --
	( 42.63,192.27) --
	( 42.71,192.19) --
	( 42.79,192.10) --
	( 42.87,192.01) --
	( 42.95,191.92) --
	( 43.03,191.84) --
	( 43.11,191.75) --
	( 43.19,191.66) --
	( 43.27,191.57) --
	( 43.35,191.48) --
	( 43.43,191.39) --
	( 43.51,191.30) --
	( 43.59,191.21) --
	( 43.67,191.12) --
	( 43.75,191.03) --
	( 43.83,190.94) --
	( 43.91,190.85) --
	( 43.99,190.76) --
	( 44.07,190.67) --
	( 44.15,190.58) --
	( 44.23,190.48) --
	( 44.31,190.39) --
	( 44.39,190.30) --
	( 44.47,190.21) --
	( 44.55,190.11) --
	( 44.63,190.02) --
	( 44.71,189.93) --
	( 44.79,189.83) --
	( 44.87,189.74) --
	( 44.95,189.65) --
	( 45.03,189.55) --
	( 45.11,189.46) --
	( 45.19,189.36) --
	( 45.27,189.27) --
	( 45.35,189.17) --
	( 45.43,189.08) --
	( 45.51,188.98) --
	( 45.59,188.89) --
	( 45.67,188.79) --
	( 45.75,188.69) --
	( 45.83,188.60) --
	( 45.91,188.50) --
	( 45.99,188.40) --
	( 46.07,188.30) --
	( 46.15,188.21) --
	( 46.23,188.11) --
	( 46.31,188.01) --
	( 46.39,187.91) --
	( 46.47,187.81) --
	( 46.55,187.72) --
	( 46.63,187.62) --
	( 46.71,187.52) --
	( 46.79,187.42) --
	( 46.87,187.32) --
	( 46.96,187.22) --
	( 47.04,187.12) --
	( 47.12,187.02) --
	( 47.20,186.92) --
	( 47.28,186.82) --
	( 47.36,186.72) --
	( 47.44,186.61) --
	( 47.52,186.51) --
	( 47.60,186.41) --
	( 47.68,186.31) --
	( 47.76,186.21) --
	( 47.84,186.11) --
	( 47.92,186.00) --
	( 48.00,185.90) --
	( 48.08,185.80) --
	( 48.16,185.69) --
	( 48.24,185.59) --
	( 48.32,185.49) --
	( 48.40,185.38) --
	( 48.48,185.28) --
	( 48.56,185.17) --
	( 48.64,185.07) --
	( 48.72,184.97) --
	( 48.80,184.86) --
	( 48.88,184.76) --
	( 48.96,184.65) --
	( 49.04,184.54) --
	( 49.12,184.44) --
	( 49.20,184.33) --
	( 49.28,184.23) --
	( 49.36,184.12) --
	( 49.44,184.01) --
	( 49.52,183.91) --
	( 49.60,183.80) --
	( 49.68,183.69) --
	( 49.76,183.58) --
	( 49.84,183.48) --
	( 49.92,183.37) --
	( 50.00,183.26) --
	( 50.08,183.15) --
	( 50.16,183.04) --
	( 50.24,182.94) --
	( 50.32,182.83) --
	( 50.40,182.72) --
	( 50.48,182.61) --
	( 50.56,182.50) --
	( 50.64,182.39) --
	( 50.72,182.28) --
	( 50.80,182.17) --
	( 50.88,182.06) --
	( 50.96,181.95) --
	( 51.04,181.84) --
	( 51.12,181.73) --
	( 51.20,181.61) --
	( 51.28,181.50) --
	( 51.37,181.39) --
	( 51.45,181.28) --
	( 51.53,181.17) --
	( 51.61,181.06) --
	( 51.69,180.94) --
	( 51.77,180.83) --
	( 51.85,180.72) --
	( 51.93,180.61) --
	( 52.01,180.49) --
	( 52.09,180.38) --
	( 52.17,180.27) --
	( 52.25,180.15) --
	( 52.33,180.04) --
	( 52.41,179.92) --
	( 52.49,179.81) --
	( 52.57,179.70) --
	( 52.65,179.58) --
	( 52.73,179.47) --
	( 52.81,179.35) --
	( 52.89,179.24) --
	( 52.97,179.12) --
	( 53.05,179.01) --
	( 53.13,178.89) --
	( 53.21,178.77) --
	( 53.29,178.66) --
	( 53.37,178.54) --
	( 53.45,178.42) --
	( 53.53,178.31) --
	( 53.61,178.19) --
	( 53.69,178.07) --
	( 53.77,177.96) --
	( 53.85,177.84) --
	( 53.93,177.72) --
	( 54.01,177.60) --
	( 54.09,177.49) --
	( 54.17,177.37) --
	( 54.25,177.25) --
	( 54.33,177.13) --
	( 54.41,177.01) --
	( 54.49,176.89) --
	( 54.57,176.77) --
	( 54.65,176.66) --
	( 54.73,176.54) --
	( 54.81,176.42) --
	( 54.89,176.30) --
	( 54.97,176.18) --
	( 55.05,176.06) --
	( 55.13,175.94) --
	( 55.21,175.82) --
	( 55.29,175.70) --
	( 55.37,175.58) --
	( 55.45,175.45) --
	( 55.53,175.33) --
	( 55.61,175.21) --
	( 55.69,175.09) --
	( 55.78,174.97) --
	( 55.86,174.85) --
	( 55.94,174.73) --
	( 56.02,174.60) --
	( 56.10,174.48) --
	( 56.18,174.36) --
	( 56.26,174.24) --
	( 56.34,174.11) --
	( 56.42,173.99) --
	( 56.50,173.87) --
	( 56.58,173.74) --
	( 56.66,173.62) --
	( 56.74,173.50) --
	( 56.82,173.37) --
	( 56.90,173.25) --
	( 56.98,173.13) --
	( 57.06,173.00) --
	( 57.14,172.88) --
	( 57.22,172.75) --
	( 57.30,172.63) --
	( 57.38,172.50) --
	( 57.46,172.38) --
	( 57.54,172.25) --
	( 57.62,172.13) --
	( 57.70,172.00) --
	( 57.78,171.88) --
	( 57.86,171.75) --
	( 57.94,171.63) --
	( 58.02,171.50) --
	( 58.10,171.37) --
	( 58.18,171.25) --
	( 58.26,171.12) --
	( 58.34,170.99) --
	( 58.42,170.87) --
	( 58.50,170.74) --
	( 58.58,170.61) --
	( 58.66,170.49) --
	( 58.74,170.36) --
	( 58.82,170.23) --
	( 58.90,170.10) --
	( 58.98,169.97) --
	( 59.06,169.85) --
	( 59.14,169.72) --
	( 59.22,169.59) --
	( 59.30,169.46) --
	( 59.38,169.33) --
	( 59.46,169.20) --
	( 59.54,169.08) --
	( 59.62,168.95) --
	( 59.70,168.82) --
	( 59.78,168.69) --
	( 59.86,168.56) --
	( 59.94,168.43) --
	( 60.02,168.30) --
	( 60.11,168.17) --
	( 60.19,168.04) --
	( 60.27,167.91) --
	( 60.35,167.78) --
	( 60.43,167.65) --
	( 60.51,167.52) --
	( 60.59,167.39) --
	( 60.67,167.26) --
	( 60.75,167.13) --
	( 60.83,166.99) --
	( 60.91,166.86) --
	( 60.99,166.73) --
	( 61.07,166.60) --
	( 61.15,166.47) --
	( 61.23,166.34) --
	( 61.31,166.21) --
	( 61.39,166.07) --
	( 61.47,165.94) --
	( 61.55,165.81) --
	( 61.63,165.68) --
	( 61.71,165.54) --
	( 61.79,165.41) --
	( 61.87,165.28) --
	( 61.95,165.15) --
	( 62.03,165.01) --
	( 62.11,164.88) --
	( 62.19,164.75) --
	( 62.27,164.61) --
	( 62.35,164.48) --
	( 62.43,164.35) --
	( 62.51,164.21) --
	( 62.59,164.08) --
	( 62.67,163.94) --
	( 62.75,163.81) --
	( 62.83,163.68) --
	( 62.91,163.54) --
	( 62.99,163.41) --
	( 63.07,163.27) --
	( 63.15,163.14) --
	( 63.23,163.00) --
	( 63.31,162.87) --
	( 63.39,162.73) --
	( 63.47,162.60) --
	( 63.55,162.46) --
	( 63.63,162.33) --
	( 63.71,162.19) --
	( 63.79,162.06) --
	( 63.87,161.92) --
	( 63.95,161.78) --
	( 64.03,161.65) --
	( 64.11,161.51) --
	( 64.19,161.38) --
	( 64.27,161.24) --
	( 64.35,161.10) --
	( 64.43,160.97) --
	( 64.52,160.83) --
	( 64.60,160.69) --
	( 64.68,160.56) --
	( 64.76,160.42) --
	( 64.84,160.28) --
	( 64.92,160.14) --
	( 65.00,160.01) --
	( 65.08,159.87) --
	( 65.16,159.73) --
	( 65.24,159.59) --
	( 65.32,159.46) --
	( 65.40,159.32) --
	( 65.48,159.18) --
	( 65.56,159.04) --
	( 65.64,158.91) --
	( 65.72,158.77) --
	( 65.80,158.63) --
	( 65.88,158.49) --
	( 65.96,158.35) --
	( 66.04,158.21) --
	( 66.12,158.07) --
	( 66.20,157.94) --
	( 66.28,157.80) --
	( 66.36,157.66) --
	( 66.44,157.52) --
	( 66.52,157.38) --
	( 66.60,157.24) --
	( 66.68,157.10) --
	( 66.76,156.96) --
	( 66.84,156.82) --
	( 66.92,156.68) --
	( 67.00,156.54) --
	( 67.08,156.40) --
	( 67.16,156.26) --
	( 67.24,156.12) --
	( 67.32,155.98) --
	( 67.40,155.84) --
	( 67.48,155.70) --
	( 67.56,155.56) --
	( 67.64,155.42) --
	( 67.72,155.28) --
	( 67.80,155.14) --
	( 67.88,155.00) --
	( 67.96,154.86) --
	( 68.04,154.72) --
	( 68.12,154.58) --
	( 68.20,154.43) --
	( 68.28,154.29) --
	( 68.36,154.15) --
	( 68.44,154.01) --
	( 68.52,153.87) --
	( 68.60,153.73) --
	( 68.68,153.59) --
	( 68.76,153.44) --
	( 68.84,153.30) --
	( 68.93,153.16) --
	( 69.01,153.02) --
	( 69.09,152.88) --
	( 69.17,152.73) --
	( 69.25,152.59) --
	( 69.33,152.45) --
	( 69.41,152.31) --
	( 69.49,152.16) --
	( 69.57,152.02) --
	( 69.65,151.88) --
	( 69.73,151.74) --
	( 69.81,151.59) --
	( 69.89,151.45) --
	( 69.97,151.31) --
	( 70.05,151.17) --
	( 70.13,151.02) --
	( 70.21,150.88) --
	( 70.29,150.74) --
	( 70.37,150.59) --
	( 70.45,150.45) --
	( 70.53,150.31) --
	( 70.61,150.16) --
	( 70.69,150.02) --
	( 70.77,149.88) --
	( 70.85,149.73) --
	( 70.93,149.59) --
	( 71.01,149.45) --
	( 71.09,149.30) --
	( 71.17,149.16) --
	( 71.25,149.01) --
	( 71.33,148.87) --
	( 71.41,148.73) --
	( 71.49,148.58) --
	( 71.57,148.44) --
	( 71.65,148.29) --
	( 71.73,148.15) --
	( 71.81,148.00) --
	( 71.89,147.86) --
	( 71.97,147.72) --
	( 72.05,147.57) --
	( 72.13,147.43) --
	( 72.21,147.28) --
	( 72.29,147.14) --
	( 72.37,146.99) --
	( 72.45,146.85) --
	( 72.53,146.70) --
	( 72.61,146.56) --
	( 72.69,146.41) --
	( 72.77,146.27) --
	( 72.85,146.12) --
	( 72.93,145.98) --
	( 73.01,145.83) --
	( 73.09,145.69) --
	( 73.17,145.54) --
	( 73.26,145.40) --
	( 73.34,145.25) --
	( 73.42,145.10) --
	( 73.50,144.96) --
	( 73.58,144.81) --
	( 73.66,144.67) --
	( 73.74,144.52) --
	( 73.82,144.38) --
	( 73.90,144.23) --
	( 73.98,144.08) --
	( 74.06,143.94) --
	( 74.14,143.79) --
	( 74.22,143.65) --
	( 74.30,143.50) --
	( 74.38,143.35) --
	( 74.46,143.21) --
	( 74.54,143.06) --
	( 74.62,142.92) --
	( 74.70,142.77) --
	( 74.78,142.62) --
	( 74.86,142.48) --
	( 74.94,142.33) --
	( 75.02,142.18) --
	( 75.10,142.04) --
	( 75.18,141.89) --
	( 75.26,141.75) --
	( 75.34,141.60) --
	( 75.42,141.45) --
	( 75.50,141.31) --
	( 75.58,141.16) --
	( 75.66,141.01) --
	( 75.74,140.87) --
	( 75.82,140.72) --
	( 75.90,140.57) --
	( 75.98,140.43) --
	( 76.06,140.28) --
	( 76.14,140.13) --
	( 76.22,139.98) --
	( 76.30,139.84) --
	( 76.38,139.69) --
	( 76.46,139.54) --
	( 76.54,139.40) --
	( 76.62,139.25) --
	( 76.70,139.10) --
	( 76.78,138.96) --
	( 76.86,138.81) --
	( 76.94,138.66) --
	( 77.02,138.51) --
	( 77.10,138.37) --
	( 77.18,138.22) --
	( 77.26,138.07) --
	( 77.34,137.93) --
	( 77.42,137.78) --
	( 77.50,137.63) --
	( 77.58,137.48) --
	( 77.67,137.34) --
	( 77.75,137.19) --
	( 77.83,137.04) --
	( 77.91,136.89) --
	( 77.99,136.75) --
	( 78.07,136.60) --
	( 78.15,136.45) --
	( 78.23,136.30) --
	( 78.31,136.16) --
	( 78.39,136.01) --
	( 78.47,135.86) --
	( 78.55,135.71) --
	( 78.63,135.57) --
	( 78.71,135.42) --
	( 78.79,135.27) --
	( 78.87,135.12) --
	( 78.95,134.98) --
	( 79.03,134.83) --
	( 79.11,134.68) --
	( 79.19,134.53) --
	( 79.27,134.38) --
	( 79.35,134.24) --
	( 79.43,134.09) --
	( 79.51,133.94) --
	( 79.59,133.79) --
	( 79.67,133.65) --
	( 79.75,133.50) --
	( 79.83,133.35) --
	( 79.91,133.20) --
	( 79.99,133.06) --
	( 80.07,132.91) --
	( 80.15,132.76) --
	( 80.23,132.61) --
	( 80.31,132.46) --
	( 80.39,132.32) --
	( 80.47,132.17) --
	( 80.55,132.02) --
	( 80.63,131.87) --
	( 80.71,131.73) --
	( 80.79,131.58) --
	( 80.87,131.43) --
	( 80.95,131.28) --
	( 81.03,131.13) --
	( 81.11,130.99) --
	( 81.19,130.84) --
	( 81.27,130.69) --
	( 81.35,130.54) --
	( 81.43,130.40) --
	( 81.51,130.25) --
	( 81.59,130.10) --
	( 81.67,129.95) --
	( 81.75,129.80) --
	( 81.83,129.66) --
	( 81.91,129.51) --
	( 81.99,129.36) --
	( 82.08,129.21) --
	( 82.16,129.07) --
	( 82.24,128.92) --
	( 82.32,128.77) --
	( 82.40,128.62) --
	( 82.48,128.47) --
	( 82.56,128.33) --
	( 82.64,128.18) --
	( 82.72,128.03) --
	( 82.80,127.88) --
	( 82.88,127.74) --
	( 82.96,127.59) --
	( 83.04,127.44) --
	( 83.12,127.29) --
	( 83.20,127.15) --
	( 83.28,127.00) --
	( 83.36,126.85) --
	( 83.44,126.70) --
	( 83.52,126.56) --
	( 83.60,126.41) --
	( 83.68,126.26) --
	( 83.76,126.11) --
	( 83.84,125.97) --
	( 83.92,125.82) --
	( 84.00,125.67) --
	( 84.08,125.52) --
	( 84.16,125.38) --
	( 84.24,125.23) --
	( 84.32,125.08) --
	( 84.40,124.94) --
	( 84.48,124.79) --
	( 84.56,124.64) --
	( 84.64,124.49) --
	( 84.72,124.35) --
	( 84.80,124.20) --
	( 84.88,124.05) --
	( 84.96,123.91) --
	( 85.04,123.76) --
	( 85.12,123.61) --
	( 85.20,123.46) --
	( 85.28,123.32) --
	( 85.36,123.17) --
	( 85.44,123.02) --
	( 85.52,122.88) --
	( 85.60,122.73) --
	( 85.68,122.58) --
	( 85.76,122.44) --
	( 85.84,122.29) --
	( 85.92,122.14) --
	( 86.00,122.00) --
	( 86.08,121.85) --
	( 86.16,121.70) --
	( 86.24,121.56) --
	( 86.32,121.41) --
	( 86.41,121.26) --
	( 86.49,121.12) --
	( 86.57,120.97) --
	( 86.65,120.82) --
	( 86.73,120.68) --
	( 86.81,120.53) --
	( 86.89,120.38) --
	( 86.97,120.24) --
	( 87.05,120.09) --
	( 87.13,119.95) --
	( 87.21,119.80) --
	( 87.29,119.65) --
	( 87.37,119.51) --
	( 87.45,119.36) --
	( 87.53,119.22) --
	( 87.61,119.07) --
	( 87.69,118.92) --
	( 87.77,118.78) --
	( 87.85,118.63) --
	( 87.93,118.49) --
	( 88.01,118.34) --
	( 88.09,118.20) --
	( 88.17,118.05) --
	( 88.25,117.90) --
	( 88.33,117.76) --
	( 88.41,117.61) --
	( 88.49,117.47) --
	( 88.57,117.32) --
	( 88.65,117.18) --
	( 88.73,117.03) --
	( 88.81,116.89) --
	( 88.89,116.74) --
	( 88.97,116.60) --
	( 89.05,116.45) --
	( 89.13,116.31) --
	( 89.21,116.16) --
	( 89.29,116.02) --
	( 89.37,115.87) --
	( 89.45,115.73) --
	( 89.53,115.58) --
	( 89.61,115.44) --
	( 89.69,115.29) --
	( 89.77,115.15) --
	( 89.85,115.01) --
	( 89.93,114.86) --
	( 90.01,114.72) --
	( 90.09,114.57) --
	( 90.17,114.43) --
	( 90.25,114.29) --
	( 90.33,114.14) --
	( 90.41,114.00) --
	( 90.49,113.85) --
	( 90.57,113.71) --
	( 90.65,113.57) --
	( 90.73,113.42) --
	( 90.82,113.28) --
	( 90.90,113.13) --
	( 90.98,112.99) --
	( 91.06,112.85) --
	( 91.14,112.70) --
	( 91.22,112.56) --
	( 91.30,112.42) --
	( 91.38,112.27) --
	( 91.46,112.13) --
	( 91.54,111.99) --
	( 91.62,111.85) --
	( 91.70,111.70) --
	( 91.78,111.56) --
	( 91.86,111.42) --
	( 91.94,111.27) --
	( 92.02,111.13) --
	( 92.10,110.99) --
	( 92.18,110.85) --
	( 92.26,110.70) --
	( 92.34,110.56) --
	( 92.42,110.42) --
	( 92.50,110.28) --
	( 92.58,110.14) --
	( 92.66,109.99) --
	( 92.74,109.85) --
	( 92.82,109.71) --
	( 92.90,109.57) --
	( 92.98,109.43) --
	( 93.06,109.29) --
	( 93.14,109.14) --
	( 93.22,109.00) --
	( 93.30,108.86) --
	( 93.38,108.72) --
	( 93.46,108.58) --
	( 93.54,108.44) --
	( 93.62,108.30) --
	( 93.70,108.16) --
	( 93.78,108.02) --
	( 93.86,107.87) --
	( 93.94,107.73) --
	( 94.02,107.59) --
	( 94.10,107.45) --
	( 94.18,107.31) --
	( 94.26,107.17) --
	( 94.34,107.03) --
	( 94.42,106.89) --
	( 94.50,106.75) --
	( 94.58,106.61) --
	( 94.66,106.47) --
	( 94.74,106.33) --
	( 94.82,106.19) --
	( 94.90,106.05) --
	( 94.98,105.91) --
	( 95.06,105.77) --
	( 95.14,105.64) --
	( 95.23,105.50) --
	( 95.31,105.36) --
	( 95.39,105.22) --
	( 95.47,105.08) --
	( 95.55,104.94) --
	( 95.63,104.80) --
	( 95.71,104.66) --
	( 95.79,104.53) --
	( 95.87,104.39) --
	( 95.95,104.25) --
	( 96.03,104.11) --
	( 96.11,103.97) --
	( 96.19,103.83) --
	( 96.27,103.70) --
	( 96.35,103.56) --
	( 96.43,103.42) --
	( 96.51,103.28) --
	( 96.59,103.15) --
	( 96.67,103.01) --
	( 96.75,102.87) --
	( 96.83,102.73) --
	( 96.91,102.60) --
	( 96.99,102.46) --
	( 97.07,102.32) --
	( 97.15,102.19) --
	( 97.23,102.05) --
	( 97.31,101.91) --
	( 97.39,101.78) --
	( 97.47,101.64) --
	( 97.55,101.51) --
	( 97.63,101.37) --
	( 97.71,101.23) --
	( 97.79,101.10) --
	( 97.87,100.96) --
	( 97.95,100.83) --
	( 98.03,100.69) --
	( 98.11,100.56) --
	( 98.19,100.42) --
	( 98.27,100.29) --
	( 98.35,100.15) --
	( 98.43,100.02) --
	( 98.51, 99.88) --
	( 98.59, 99.75) --
	( 98.67, 99.61) --
	( 98.75, 99.48) --
	( 98.83, 99.34) --
	( 98.91, 99.21) --
	( 98.99, 99.07) --
	( 99.07, 98.94) --
	( 99.15, 98.81) --
	( 99.23, 98.67) --
	( 99.31, 98.54) --
	( 99.39, 98.41) --
	( 99.47, 98.27) --
	( 99.56, 98.14) --
	( 99.64, 98.01) --
	( 99.72, 97.87) --
	( 99.80, 97.74) --
	( 99.88, 97.61) --
	( 99.96, 97.48) --
	(100.04, 97.34) --
	(100.12, 97.21) --
	(100.20, 97.08) --
	(100.28, 96.95) --
	(100.36, 96.81) --
	(100.44, 96.68) --
	(100.52, 96.55) --
	(100.60, 96.42) --
	(100.68, 96.29) --
	(100.76, 96.16) --
	(100.84, 96.03) --
	(100.92, 95.90) --
	(101.00, 95.76) --
	(101.08, 95.63) --
	(101.16, 95.50) --
	(101.24, 95.37) --
	(101.32, 95.24) --
	(101.40, 95.11) --
	(101.48, 94.98) --
	(101.56, 94.85) --
	(101.64, 94.72) --
	(101.72, 94.59) --
	(101.80, 94.46) --
	(101.88, 94.33) --
	(101.96, 94.21) --
	(102.04, 94.08) --
	(102.12, 93.95) --
	(102.20, 93.82) --
	(102.28, 93.69) --
	(102.36, 93.56) --
	(102.44, 93.43) --
	(102.52, 93.31) --
	(102.60, 93.18) --
	(102.68, 93.05) --
	(102.76, 92.92) --
	(102.84, 92.79) --
	(102.92, 92.67) --
	(103.00, 92.54) --
	(103.08, 92.41) --
	(103.16, 92.29) --
	(103.24, 92.16) --
	(103.32, 92.03) --
	(103.40, 91.91) --
	(103.48, 91.78) --
	(103.56, 91.65) --
	(103.64, 91.53) --
	(103.72, 91.40) --
	(103.80, 91.27) --
	(103.88, 91.15) --
	(103.97, 91.02) --
	(104.05, 90.90) --
	(104.13, 90.77) --
	(104.21, 90.65) --
	(104.29, 90.52) --
	(104.37, 90.40) --
	(104.45, 90.27) --
	(104.53, 90.15) --
	(104.61, 90.03) --
	(104.69, 89.90) --
	(104.77, 89.78) --
	(104.85, 89.65) --
	(104.93, 89.53) --
	(105.01, 89.41) --
	(105.09, 89.28) --
	(105.17, 89.16) --
	(105.25, 89.04) --
	(105.33, 88.91) --
	(105.41, 88.79) --
	(105.49, 88.67) --
	(105.57, 88.55) --
	(105.65, 88.43) --
	(105.73, 88.30) --
	(105.81, 88.18) --
	(105.89, 88.06) --
	(105.97, 87.94) --
	(106.05, 87.82) --
	(106.13, 87.70) --
	(106.21, 87.58) --
	(106.29, 87.45) --
	(106.37, 87.33) --
	(106.45, 87.21) --
	(106.53, 87.09) --
	(106.61, 86.97) --
	(106.69, 86.85) --
	(106.77, 86.73) --
	(106.85, 86.61) --
	(106.93, 86.49) --
	(107.01, 86.38) --
	(107.09, 86.26) --
	(107.17, 86.14) --
	(107.25, 86.02) --
	(107.33, 85.90) --
	(107.41, 85.78) --
	(107.49, 85.66) --
	(107.57, 85.55) --
	(107.65, 85.43) --
	(107.73, 85.31) --
	(107.81, 85.19) --
	(107.89, 85.08) --
	(107.97, 84.96) --
	(108.05, 84.84) --
	(108.13, 84.73) --
	(108.21, 84.61) --
	(108.29, 84.49) --
	(108.38, 84.38) --
	(108.46, 84.26) --
	(108.54, 84.14) --
	(108.62, 84.03) --
	(108.70, 83.91) --
	(108.78, 83.80) --
	(108.86, 83.68) --
	(108.94, 83.57) --
	(109.02, 83.45) --
	(109.10, 83.34) --
	(109.18, 83.23) --
	(109.26, 83.11) --
	(109.34, 83.00) --
	(109.42, 82.88) --
	(109.50, 82.77) --
	(109.58, 82.66) --
	(109.66, 82.54) --
	(109.74, 82.43) --
	(109.82, 82.32) --
	(109.90, 82.21) --
	(109.98, 82.09) --
	(110.06, 81.98) --
	(110.14, 81.87) --
	(110.22, 81.76) --
	(110.30, 81.65) --
	(110.38, 81.53) --
	(110.46, 81.42) --
	(110.54, 81.31) --
	(110.62, 81.20) --
	(110.70, 81.09) --
	(110.78, 80.98) --
	(110.86, 80.87) --
	(110.94, 80.76) --
	(111.02, 80.65) --
	(111.10, 80.54) --
	(111.18, 80.43) --
	(111.26, 80.32) --
	(111.34, 80.21) --
	(111.42, 80.10) --
	(111.50, 80.00) --
	(111.58, 79.89) --
	(111.66, 79.78) --
	(111.74, 79.67) --
	(111.82, 79.56) --
	(111.90, 79.46) --
	(111.98, 79.35) --
	(112.06, 79.24) --
	(112.14, 79.14) --
	(112.22, 79.03) --
	(112.30, 78.92) --
	(112.38, 78.82) --
	(112.46, 78.71) --
	(112.54, 78.60) --
	(112.62, 78.50) --
	(112.71, 78.39) --
	(112.79, 78.29) --
	(112.87, 78.18) --
	(112.95, 78.08) --
	(113.03, 77.97) --
	(113.11, 77.87) --
	(113.19, 77.77) --
	(113.27, 77.66) --
	(113.35, 77.56) --
	(113.43, 77.45) --
	(113.51, 77.35) --
	(113.59, 77.25) --
	(113.67, 77.15) --
	(113.75, 77.04) --
	(113.83, 76.94) --
	(113.91, 76.84) --
	(113.99, 76.74) --
	(114.07, 76.63) --
	(114.15, 76.53) --
	(114.23, 76.43) --
	(114.31, 76.33) --
	(114.39, 76.23) --
	(114.47, 76.13) --
	(114.55, 76.03) --
	(114.63, 75.93) --
	(114.71, 75.83) --
	(114.79, 75.73) --
	(114.87, 75.63) --
	(114.95, 75.53) --
	(115.03, 75.43) --
	(115.11, 75.33) --
	(115.19, 75.23) --
	(115.27, 75.14) --
	(115.35, 75.04) --
	(115.43, 74.94) --
	(115.51, 74.84) --
	(115.59, 74.74) --
	(115.67, 74.65) --
	(115.75, 74.55) --
	(115.83, 74.45) --
	(115.91, 74.36) --
	(115.99, 74.26) --
	(116.07, 74.17) --
	(116.15, 74.07) --
	(116.23, 73.97) --
	(116.31, 73.88) --
	(116.39, 73.78) --
	(116.47, 73.69) --
	(116.55, 73.59) --
	(116.63, 73.50) --
	(116.71, 73.41) --
	(116.79, 73.31) --
	(116.87, 73.22) --
	(116.95, 73.13) --
	(117.03, 73.03) --
	(117.12, 72.94) --
	(117.20, 72.85) --
	(117.28, 72.75) --
	(117.36, 72.66) --
	(117.44, 72.57) --
	(117.52, 72.48) --
	(117.60, 72.39) --
	(117.68, 72.30) --
	(117.76, 72.21) --
	(117.84, 72.11) --
	(117.92, 72.02) --
	(118.00, 71.93) --
	(118.08, 71.84) --
	(118.16, 71.75) --
	(118.24, 71.67) --
	(118.32, 71.58) --
	(118.40, 71.49) --
	(118.48, 71.40) --
	(118.56, 71.31) --
	(118.64, 71.22) --
	(118.72, 71.13) --
	(118.80, 71.05) --
	(118.88, 70.96) --
	(118.96, 70.87) --
	(119.04, 70.79) --
	(119.12, 70.70) --
	(119.20, 70.61) --
	(119.28, 70.53) --
	(119.36, 70.44) --
	(119.44, 70.35) --
	(119.52, 70.27) --
	(119.60, 70.18) --
	(119.68, 70.10) --
	(119.76, 70.01) --
	(119.84, 69.93) --
	(119.92, 69.85) --
	(120.00, 69.76) --
	(120.08, 69.68) --
	(120.16, 69.60) --
	(120.24, 69.51) --
	(120.32, 69.43) --
	(120.40, 69.35) --
	(120.48, 69.27) --
	(120.56, 69.18) --
	(120.64, 69.10) --
	(120.72, 69.02);
\end{scope}
\begin{scope}
\path[clip] (  0.00,  0.00) rectangle (469.75,216.81);
\definecolor{drawColor}{RGB}{0,0,0}

\path[draw=drawColor,line width= 0.4pt,line join=round,line cap=round] (120.72, 30.60) -- (361.03, 30.60);

\path[draw=drawColor,line width= 0.4pt,line join=round,line cap=round] (120.72, 30.60) -- (120.72, 24.60);

\path[draw=drawColor,line width= 0.4pt,line join=round,line cap=round] (200.83, 30.60) -- (200.83, 24.60);

\path[draw=drawColor,line width= 0.4pt,line join=round,line cap=round] (280.93, 30.60) -- (280.93, 24.60);

\path[draw=drawColor,line width= 0.4pt,line join=round,line cap=round] (361.03, 30.60) -- (361.03, 24.60);

\node[text=drawColor,anchor=base,inner sep=0pt, outer sep=0pt, scale=  1.00] at (120.72,  9.00) {b};

\node[text=drawColor,anchor=base,inner sep=0pt, outer sep=0pt, scale=  1.00] at (200.83,  9.00) {b + 1};

\node[text=drawColor,anchor=base,inner sep=0pt, outer sep=0pt, scale=  1.00] at (280.93,  9.00) {b + 2};

\node[text=drawColor,anchor=base,inner sep=0pt, outer sep=0pt, scale=  1.00] at (361.03,  9.00) {b + 3};

\path[draw=drawColor,line width= 0.4pt,line join=round,line cap=round] ( 24.60, 36.59) -- ( 24.60, 36.59);

\path[draw=drawColor,line width= 0.4pt,line join=round,line cap=round] ( 24.60, 36.59) -- ( 18.60, 36.59);

\node[text=drawColor,rotate= 90.00,anchor=base,inner sep=0pt, outer sep=0pt, scale=  1.00] at ( 10.20, 36.59) {0};
\end{scope}
\begin{scope}
\path[clip] ( 24.60, 30.60) rectangle (457.15,192.21);
\definecolor{drawColor}{RGB}{0,0,0}

\path[draw=drawColor,line width= 0.8pt,dash pattern=on 4pt off 4pt ,line join=round,line cap=round] (120.72, 69.02) --
	(120.96, 68.78) --
	(121.20, 68.54) --
	(121.44, 68.30) --
	(121.69, 68.06) --
	(121.93, 67.83) --
	(122.17, 67.60) --
	(122.41, 67.37) --
	(122.65, 67.15) --
	(122.89, 66.92) --
	(123.13, 66.70) --
	(123.37, 66.48) --
	(123.61, 66.27) --
	(123.85, 66.05) --
	(124.09, 65.84) --
	(124.33, 65.63) --
	(124.57, 65.42) --
	(124.81, 65.22) --
	(125.05, 65.01) --
	(125.29, 64.81) --
	(125.53, 64.61) --
	(125.77, 64.42) --
	(126.02, 64.22) --
	(126.26, 64.03) --
	(126.50, 63.84) --
	(126.74, 63.65) --
	(126.98, 63.47) --
	(127.22, 63.28) --
	(127.46, 63.10) --
	(127.70, 62.92) --
	(127.94, 62.75) --
	(128.18, 62.57) --
	(128.42, 62.40) --
	(128.66, 62.23) --
	(128.90, 62.06) --
	(129.14, 61.89) --
	(129.38, 61.73) --
	(129.62, 61.56) --
	(129.86, 61.40) --
	(130.10, 61.25) --
	(130.35, 61.09) --
	(130.59, 60.94) --
	(130.83, 60.78) --
	(131.07, 60.63) --
	(131.31, 60.48) --
	(131.55, 60.34) --
	(131.79, 60.19) --
	(132.03, 60.05) --
	(132.27, 59.91) --
	(132.51, 59.77) --
	(132.75, 59.64) --
	(132.99, 59.50) --
	(133.23, 59.37) --
	(133.47, 59.24) --
	(133.71, 59.11) --
	(133.95, 58.98) --
	(134.19, 58.86) --
	(134.43, 58.73) --
	(134.68, 58.61) --
	(134.92, 58.49) --
	(135.16, 58.37) --
	(135.40, 58.26) --
	(135.64, 58.14) --
	(135.88, 58.03) --
	(136.12, 57.92) --
	(136.36, 57.81) --
	(136.60, 57.70) --
	(136.84, 57.60) --
	(137.08, 57.49) --
	(137.32, 57.39) --
	(137.56, 57.29) --
	(137.80, 57.19) --
	(138.04, 57.10) --
	(138.28, 57.00) --
	(138.52, 56.91) --
	(138.76, 56.82) --
	(139.01, 56.73) --
	(139.25, 56.64) --
	(139.49, 56.55) --
	(139.73, 56.47) --
	(139.97, 56.38) --
	(140.21, 56.30) --
	(140.45, 56.22) --
	(140.69, 56.14) --
	(140.93, 56.07) --
	(141.17, 55.99) --
	(141.41, 55.92) --
	(141.65, 55.85) --
	(141.89, 55.78) --
	(142.13, 55.71) --
	(142.37, 55.64) --
	(142.61, 55.57) --
	(142.85, 55.51) --
	(143.09, 55.45) --
	(143.33, 55.39) --
	(143.58, 55.33) --
	(143.82, 55.27) --
	(144.06, 55.21) --
	(144.30, 55.15) --
	(144.54, 55.10) --
	(144.78, 55.05) --
	(145.02, 55.00) --
	(145.26, 54.95) --
	(145.50, 54.90) --
	(145.74, 54.85) --
	(145.98, 54.81) --
	(146.22, 54.76) --
	(146.46, 54.72) --
	(146.70, 54.68) --
	(146.94, 54.64) --
	(147.18, 54.60) --
	(147.42, 54.56) --
	(147.66, 54.53) --
	(147.91, 54.49) --
	(148.15, 54.46) --
	(148.39, 54.43) --
	(148.63, 54.39) --
	(148.87, 54.36) --
	(149.11, 54.34) --
	(149.35, 54.31) --
	(149.59, 54.28) --
	(149.83, 54.26) --
	(150.07, 54.24) --
	(150.31, 54.21) --
	(150.55, 54.19) --
	(150.79, 54.17) --
	(151.03, 54.15) --
	(151.27, 54.14) --
	(151.51, 54.12) --
	(151.75, 54.11) --
	(151.99, 54.09) --
	(152.24, 54.08) --
	(152.48, 54.07) --
	(152.72, 54.06) --
	(152.96, 54.05) --
	(153.20, 54.04) --
	(153.44, 54.03) --
	(153.68, 54.03) --
	(153.92, 54.02) --
	(154.16, 54.02) --
	(154.40, 54.02) --
	(154.64, 54.01) --
	(154.88, 54.01) --
	(155.12, 54.01) --
	(155.36, 54.01) --
	(155.60, 54.02) --
	(155.84, 54.02) --
	(156.08, 54.02) --
	(156.32, 54.03) --
	(156.57, 54.04) --
	(156.81, 54.04) --
	(157.05, 54.05) --
	(157.29, 54.06) --
	(157.53, 54.07) --
	(157.77, 54.08) --
	(158.01, 54.09) --
	(158.25, 54.10) --
	(158.49, 54.12) --
	(158.73, 54.13) --
	(158.97, 54.15) --
	(159.21, 54.16) --
	(159.45, 54.18) --
	(159.69, 54.20) --
	(159.93, 54.22) --
	(160.17, 54.24) --
	(160.41, 54.26) --
	(160.65, 54.28) --
	(160.89, 54.30) --
	(161.14, 54.32) --
	(161.38, 54.35) --
	(161.62, 54.37) --
	(161.86, 54.40) --
	(162.10, 54.42) --
	(162.34, 54.45) --
	(162.58, 54.47) --
	(162.82, 54.50) --
	(163.06, 54.53) --
	(163.30, 54.56) --
	(163.54, 54.59) --
	(163.78, 54.62) --
	(164.02, 54.65) --
	(164.26, 54.68) --
	(164.50, 54.72) --
	(164.74, 54.75) --
	(164.98, 54.78) --
	(165.22, 54.82) --
	(165.47, 54.85) --
	(165.71, 54.89) --
	(165.95, 54.93) --
	(166.19, 54.96) --
	(166.43, 55.00) --
	(166.67, 55.04) --
	(166.91, 55.08) --
	(167.15, 55.12) --
	(167.39, 55.16) --
	(167.63, 55.20) --
	(167.87, 55.24) --
	(168.11, 55.28) --
	(168.35, 55.32) --
	(168.59, 55.36) --
	(168.83, 55.40) --
	(169.07, 55.45) --
	(169.31, 55.49) --
	(169.55, 55.53) --
	(169.80, 55.58) --
	(170.04, 55.62) --
	(170.28, 55.67) --
	(170.52, 55.71) --
	(170.76, 55.76) --
	(171.00, 55.81) --
	(171.24, 55.85) --
	(171.48, 55.90) --
	(171.72, 55.95) --
	(171.96, 55.99) --
	(172.20, 56.04) --
	(172.44, 56.09) --
	(172.68, 56.14) --
	(172.92, 56.19) --
	(173.16, 56.24) --
	(173.40, 56.29) --
	(173.64, 56.34) --
	(173.88, 56.39) --
	(174.13, 56.44) --
	(174.37, 56.49) --
	(174.61, 56.54) --
	(174.85, 56.59) --
	(175.09, 56.64) --
	(175.33, 56.69) --
	(175.57, 56.75) --
	(175.81, 56.80) --
	(176.05, 56.85) --
	(176.29, 56.90) --
	(176.53, 56.96) --
	(176.77, 57.01) --
	(177.01, 57.06) --
	(177.25, 57.11) --
	(177.49, 57.17) --
	(177.73, 57.22) --
	(177.97, 57.27) --
	(178.21, 57.33) --
	(178.46, 57.38) --
	(178.70, 57.43) --
	(178.94, 57.49) --
	(179.18, 57.54) --
	(179.42, 57.59) --
	(179.66, 57.65) --
	(179.90, 57.70) --
	(180.14, 57.75) --
	(180.38, 57.81) --
	(180.62, 57.86) --
	(180.86, 57.92) --
	(181.10, 57.97) --
	(181.34, 58.02) --
	(181.58, 58.08) --
	(181.82, 58.13) --
	(182.06, 58.18) --
	(182.30, 58.24) --
	(182.54, 58.29) --
	(182.78, 58.34) --
	(183.03, 58.40) --
	(183.27, 58.45) --
	(183.51, 58.50) --
	(183.75, 58.55) --
	(183.99, 58.61) --
	(184.23, 58.66) --
	(184.47, 58.71) --
	(184.71, 58.76) --
	(184.95, 58.82) --
	(185.19, 58.87) --
	(185.43, 58.92) --
	(185.67, 58.97) --
	(185.91, 59.02) --
	(186.15, 59.07) --
	(186.39, 59.12) --
	(186.63, 59.17) --
	(186.87, 59.22) --
	(187.11, 59.27) --
	(187.36, 59.32) --
	(187.60, 59.37) --
	(187.84, 59.42) --
	(188.08, 59.47) --
	(188.32, 59.52) --
	(188.56, 59.57) --
	(188.80, 59.61) --
	(189.04, 59.66) --
	(189.28, 59.71) --
	(189.52, 59.76) --
	(189.76, 59.80) --
	(190.00, 59.85) --
	(190.24, 59.89) --
	(190.48, 59.94) --
	(190.72, 59.98) --
	(190.96, 60.03) --
	(191.20, 60.07) --
	(191.44, 60.11) --
	(191.69, 60.16) --
	(191.93, 60.20) --
	(192.17, 60.24) --
	(192.41, 60.28) --
	(192.65, 60.32) --
	(192.89, 60.36) --
	(193.13, 60.40) --
	(193.37, 60.44) --
	(193.61, 60.48) --
	(193.85, 60.52) --
	(194.09, 60.56) --
	(194.33, 60.60) --
	(194.57, 60.63) --
	(194.81, 60.67) --
	(195.05, 60.70) --
	(195.29, 60.74) --
	(195.53, 60.77) --
	(195.77, 60.81) --
	(196.02, 60.84) --
	(196.26, 60.87) --
	(196.50, 60.90) --
	(196.74, 60.94) --
	(196.98, 60.97) --
	(197.22, 61.00) --
	(197.46, 61.02) --
	(197.70, 61.05) --
	(197.94, 61.08) --
	(198.18, 61.11) --
	(198.42, 61.13) --
	(198.66, 61.16) --
	(198.90, 61.18) --
	(199.14, 61.21) --
	(199.38, 61.23) --
	(199.62, 61.25) --
	(199.86, 61.28) --
	(200.10, 61.30) --
	(200.35, 61.32) --
	(200.59, 61.34) --
	(200.83, 61.35) --
	(201.07, 61.37) --
	(201.31, 61.39) --
	(201.55, 61.40) --
	(201.79, 61.42) --
	(202.03, 61.43) --
	(202.27, 61.45) --
	(202.51, 61.46) --
	(202.75, 61.47) --
	(202.99, 61.48) --
	(203.23, 61.49) --
	(203.47, 61.50) --
	(203.71, 61.51) --
	(203.95, 61.52) --
	(204.19, 61.53) --
	(204.43, 61.53) --
	(204.67, 61.54) --
	(204.92, 61.54) --
	(205.16, 61.55) --
	(205.40, 61.55) --
	(205.64, 61.55) --
	(205.88, 61.56) --
	(206.12, 61.56) --
	(206.36, 61.56) --
	(206.60, 61.56) --
	(206.84, 61.55) --
	(207.08, 61.55) --
	(207.32, 61.55) --
	(207.56, 61.55) --
	(207.80, 61.54) --
	(208.04, 61.54) --
	(208.28, 61.53) --
	(208.52, 61.52) --
	(208.76, 61.52) --
	(209.00, 61.51) --
	(209.25, 61.50) --
	(209.49, 61.49) --
	(209.73, 61.48) --
	(209.97, 61.47) --
	(210.21, 61.46) --
	(210.45, 61.45) --
	(210.69, 61.44) --
	(210.93, 61.42) --
	(211.17, 61.41) --
	(211.41, 61.39) --
	(211.65, 61.38) --
	(211.89, 61.36) --
	(212.13, 61.34) --
	(212.37, 61.33) --
	(212.61, 61.31) --
	(212.85, 61.29) --
	(213.09, 61.27) --
	(213.33, 61.25) --
	(213.58, 61.23) --
	(213.82, 61.21) --
	(214.06, 61.19) --
	(214.30, 61.16) --
	(214.54, 61.14) --
	(214.78, 61.12) --
	(215.02, 61.09) --
	(215.26, 61.07) --
	(215.50, 61.04) --
	(215.74, 61.02) --
	(215.98, 60.99) --
	(216.22, 60.96) --
	(216.46, 60.93) --
	(216.70, 60.90) --
	(216.94, 60.87) --
	(217.18, 60.84) --
	(217.42, 60.81) --
	(217.66, 60.78) --
	(217.91, 60.75) --
	(218.15, 60.72) --
	(218.39, 60.69) --
	(218.63, 60.65) --
	(218.87, 60.62) --
	(219.11, 60.58) --
	(219.35, 60.55) --
	(219.59, 60.51) --
	(219.83, 60.48) --
	(220.07, 60.44) --
	(220.31, 60.40) --
	(220.55, 60.36) --
	(220.79, 60.33) --
	(221.03, 60.29) --
	(221.27, 60.25) --
	(221.51, 60.21) --
	(221.75, 60.17) --
	(221.99, 60.13) --
	(222.23, 60.08) --
	(222.48, 60.04) --
	(222.72, 60.00) --
	(222.96, 59.96) --
	(223.20, 59.91) --
	(223.44, 59.87) --
	(223.68, 59.82) --
	(223.92, 59.78) --
	(224.16, 59.73) --
	(224.40, 59.69) --
	(224.64, 59.64) --
	(224.88, 59.59) --
	(225.12, 59.55) --
	(225.36, 59.50) --
	(225.60, 59.45) --
	(225.84, 59.40) --
	(226.08, 59.35) --
	(226.32, 59.30) --
	(226.56, 59.25) --
	(226.81, 59.20) --
	(227.05, 59.15) --
	(227.29, 59.10) --
	(227.53, 59.05) --
	(227.77, 58.99) --
	(228.01, 58.94) --
	(228.25, 58.89) --
	(228.49, 58.83) --
	(228.73, 58.78) --
	(228.97, 58.72) --
	(229.21, 58.67) --
	(229.45, 58.61) --
	(229.69, 58.56) --
	(229.93, 58.50) --
	(230.17, 58.44) --
	(230.41, 58.39) --
	(230.65, 58.33) --
	(230.89, 58.27) --
	(231.14, 58.21) --
	(231.38, 58.16) --
	(231.62, 58.10) --
	(231.86, 58.04) --
	(232.10, 57.98) --
	(232.34, 57.92) --
	(232.58, 57.86) --
	(232.82, 57.80) --
	(233.06, 57.74) --
	(233.30, 57.67) --
	(233.54, 57.61) --
	(233.78, 57.55) --
	(234.02, 57.49) --
	(234.26, 57.42) --
	(234.50, 57.36) --
	(234.74, 57.30) --
	(234.98, 57.23) --
	(235.22, 57.17) --
	(235.47, 57.11) --
	(235.71, 57.04) --
	(235.95, 56.98) --
	(236.19, 56.91) --
	(236.43, 56.84) --
	(236.67, 56.78) --
	(236.91, 56.71) --
	(237.15, 56.64) --
	(237.39, 56.58) --
	(237.63, 56.51) --
	(237.87, 56.44) --
	(238.11, 56.38) --
	(238.35, 56.31) --
	(238.59, 56.24) --
	(238.83, 56.17) --
	(239.07, 56.10) --
	(239.31, 56.03) --
	(239.55, 55.96) --
	(239.80, 55.89) --
	(240.04, 55.82) --
	(240.28, 55.75) --
	(240.52, 55.68) --
	(240.76, 55.61) --
	(241.00, 55.54) --
	(241.24, 55.47) --
	(241.48, 55.40) --
	(241.72, 55.33) --
	(241.96, 55.26) --
	(242.20, 55.18) --
	(242.44, 55.11) --
	(242.68, 55.04) --
	(242.92, 54.97) --
	(243.16, 54.89) --
	(243.40, 54.82) --
	(243.64, 54.75) --
	(243.88, 54.67) --
	(244.12, 54.60) --
	(244.37, 54.53) --
	(244.61, 54.45) --
	(244.85, 54.38) --
	(245.09, 54.30) --
	(245.33, 54.23) --
	(245.57, 54.16) --
	(245.81, 54.08) --
	(246.05, 54.01) --
	(246.29, 53.93) --
	(246.53, 53.85) --
	(246.77, 53.78) --
	(247.01, 53.70) --
	(247.25, 53.63) --
	(247.49, 53.55) --
	(247.73, 53.48) --
	(247.97, 53.40) --
	(248.21, 53.32) --
	(248.45, 53.25) --
	(248.70, 53.17) --
	(248.94, 53.09) --
	(249.18, 53.02) --
	(249.42, 52.94) --
	(249.66, 52.86) --
	(249.90, 52.79) --
	(250.14, 52.71) --
	(250.38, 52.63) --
	(250.62, 52.55) --
	(250.86, 52.48) --
	(251.10, 52.40) --
	(251.34, 52.32) --
	(251.58, 52.24) --
	(251.82, 52.17) --
	(252.06, 52.09) --
	(252.30, 52.01) --
	(252.54, 51.93) --
	(252.78, 51.85) --
	(253.03, 51.78) --
	(253.27, 51.70) --
	(253.51, 51.62) --
	(253.75, 51.54) --
	(253.99, 51.46) --
	(254.23, 51.38) --
	(254.47, 51.31) --
	(254.71, 51.23) --
	(254.95, 51.15) --
	(255.19, 51.07) --
	(255.43, 50.99) --
	(255.67, 50.91) --
	(255.91, 50.83) --
	(256.15, 50.76) --
	(256.39, 50.68) --
	(256.63, 50.60) --
	(256.87, 50.52) --
	(257.11, 50.44) --
	(257.36, 50.36) --
	(257.60, 50.28) --
	(257.84, 50.21) --
	(258.08, 50.13) --
	(258.32, 50.05) --
	(258.56, 49.97) --
	(258.80, 49.89) --
	(259.04, 49.81) --
	(259.28, 49.74) --
	(259.52, 49.66) --
	(259.76, 49.58) --
	(260.00, 49.50) --
	(260.24, 49.42) --
	(260.48, 49.34) --
	(260.72, 49.27) --
	(260.96, 49.19) --
	(261.20, 49.11) --
	(261.44, 49.03) --
	(261.68, 48.95) --
	(261.93, 48.88) --
	(262.17, 48.80) --
	(262.41, 48.72) --
	(262.65, 48.64) --
	(262.89, 48.57) --
	(263.13, 48.49) --
	(263.37, 48.41) --
	(263.61, 48.34) --
	(263.85, 48.26) --
	(264.09, 48.18) --
	(264.33, 48.10) --
	(264.57, 48.03) --
	(264.81, 47.95) --
	(265.05, 47.87) --
	(265.29, 47.80) --
	(265.53, 47.72) --
	(265.77, 47.65) --
	(266.01, 47.57) --
	(266.26, 47.49) --
	(266.50, 47.42) --
	(266.74, 47.34) --
	(266.98, 47.27) --
	(267.22, 47.19) --
	(267.46, 47.12) --
	(267.70, 47.04) --
	(267.94, 46.97) --
	(268.18, 46.89) --
	(268.42, 46.82) --
	(268.66, 46.75) --
	(268.90, 46.67) --
	(269.14, 46.60) --
	(269.38, 46.52) --
	(269.62, 46.45) --
	(269.86, 46.38) --
	(270.10, 46.30) --
	(270.34, 46.23) --
	(270.59, 46.16) --
	(270.83, 46.09) --
	(271.07, 46.01) --
	(271.31, 45.94) --
	(271.55, 45.87) --
	(271.79, 45.80) --
	(272.03, 45.73) --
	(272.27, 45.66) --
	(272.51, 45.58) --
	(272.75, 45.51) --
	(272.99, 45.44) --
	(273.23, 45.37) --
	(273.47, 45.30) --
	(273.71, 45.23) --
	(273.95, 45.16) --
	(274.19, 45.10) --
	(274.43, 45.03) --
	(274.67, 44.96) --
	(274.92, 44.89) --
	(275.16, 44.82) --
	(275.40, 44.75) --
	(275.64, 44.69) --
	(275.88, 44.62) --
	(276.12, 44.55) --
	(276.36, 44.49) --
	(276.60, 44.42) --
	(276.84, 44.35) --
	(277.08, 44.29) --
	(277.32, 44.22) --
	(277.56, 44.16) --
	(277.80, 44.09) --
	(278.04, 44.03) --
	(278.28, 43.96) --
	(278.52, 43.90) --
	(278.76, 43.84) --
	(279.00, 43.77) --
	(279.25, 43.71) --
	(279.49, 43.65) --
	(279.73, 43.58) --
	(279.97, 43.52) --
	(280.21, 43.46) --
	(280.45, 43.40) --
	(280.69, 43.34) --
	(280.93, 43.28) --
	(281.17, 43.22) --
	(281.41, 43.16) --
	(281.65, 43.10) --
	(281.89, 43.04) --
	(282.13, 42.98) --
	(282.37, 42.92) --
	(282.61, 42.87) --
	(282.85, 42.81) --
	(283.09, 42.75) --
	(283.33, 42.69) --
	(283.57, 42.64) --
	(283.82, 42.58) --
	(284.06, 42.52) --
	(284.30, 42.47) --
	(284.54, 42.41) --
	(284.78, 42.36) --
	(285.02, 42.30) --
	(285.26, 42.25) --
	(285.50, 42.20) --
	(285.74, 42.14) --
	(285.98, 42.09) --
	(286.22, 42.04) --
	(286.46, 41.99) --
	(286.70, 41.93) --
	(286.94, 41.88) --
	(287.18, 41.83) --
	(287.42, 41.78) --
	(287.66, 41.73) --
	(287.90, 41.68) --
	(288.15, 41.63) --
	(288.39, 41.58) --
	(288.63, 41.53) --
	(288.87, 41.48) --
	(289.11, 41.43) --
	(289.35, 41.38) --
	(289.59, 41.33) --
	(289.83, 41.29) --
	(290.07, 41.24) --
	(290.31, 41.19) --
	(290.55, 41.14) --
	(290.79, 41.10) --
	(291.03, 41.05) --
	(291.27, 41.01) --
	(291.51, 40.96) --
	(291.75, 40.92) --
	(291.99, 40.87) --
	(292.23, 40.83) --
	(292.48, 40.78) --
	(292.72, 40.74) --
	(292.96, 40.69) --
	(293.20, 40.65) --
	(293.44, 40.61) --
	(293.68, 40.56) --
	(293.92, 40.52) --
	(294.16, 40.48) --
	(294.40, 40.44) --
	(294.64, 40.40) --
	(294.88, 40.35) --
	(295.12, 40.31) --
	(295.36, 40.27) --
	(295.60, 40.23) --
	(295.84, 40.19) --
	(296.08, 40.15) --
	(296.32, 40.11) --
	(296.56, 40.07) --
	(296.81, 40.04) --
	(297.05, 40.00) --
	(297.29, 39.96) --
	(297.53, 39.92) --
	(297.77, 39.88) --
	(298.01, 39.85) --
	(298.25, 39.81) --
	(298.49, 39.77) --
	(298.73, 39.73) --
	(298.97, 39.70) --
	(299.21, 39.66) --
	(299.45, 39.63) --
	(299.69, 39.59) --
	(299.93, 39.56) --
	(300.17, 39.52) --
	(300.41, 39.49) --
	(300.65, 39.45) --
	(300.89, 39.42) --
	(301.13, 39.38) --
	(301.38, 39.35) --
	(301.62, 39.32) --
	(301.86, 39.28) --
	(302.10, 39.25) --
	(302.34, 39.22) --
	(302.58, 39.19) --
	(302.82, 39.15) --
	(303.06, 39.12) --
	(303.30, 39.09) --
	(303.54, 39.06) --
	(303.78, 39.03) --
	(304.02, 39.00) --
	(304.26, 38.97) --
	(304.50, 38.94) --
	(304.74, 38.91) --
	(304.98, 38.88) --
	(305.22, 38.85) --
	(305.46, 38.82) --
	(305.71, 38.79) --
	(305.95, 38.76) --
	(306.19, 38.73) --
	(306.43, 38.71) --
	(306.67, 38.68) --
	(306.91, 38.65) --
	(307.15, 38.62) --
	(307.39, 38.60) --
	(307.63, 38.57) --
	(307.87, 38.54) --
	(308.11, 38.52) --
	(308.35, 38.49) --
	(308.59, 38.46) --
	(308.83, 38.44) --
	(309.07, 38.41) --
	(309.31, 38.39) --
	(309.55, 38.36) --
	(309.79, 38.34) --
	(310.04, 38.31) --
	(310.28, 38.29) --
	(310.52, 38.26) --
	(310.76, 38.24) --
	(311.00, 38.22) --
	(311.24, 38.19) --
	(311.48, 38.17) --
	(311.72, 38.15) --
	(311.96, 38.12) --
	(312.20, 38.10) --
	(312.44, 38.08) --
	(312.68, 38.06) --
	(312.92, 38.04) --
	(313.16, 38.01) --
	(313.40, 37.99) --
	(313.64, 37.97) --
	(313.88, 37.95) --
	(314.12, 37.93) --
	(314.37, 37.91) --
	(314.61, 37.89) --
	(314.85, 37.87) --
	(315.09, 37.85) --
	(315.33, 37.83) --
	(315.57, 37.81) --
	(315.81, 37.79) --
	(316.05, 37.77) --
	(316.29, 37.75) --
	(316.53, 37.73) --
	(316.77, 37.71) --
	(317.01, 37.70) --
	(317.25, 37.68) --
	(317.49, 37.66) --
	(317.73, 37.64) --
	(317.97, 37.63) --
	(318.21, 37.61) --
	(318.45, 37.59) --
	(318.70, 37.57) --
	(318.94, 37.56) --
	(319.18, 37.54) --
	(319.42, 37.52) --
	(319.66, 37.51) --
	(319.90, 37.49) --
	(320.14, 37.48) --
	(320.38, 37.46) --
	(320.62, 37.44) --
	(320.86, 37.43) --
	(321.10, 37.41) --
	(321.34, 37.40) --
	(321.58, 37.39) --
	(321.82, 37.37) --
	(322.06, 37.36) --
	(322.30, 37.34) --
	(322.54, 37.33) --
	(322.78, 37.31) --
	(323.02, 37.30) --
	(323.27, 37.29) --
	(323.51, 37.27) --
	(323.75, 37.26) --
	(323.99, 37.25) --
	(324.23, 37.23) --
	(324.47, 37.22) --
	(324.71, 37.21) --
	(324.95, 37.20) --
	(325.19, 37.19) --
	(325.43, 37.17) --
	(325.67, 37.16) --
	(325.91, 37.15) --
	(326.15, 37.14) --
	(326.39, 37.13) --
	(326.63, 37.12) --
	(326.87, 37.10) --
	(327.11, 37.09) --
	(327.35, 37.08) --
	(327.60, 37.07) --
	(327.84, 37.06) --
	(328.08, 37.05) --
	(328.32, 37.04) --
	(328.56, 37.03) --
	(328.80, 37.02) --
	(329.04, 37.01) --
	(329.28, 37.00) --
	(329.52, 36.99) --
	(329.76, 36.98) --
	(330.00, 36.97) --
	(330.24, 36.97) --
	(330.48, 36.96) --
	(330.72, 36.95) --
	(330.96, 36.94) --
	(331.20, 36.93) --
	(331.44, 36.92) --
	(331.68, 36.91) --
	(331.93, 36.91) --
	(332.17, 36.90) --
	(332.41, 36.89) --
	(332.65, 36.88) --
	(332.89, 36.88) --
	(333.13, 36.87) --
	(333.37, 36.86) --
	(333.61, 36.85) --
	(333.85, 36.85) --
	(334.09, 36.84) --
	(334.33, 36.83) --
	(334.57, 36.83) --
	(334.81, 36.82) --
	(335.05, 36.81) --
	(335.29, 36.81) --
	(335.53, 36.80) --
	(335.77, 36.80) --
	(336.01, 36.79) --
	(336.26, 36.78) --
	(336.50, 36.78) --
	(336.74, 36.77) --
	(336.98, 36.77) --
	(337.22, 36.76) --
	(337.46, 36.76) --
	(337.70, 36.75) --
	(337.94, 36.75) --
	(338.18, 36.74) --
	(338.42, 36.74) --
	(338.66, 36.73) --
	(338.90, 36.73) --
	(339.14, 36.72) --
	(339.38, 36.72) --
	(339.62, 36.71) --
	(339.86, 36.71) --
	(340.10, 36.70) --
	(340.34, 36.70) --
	(340.59, 36.70) --
	(340.83, 36.69) --
	(341.07, 36.69) --
	(341.31, 36.69) --
	(341.55, 36.68) --
	(341.79, 36.68) --
	(342.03, 36.67) --
	(342.27, 36.67) --
	(342.51, 36.67) --
	(342.75, 36.67) --
	(342.99, 36.66) --
	(343.23, 36.66) --
	(343.47, 36.66) --
	(343.71, 36.65) --
	(343.95, 36.65) --
	(344.19, 36.65) --
	(344.43, 36.65) --
	(344.67, 36.64) --
	(344.91, 36.64) --
	(345.16, 36.64) --
	(345.40, 36.64) --
	(345.64, 36.63) --
	(345.88, 36.63) --
	(346.12, 36.63) --
	(346.36, 36.63) --
	(346.60, 36.62) --
	(346.84, 36.62) --
	(347.08, 36.62) --
	(347.32, 36.62) --
	(347.56, 36.62) --
	(347.80, 36.62) --
	(348.04, 36.61) --
	(348.28, 36.61) --
	(348.52, 36.61) --
	(348.76, 36.61) --
	(349.00, 36.61) --
	(349.24, 36.61) --
	(349.49, 36.61) --
	(349.73, 36.60) --
	(349.97, 36.60) --
	(350.21, 36.60) --
	(350.45, 36.60) --
	(350.69, 36.60) --
	(350.93, 36.60) --
	(351.17, 36.60) --
	(351.41, 36.60) --
	(351.65, 36.60) --
	(351.89, 36.60) --
	(352.13, 36.59) --
	(352.37, 36.59) --
	(352.61, 36.59) --
	(352.85, 36.59) --
	(353.09, 36.59) --
	(353.33, 36.59) --
	(353.57, 36.59) --
	(353.82, 36.59) --
	(354.06, 36.59) --
	(354.30, 36.59) --
	(354.54, 36.59) --
	(354.78, 36.59) --
	(355.02, 36.59) --
	(355.26, 36.59) --
	(355.50, 36.59) --
	(355.74, 36.59) --
	(355.98, 36.59) --
	(356.22, 36.59) --
	(356.46, 36.59) --
	(356.70, 36.59) --
	(356.94, 36.59) --
	(357.18, 36.59) --
	(357.42, 36.59) --
	(357.66, 36.59) --
	(357.90, 36.59) --
	(358.15, 36.59) --
	(358.39, 36.59) --
	(358.63, 36.59) --
	(358.87, 36.59) --
	(359.11, 36.59) --
	(359.35, 36.59) --
	(359.59, 36.59) --
	(359.83, 36.59) --
	(360.07, 36.59) --
	(360.31, 36.59) --
	(360.55, 36.59) --
	(360.79, 36.59) --
	(361.03, 36.59);
\definecolor{drawColor}{RGB}{0,0,0}

\path[draw=drawColor,draw opacity=0.30,line width= 0.4pt,line join=round,line cap=round] ( 24.60, 36.59) -- (457.16, 36.59);
\definecolor{drawColor}{RGB}{0,0,0}

\path[draw=drawColor,draw opacity=0.10,line width= 0.4pt,line join=round,line cap=round] (120.72, 30.60) -- (120.72,192.21);

\path[draw=drawColor,draw opacity=0.10,line width= 0.4pt,line join=round,line cap=round] (200.83, 30.60) -- (200.83,192.21);

\path[draw=drawColor,draw opacity=0.10,line width= 0.4pt,line join=round,line cap=round] (280.93, 30.60) -- (280.93,192.21);

\path[draw=drawColor,draw opacity=0.10,line width= 0.4pt,line join=round,line cap=round] (361.03, 30.60) -- (361.03,192.21);

\path[draw=drawColor,draw opacity=0.10,line width= 0.4pt,line join=round,line cap=round] (441.13, 30.60) -- (441.13,192.21);
\definecolor{drawColor}{RGB}{0,0,0}

\path[draw=drawColor,line width= 0.4pt,line join=round,line cap=round] (120.72, 30.60) -- (120.72,192.21);

\path[draw=drawColor,line width= 0.8pt,line join=round,line cap=round] (361.03, 36.59) --
	(362.63, 36.59) --
	(364.24, 36.59) --
	(365.84, 36.59) --
	(367.44, 36.59) --
	(369.04, 36.59) --
	(370.64, 36.59) --
	(372.25, 36.59) --
	(373.85, 36.59) --
	(375.45, 36.59) --
	(377.05, 36.59) --
	(378.65, 36.59) --
	(380.26, 36.59) --
	(381.86, 36.59) --
	(383.46, 36.59) --
	(385.06, 36.59) --
	(386.66, 36.59) --
	(388.27, 36.59) --
	(389.87, 36.59) --
	(391.47, 36.59) --
	(393.07, 36.59) --
	(394.67, 36.59) --
	(396.28, 36.59) --
	(397.88, 36.59) --
	(399.48, 36.59) --
	(401.08, 36.59) --
	(402.69, 36.59) --
	(404.29, 36.59) --
	(405.89, 36.59) --
	(407.49, 36.59) --
	(409.09, 36.59) --
	(410.70, 36.59) --
	(412.30, 36.59) --
	(413.90, 36.59) --
	(415.50, 36.59) --
	(417.10, 36.59) --
	(418.71, 36.59) --
	(420.31, 36.59) --
	(421.91, 36.59) --
	(423.51, 36.59) --
	(425.11, 36.59) --
	(426.72, 36.59) --
	(428.32, 36.59) --
	(429.92, 36.59) --
	(431.52, 36.59) --
	(433.12, 36.59) --
	(434.73, 36.59) --
	(436.33, 36.59) --
	(437.93, 36.59) --
	(439.53, 36.59) --
	(441.13, 36.59) --
	(442.74, 36.59) --
	(444.34, 36.59) --
	(445.94, 36.59) --
	(447.54, 36.59) --
	(449.14, 36.59) --
	(450.75, 36.59) --
	(452.35, 36.59) --
	(453.95, 36.59) --
	(455.55, 36.59) --
	(457.16, 36.59) --
	(458.76, 36.59) --
	(460.36, 36.59) --
	(461.96, 36.59) --
	(463.56, 36.59) --
	(465.17, 36.59) --
	(466.77, 36.59) --
	(468.37, 36.59) --
	(469.97, 36.59) --
	(471.57, 36.59) --
	(473.18, 36.59) --
	(474.78, 36.59) --
	(476.38, 36.59) --
	(477.98, 36.59) --
	(479.58, 36.59) --
	(481.19, 36.59) --
	(482.79, 36.59) --
	(484.39, 36.59) --
	(485.99, 36.59) --
	(487.59, 36.59) --
	(489.20, 36.59) --
	(490.80, 36.59) --
	(492.40, 36.59) --
	(494.00, 36.59) --
	(495.60, 36.59) --
	(497.21, 36.59) --
	(498.81, 36.59) --
	(500.41, 36.59) --
	(502.01, 36.59) --
	(503.61, 36.59) --
	(505.22, 36.59) --
	(506.82, 36.59) --
	(508.42, 36.59) --
	(510.02, 36.59) --
	(511.62, 36.59) --
	(513.23, 36.59) --
	(514.83, 36.59) --
	(516.43, 36.59) --
	(518.03, 36.59) --
	(519.64, 36.59) --
	(521.24, 36.59);

\node[text=drawColor,anchor=base,inner sep=0pt, outer sep=0pt, scale=  1.00] at ( 56.64, 69.74) {\(v_0(x)\)};

\path[draw=drawColor,line width= 0.4pt,line join=round,line cap=round] ( 56.64, 81.48) -- (104.70, 89.88);

\path[draw=drawColor,line width= 0.4pt,line join=round,line cap=round] ( 99.16, 85.24) --
	(104.70, 89.88) --
	( 97.92, 92.36);

\node[text=drawColor,anchor=base,inner sep=0pt, outer sep=0pt, scale=  1.00] at (304.96, 87.48) {\(\bar v(x)\)};

\path[draw=drawColor,line width= 0.4pt,line join=round,line cap=round] (304.96, 81.48) -- (240.88, 55.58);

\path[draw=drawColor,line width= 0.4pt,line join=round,line cap=round] (245.33, 61.27) --
	(240.88, 55.58) --
	(248.03, 54.57);
\end{scope}
\end{tikzpicture}